\documentclass[11pt,reqno]{amsart}
\usepackage[cm]{fullpage}
\usepackage{mathrsfs,amssymb,graphicx,verbatim,amsmath,amsfonts}
\usepackage{paralist}
\usepackage{esint}
\usepackage[autostyle]{csquotes}
\usepackage[breaklinks,pdfstartview=FitH]{hyperref}
\usepackage{upgreek}
\usepackage[mathscr]{euscript}
\usepackage{dutchcal}
\usepackage[T1]{fontenc}

\usepackage{graphicx}
\usepackage{sidecap}
\usepackage[all]{xy}
\usepackage[font=small,labelfont=bf,labelsep=endash]{caption}
\usepackage{enumitem}
\newcommand{\sss}{\mathsf{s}}
\newcommand{\tp}{\widehat{\otimes}}

\newcommand{\ray}{\mathscr{R}}
\renewcommand{\eta}{\upeta}

\renewcommand{\epsilon}{\upepsilon}
\newcommand{\MM}{\mathcal{M}}

\renewcommand{\chi}{\upchi}
\renewcommand{\theta}{\uptheta}

\renewcommand{\le}{\leqslant}
\renewcommand{\ge}{\geqslant}

\renewcommand{\setminus}{\smallsetminus}
\renewcommand{\gamma}{\upgamma}
\renewcommand{\lambda}{\uplambda}
\renewcommand{\alpha}{\upalpha}
\renewcommand{\delta}{\updelta}
\renewcommand{\beta}{\upbeta}
\renewcommand{\omega}{\upomega}
\renewcommand{\nu}{\upnu}
\renewcommand{\mu}{\upmu}
\renewcommand{\psi}{\uppsi}
\renewcommand{\phi}{\upphi}
\renewcommand{\rho}{\uprho}
\renewcommand{\kappa}{\upkappa}

\newcommand{\OO}{\mathsf{O}}

\renewcommand{\AA}{\mathsf{A}}

\renewcommand{\tau}{\uptau}
\newcommand{\Id}{\mathsf{Id}}
\renewcommand{\tt}{\mathsf{t}}
\newcommand{\ud}[0]{\,\mathrm{d}}

\newcommand{\spn}{\mathrm{\bf span}}
\makeatletter
\def\moverlay{\mathpalette\mov@rlay}
\def\mov@rlay#1#2{\leavevmode\vtop{%
   \baselineskip\z@skip \lineskiplimit-\maxdimen
   \ialign{\hfil$\m@th#1##$\hfil\cr#2\crcr}}}
\newcommand{\charfusion}[3][\mathord]{
    #1{\ifx#1\mathop\vphantom{#2}\fi
        \mathpalette\mov@rlay{#2\cr#3}
      }
    \ifx#1\mathop\expandafter\displaylimits\fi}
\makeatother

\newcommand{\M}{\mathsf{M}}

\newcommand{\m}{\{1,\ldots,m\}}

\newcommand{\vol}{\mathrm{\bf vol}}
\newcommand{\proj}{\mathsf{Proj}}

\newcommand{\NN}{\mathcal{N}}

\newcommand{\kk}{\mathsf{k}}

\newcommand{\dd}{\mathsf{d}}
\renewcommand{\Pr}{\mathbb{P}}

\newcommand{\sign}{\mathrm{\bf sign}}
\newcommand{\rank}{\mathrm{\bf rank}}

\newcommand{\n}{\{1,\ldots,n\}}

\renewcommand{\k}{\{1,\ldots,k\}}

\newcommand{\sub}{\mathscr{C}}

\renewcommand{\d}{\delta}

\newcommand{\e}{\varepsilon}
\newcommand{\R}{\mathbb R}
\renewcommand{\1}{\mathbf 1}

\newcommand{\cc}{\mathsf{c}}
\newcommand{\G}{\mathsf{G}}
\newtheorem{theorem}{Theorem}
\newtheorem{lemma}[theorem]{Lemma}
\newtheorem{proposition}[theorem]{Proposition}
\newtheorem{claim}[theorem]{Claim}

\newtheorem{corollary}[theorem]{Corollary}

\newtheorem{definition}[theorem]{Definition}

\theoremstyle{remark}
\newtheorem{remark}[theorem]{Remark}

\newtheorem{question}[theorem]{Question}
\newtheorem{problem}[theorem]{Problem}

\renewcommand{\pi}{\uppi}
\renewcommand{\zeta}{\upzeta}
\newcommand{\mA}{\mathsf{A}}
\newcommand{\mB}{\mathsf{B}}

\newcommand{\YY}{\mathcal{Y}}

\renewcommand{\S}{\mathsf{S}}
\renewcommand{\subset}{\subseteq}

\newcommand{\C}{\mathbb C}

\newcommand{\E}{\mathbb{E}}

\newcommand{\W}{\mathsf{W}}

\renewcommand{\sigma}{\upsigma}

\newcommand{\N}{\mathbb N}
\newcommand{\Z}{\mathbb Z}

\newcommand{\eqdef}{\stackrel{\mathrm{def}}{=}}

\DeclareMathOperator{\diam}{\mathbf{diam}}

\newcommand{\EE}{\mathsf{E}}
\newcommand{\JJ}{\mathsf{J}}
\renewcommand{\emptyset}{\varnothing}

\begin{document}

%\title{Metric dimension reduction}

\title{Metric dimension reduction: A snapshot of the Ribe~program}

%\begin{titlepage}

\author{Assaf Naor}
%\address{Mathematics Department\\ Princeton University}
\address{Mathematics Department, Princeton University, Princeton, New Jersey 08544-1000, USA.}
\email{naor@math.princeton.edu}
%\thanks{Supported in part by the BSF, the Packard Foundation and the Simons Foundation.}

\thanks{Supported in part by NSF grant CCF-1412958, the Packard Foundation and the Simons Foundation. This work was  was carried out under the auspices of the Simons Algorithms and Geometry (A\&G) Think Tank, and was completed while the author was a member of the Institute for Advanced Study}

%\thanks{Supported by the NSF, the Packard Foundation and the Simons Foundation.}

%\date{\today}
%\keywords{Dimension reduction, Metric embeddings, Nuclear norm, Schatten--von Neumann classes, Lipschitz quotients, Markov %convexity.}
%\subjclass[2010]{30L05, 46B85, 46B20, 46B80}

%\date{\today}
%\keywords{Lipschitz extension}

\maketitle

%\vspace{-0.1in}

%\begin{abstract}

%\end{abstract}

%\thispagestyle{empty}
%\end{titlepage}

%\tableofcontents

\section{Introduction}

The purpose of this article is to survey some of the context, achievements, challenges and mysteries of the field of {\em metric dimension reduction}, including new perspectives on major older results as well as recent advances.

%We will also briefly indicate a development which is ongoing at the time of writing, demonstrating that the omnipresent "curse of dimensionality" %is to some extent absent from one of  the most basic general tasks of algorithmic information retrieval.
%Aspects of the Ribe program that provide motivation and context will be explained.

From the point of view of theoretical computer science,  mathematicians "stumbled upon" metric dimension reduction in the early 1980s, as exemplified   by the following quote~\cite{Vem04}.

\blockquote{{ \small \em Two decades ago, analysts stumbled upon a surprising fact [...],
the Johnson--Lindenstrauss Lemma, as a crucial tool in their project of extending
functions in continuous ways. This result [...] says that, if you project $n$ points in some high dimensional space down to a random $O(\log n)$-dimensional plane, the chances are
overwhelming that all distances will be preserved  within a small relative error.
So, if distance is all you care about, there is no reason to stay in high dimensions!

\begin{flushright}{\small \em C.~Papadimitriou, 2004 (forward to {\em The random projection method} by S.~Vempala).}\end{flushright}}}
The above  use of the term  "stumbled upon" is  justified, because it would be fair to say that at the inception of this research direction mathematicians did not anticipate  the remarkable swath of its later impact on algorithms. However, rather than being discovered accidentally, the investigations that will be surveyed here can be motivated by classical issues in metric geometry. From the internal perspective of pure mathematics, it would be more befitting to state that the aforementioned early work stumbled upon the unexpected depth,  difficulty and richness of basic questions on the relation between "rough quantitative geometry" and dimension. Despite major efforts by many mathematicians over the past four decades, such questions  remain by and large stubbornly open.

We will explain below key ideas of major developments in metric dimension reduction, and also describe the larger mathematical landscape that partially motivates these investigations, most notably the {\em bi-Lipschitz embedding problem into $\R^n$} and the {\em Ribe program}. By choosing to focus on aspects of this area within  pure mathematics, we will put aside  the large (and growing) literature that investigates algorithmic ramifications of metric dimension reduction. Such applications warrant a separate treatment that is far beyond the scope of the present exposition; some aspects of that material are covered in the monographs~\cite{Mat02,Vem04,Har11} and the surveys~\cite{Ind01,Lin02}, as well as the articles of Andoni--Indyk--Razenshteyn and Arora in the present volume.

%\footnote{Including, as an indication of some examples, approximation algorithms~\cite{LLR}, nearest neighbor search~\cite{IM99,KOR00}, low-rank %approximation of matrices~\cite{PRTV00}, clustering~\cite{Sch00}, and many more.}

\begin{remark} The broader term {\em dimension reduction} is used ubiquitously in statistics and machine learning, with striking applications whose full rigorous understanding sometimes awaits the scrutiny of mathematicians (see e.g.~\cite{Bur10}). A common (purposefully vague) description of this term is the desire to decrease the degrees of freedom of a high-dimensional data set while approximately preserving some of its  pertinent features; formulated in such great generality, the area includes topics such as neural networks (see e.g.~\cite{HS06}).  The commonly used term {\em curse of dimensionality} refers to the perceived impossibility of this goal in many settings, and that the performance (running time, storage space) of certain algorithmic tasks must deteriorate exponentially as the underlying dimension grows. But, sometimes it does seem that certain high-dimensional data sets can be realized faithfully using  a small number of latent variables as auxiliary "coordinates." Here we restrict ourselves exclusively to  {\em metric} dimension reduction, i.e., to notions of faithfulness of low-dimensional representations that require the (perhaps quite rough) preservation of pairwise distances, including ways to prove the impossibility thereof.
\end{remark}

\subsection*{Roadmap} The rest of the Introduction is an extensive and detailed account of the area of metric dimension reduction, including statements of most of the main known results,  background and context, and many important open questions. The Introduction is thus an expository account  of the field, so those readers who do not wish to delve into some proofs, could read it separately from the rest of the text. The remaining sections contain further details  and complete justifications of  those statements that have not appeared in the literature.

\subsection{Bi-Lipschitz embeddings}\label{sec:bilip} Fix $\alpha\ge 1$. A metric space $(\MM,d_\MM)$ is said to embed with distortion $\alpha$ into a metric space $(\NN,d_\NN)$ if there is a mapping (an embedding) $f:\MM\to \NN$ and (a scaling factor) $\tau>0$ such that
\begin{equation}\label{eq:def embedding}
\forall\, x,y\in \MM,\qquad \tau d_{\MM}(x,y)\le d_\NN\big(f(x),f(y)\big)\le \alpha \tau d_\MM(x,y).
\end{equation}
The infimum over $\alpha\in [1,\infty]$ for which $(\MM,d_\MM)$ embeds with distortion $\alpha$ into $(\NN,d_\NN)$ is denoted $\cc_{(\NN,d_\NN)}(\MM,d_\MM)$, or $\cc_\NN(\MM)$ if the underlying metrics are clear from the context. If $\cc_\NN(\MM)<\infty$, then $(\MM,d_\MM)$ is said to admit a bi-Lipschitz embedding into $(\NN,d_\NN)$. Given $p\in [1,\infty)$, if $\NN$ is an $L_p(\mu)$ space into which $\MM$ admits a bi-Lipschitz embedding, then we use the notation $\cc_{L_p(\mu)}(\MM)=\cc_p(\MM)$. The numerical invariant $\cc_2(\MM)$, which measures the extent to which $\MM$ is close to being a (subset of a) Euclidean geometry, is called the {\em Euclidean distortion} of $\MM$.

A century of intensive research into bi-Lipschitz embeddings led to a rich theory with many deep achievements, but the following problem, which is one of the first questions that arise naturally in this direction, remains a major longstanding mystery; see e.g.~\cite{Sem99,LP01,Hei03,NN12}. Analogous issues in the context of topological dimension,  differentiable manifolds and Riemannian manifolds  were famously settled by Menger~\cite{Men28} and N\"obeling~\cite{Nob31}, Whitney~\cite{Whi36} and Nash~\cite{Nas54}, respectively.

\begin{problem}[the bi-Lipschitz embedding problem into $\R^k$]\label{Q:bilip Rk} Obtain an intrinsic characterization of those metric spaces $(\MM,d_\MM)$ that admit a bi-Lipschitz embedding into $\R^k$ for some $k\in \N$.
\end{problem}
Problem~\ref{Q:bilip Rk} is one of the qualitative underpinnings of the issues that will be surveyed here.  We say that it is "qualitative" because it ignores the magnitude of the distortion altogether, and therefore one does not need to specify which norm on $\R^k$ is considered in Problem~\ref{Q:bilip Rk}, since all the norms on $\R^k$ are (bi-Lipschitz) equivalent. Problem~\ref{Q:bilip Rk}  is also (purposefully) somewhat vague, because the notion of "intrinsic characterization" is not  well-defined. We will return to this matter in Section~\ref{sec:doubling} below, where candidates for such a characterization are discussed. At this juncture, it suffices to illustrate what Problem~\ref{Q:bilip Rk} aims to achieve through the following useful example. If one does not impose any restriction on the target dimension and allows for a bi-Lipschitz embedding into an infinite dimensional Hilbert space, then the following intrinsic characterization is available. A metric space $(\MM,d_\MM)$ admits a bi-Lipschitz embedding into a Hilbert space if and only if there exists $C=C_\MM\in [0,1)$ such that for every $n\in \N$ and every positive semidefinite symmetric matrix $\mA=(a_{ij})\in \M_n(\R)$ all of whose rows sum to zero (i.e., $\sum_{j=1}^na_{ij}=0$ for every $i\in \n$), the following quadratic distance inequality holds true.
\begin{equation}\label{eq:K LLR}
\forall\, x_1,\ldots,x_n\in \MM,\qquad \sum_{i=1}^n\sum_{j=1}^na_{ij}d_\MM(x_i,x_j)^2\le C\sum_{i=1}^n\sum_{j=1}^n|a_{ij}|d_\MM(x_i,x_j)^2.
\end{equation}
In fact, one can refine this statement quantitatively as follows. A metric space embeds with distortion $\alpha\in [1,\infty)$ into a Hilbert space if and only if  in the setting of~\eqref{eq:K LLR} we have
\begin{equation}\label{eq:alpha LLR}
\forall\, x_1,\ldots,x_n\in \MM,\qquad \sum_{i=1}^n\sum_{j=1}^na_{ij}d_\MM(x_i,x_j)^2\le \frac{\alpha^2-1}{\alpha^2+1}\sum_{i=1}^n\sum_{j=1}^n|a_{ij}|d_\MM(x_i,x_j)^2.
\end{equation}
The case $\alpha=1$ of~\eqref{eq:alpha LLR}, i.e., the case of isometric embeddings, is  a famous classical theorem of Schoenberg~\cite{Sch38}, and the general case of~\eqref{eq:alpha LLR} is due to Linial, London and Rabinovich~\cite[Corollary~3.5]{LLR}. The above characterization is clearly intrinsic, as it is a family of finitary distance inequalities among points of $\MM$ that do not make reference to any other auxiliary/external object. With such a characterization at hand, one could  examine the internal structure of a given metric space so as to determine if it can be represented faithfully as a subset of a Hilbert space. Indeed, \cite{LLR} uses~\eqref{eq:alpha LLR} to obtain an algorithm that takes as input an $n$-point metric space $(\MM,d_\MM)$ and outputs in polynomial time an arbitrarily good approximation to its Euclidean distortion $\cc_2(\MM)$.

A meaningful answer to Problem~\ref{Q:bilip Rk}  could in principle lead to a method for determining if a member of a  family $\mathcal{F}$ of metric spaces admits an embedding with specified distortion into a member of a family $\mathcal{G}$ of low dimensional normed spaces.
%, and might even provide a procedure to find such an embedding.
Formulated in such great generality, this type of question encompasses all of the investigations into metric dimension reduction that will be discussed in what follows, except that we will also examine analogous issues for embeddings with guarantees that are substantially  weaker (though natural and useful) than the "vanilla"  bi-Lipschitz requirement~\eqref{eq:def embedding}.

\begin{remark}Analogues of the above questions are very natural also when the target low-dimensional geometries are not necessarily normed spaces. Formulating meaningful goals in  such a setting is not as straightforward as it is for normed spaces, e.g.~requiring that the target space is a manifold of low topological dimension is not very useful, so one must impose geometric restrictions on the target manifold. As another example (to which we will briefly return later), one could ask about embeddings into spaces of probability measures that are equipped with a Wasserstein (transportation cost) metric, with control on the size of the underlying metric space. At present, issues of this type are largely an uncharted terrain whose exploration is likely to be interesting and useful.
\end{remark}

\subsection{Local theory and the Ribe program} Besides being motivated by the bi-Lipschitz embedding problem into $\R^k$, much of the inspiration for the studies that will be presented below comes from a major endeavour in metric geometry called the {\em Ribe program}. This is a large and active research area that has been (partially) surveyed in~\cite{Kal08,Nao12,Bal13,Ost13,God17}. It would be highly unrealistic to attempt to cover it comprehensively here, but we will next present a self-contained general introduction to the Ribe program  that is intended for non-experts, including   aspects that are relevant to the ensuing discussion on metric dimension reduction.

Martin Ribe was a mathematician who in the 1970s obtained a few beautiful results in functional analysis, prior to leaving mathematics. Among his achievements is a very influential rigidity theorem~\cite{Ribe76} which shows that the local linear theory of Banach spaces could in principle be described using only distances between points, and hence it could potentially apply to general metric spaces.

Before formulating the above statement precisely, it is instructive to consider a key example. The {\em infimal cotype}~\cite{Mau03} $q_X$ of a Banach space $(X,\|\cdot\|)$ is the infimum over those $q\in [2,\infty]$ for which\footnote{In addition to the standard $``O"$ notation, we will use throughout this article the following standard and convenient asymptotic notation. Given two quantities $Q,Q'>0$, the notations
$Q\lesssim Q'$ and $Q'\gtrsim Q$ mean that $Q\le CQ'$ for some
universal constant $C>0$. The notation $Q\asymp Q'$
stands for $(Q\lesssim Q') \wedge  (Q'\lesssim Q)$. If  we need to allow for dependence on parameters, we indicate this by subscripts. For example, in the presence of  auxiliary objects (e.g.~numbers or spaces) $\phi,\mathfrak{Z}$, the notation $Q\lesssim_{\phi,\mathfrak{Z}} Q'$ means that $Q\le C(\phi,\mathfrak{Z})Q' $, where $C(\phi,\mathfrak{Z}) >0$ is allowed to depend only on $\phi,\mathfrak{Z}$; similarly for the notations $Q\gtrsim_{\phi,\mathfrak{Z}} Q'$ and $Q\asymp_{\phi,\mathfrak{Z}} Q'$.}
\begin{equation}\label{def:cotype}
\forall\, n\in \N,\ \forall\, x_1,\ldots,x_n\in X,\qquad \sum_{i=1}^n \|x_i\|^2\lesssim_{X,q}
\frac{n^{1-\frac2{q}}}{2^n}\sum_{\e\in \{-1,1\}^n}\bigg\|\sum_{i=1}^n \e_ix_i\bigg\|^2.
\end{equation}
In the special case $x_1=\ldots=x_n=x\in X\setminus\{0\}$, the left hand side of~\eqref{def:cotype} is equal to $n\|x\|^2$ and by expanding the squares one computes that the right hand side of~\eqref{def:cotype} is equal to $n^{2(1-1/q)}\|x\|^2$. Hence~\eqref{def:cotype} necessitates that $q\ge 2$, which explains why we imposed this restriction on $q$ at the outset.  Note also that~\eqref{def:cotype} holds true in any Banach space when $q=\infty$. This is a quick consequence of the convexity of the mapping $x\mapsto \|x\|^2$, since for every $\e\in \{-1,1\}^n$ and $i \in \n$ we have
\begin{multline}\label{convex priove cotype}
\|x_i\|^2=\bigg\|\frac{(\e_1x_1+\ldots+\e_nx_n)+(-\e_1x_1-\ldots-\e_{i-1}x_{i-1}+\e_i x_i-\e_{i+1}x_{i+1}\ldots-\e_n x_n)}{2}\bigg\|^2\\
\le \frac{\|\e_1x_1+\ldots+\e_nx_n\|^2+\|-\e_1x_1-\ldots-\e_{i-1}x_{i-1}+\e_i x_i-\e_{i+1}x_{i+1}\ldots-\e_n x_n\|^2}{2}.
\end{multline}
By averaging~\eqref{convex priove cotype} over $\e\in \{-1,1\}^n$ and $i \in \n$ we see that~\eqref{def:cotype} holds if $q=\infty$.  So, one could view~\eqref{def:cotype} for $q<\infty$ as a requirement that the norm $\|\cdot\|:X\to [0,\infty)$ has a property that is asymptotically stronger than mere convexity. When $X=\ell_\infty$, this requirement  does not hold for any $q<\infty$, since if  $\{x_i\}_{i=1}^n$ are the first $n$ elements of the coordinate basis, then the left hand side of~\eqref{def:cotype} equals $n$ while its right hand side   equals $n^{1-1/q}$.

 Maurey and Pisier proved~\cite{MP73} that the above obstruction to having $q_X<\infty$ is actually the {\em only possible} such obstruction. Thus, by ruling out the presence of copies of $\{\ell_\infty^n\}_{n=1}^\infty$ in $X$ one immediately deduces the  "upgraded" (asymptotically stronger as $n\to \infty$) randomized convexity inequality~\eqref{def:cotype} for some $q<\infty$.

\begin{theorem}\label{thm:MP cotype} The following conditions are equivalent for every Banach space $(X,\|\cdot\|)$.
\begin{itemize}
\item There is no $\alpha\in [1,\infty)$ such that  $\ell_\infty^n$ is $\alpha$-isomorphic to a subspace of $X$ for every $n\in \N$.
\item $q_X<\infty$.
\end{itemize}
\end{theorem}
The (standard) terminology that is used in Theorem~\ref{thm:MP cotype} is that given $\alpha\in [1,\infty)$, a Banach space $(Y,\|\cdot\|_Y)$ is said to be $\alpha$-isomorphic to a subspace of a Banach space $(Z,\|\cdot\|_Z)$ if there is a linear operator $T:Y\to Z$ satisfying $\|y\|_Y\le \|Ty\|_Z\le \alpha\|y\|_Y$ for every $y\in Y$; this is the same as saying that $Y$ embeds into $Z$ with distortion $\alpha$ via an embedding that is a linear operator.
%In analogy to the quantity $\cc_Z(Y)$, the Banach--Mazur distance of $Y$ to a subspace of $Z$ is defined to be the infimum over those $\alpha\in %[1,\infty]$ for which $Y$ is $\alpha$-isomorphic to a subspace of $Z$.

Suppose that $X$ and $Y$ are Banach spaces that are uniformly homeomorphic, i.e., there is a bijection $f:X\to Y$ such that both $f$ and $f^{-1}$ are uniformly continuous. By the aforementioned rigidity theorem of Ribe (which will be formulated below in full generality),  this implies in particular that $q_X=q_Y$.  So, despite the fact that  the requirement~\eqref{def:cotype} involves linear operations (summation and sign changes) that do not make sense in general metric spaces, it is in fact preserved by purely metric (quantitatively continuous, though potentially very complicated) deformations. Therefore, in principle~\eqref{def:cotype} could be characterized while only making reference  to  distances between  points in $X$. More generally, Ribe's rigidity theorem makes an analogous assertion for {\em any} isomorphic local linear property of a Banach space; we will define formally those properties in a  moment, but,  informally, they are requirements in the spirit of~\eqref{def:cotype} that depend only on the finite dimensional subspaces of the given Banach space  and are stable under linear isomorphisms that could potentially incur a large error.

The purely metric reformulation of~\eqref{def:cotype}  about which we speculated above is only suggested but not guaranteed by Ribe's theorem. From Ribe's statement we will only infer an indication that there might be a "hidden dictionary" for translating certain linear properties into metric properties, but we will  not be certain that any specific "entry" of this dictionary (e.g.~the entry for, say, "$q_X=\pi$") does in fact exist, and even if it does exist, we will not have an indication what it says. A hallmark of the Ribe program is that at its core it is a search for a family of analogies and definitions, rather than being a collection of specific conjectures. Once such analogies are made and the corresponding questions are formulated, their value is of course determined by the usefulness/depth of the phenomena that they uncover and the theorems that could be proved about them. Thus far, not all of the steps of this endeavour turned out to have a positive answer, but the vast majority did. This had major impact on the study of metric spaces that a priori have nothing to do with Banach spaces, such as graphs, manifolds, groups, and metrics that arise in algorithms (e.g.~as continuous relaxations).

The first written formulation of the plan to uncover a hidden dictionary between normed spaces and metric spaces is the following quote of Bourgain~\cite{Bourgain-trees}, a decade after Ribe's theorem appeared.

\blockquote{{ \small \em It follows in particular from Ribe's result [...] that the
notions from local theory of normed spaces are determined by the metric
structure of the space and thus have a purely metrical formulation. The next step
consists in studying these metrical concepts in general metric spaces in an
attempt to develop an analogue of the linear theory. A detailed exposition of this
program will appear in J. Lindenstrauss's forthcoming survey [...] in our "dictionary" linear operators are translated in Lipschitz maps, the operator norm by the Lipschitz constant of the map [...] The translations of
"Banach-Mazur distance" and "finite-representability" in linear theory are
immediate. At the roots of the local theory of normed spaces are properties such
as type, cotype, superreflexivity [...]  The
analogue of type in the geometry of metric spaces is [...] A simple metrical invariant replacing the notion of cotype
was not yet discovered.
\begin{flushright}{\small \em J.~Bourgain, 1986.}\end{flushright}}}
Unfortunately, the survey of Lindenstrauss that is  mentioned  above  never appeared. Nonetheless, Lindenstrauss had massive impact on this area as a leader who helped  set the course of the Ribe program, as well as due to the important theorems that he proved in this direction. In fact, the article~\cite{JL84} of Johnson and Lindenstrauss, where the aforementioned metric dimension reduction lemma was proved,  appeared a few years before~\cite{Bourgain-trees} and contained inspirational (even prophetic) ideas that had major subsequent impact on the Ribe program (including on Bourgain's works in this area). In the above quote, we removed the text describing "the analogue of type in the geometry of metric spaces" so as to not digress; it refers to the influential work of Bourgain, Milman and Wolfson~\cite{BMW} (see also the earlier work of Enflo~\cite{Enf78} and the subsequent work of Pisier~\cite{Pisier-type}). "Superreflexivity" was the main focus of~\cite{Bourgain-trees}, where the corresponding step of the Ribe program was completed (we will later discuss and use a refinement of this solution). An answer to the above mentioned question on cotype, which we will soon describe, was subsequently found by Mendel and the author~\cite{MN08-cotype}.
We will next  explain the terminology "finite-representability" in the above quote, so as to facilitate the ensuing discussion.

\subsubsection{Finite representability} The first decades of work on the geometry of Banach spaces focused almost entirely on  an inherently infinite dimensional theory.  This was governed by Banach's partial ordering of Banach spaces~\cite[Chapter~7]{Ban32}, which declares that a Banach space $X$ has "linear dimension" at most that of a Banach space $Y$ if there exists $\alpha\ge 1$ such that $X$ is $\alpha$-isomorphic to a subspace of $Y$. In a remarkable feat  of foresight, the following quote of Grothendieck~\cite{Gro53-dvo-conj} heralded the {\em local theory of Banach spaces}, by shifting attention to the geometry of the finite dimensional subspaces of a Banach space as a way to understand its global structure.

\blockquote{{ \small\em assouplissons la notion de "dimension lin\'eaire" de Banach, en disant que
l'espace norm\'e E a un {\em type lin\'eaire} inf\'erieur \`a celui d'un espace norm\'e F, si
on peut trouver un $M>0$ fixe tel que tout sous-espace de dimension finie $E_1$
de E soit isomorphe "\`a M pr\`es" \`a un sous-espace $F_1$ de $F$ (i.e. il existe une
application lin\'eaire biunivoque de $E_1$ sur $F_1$ , de norme $\le 1$, dont l\'application inverse a une norme $\le 1+M$).

\begin{flushright}{\small \em A.~Grothendieck, 1953.}\end{flushright}}}
Grothendieck's work in the 1950s exhibited astounding (technical and conceptual) ingenuity and insight that go well-beyond merely defining a key concept, as he did above. In particular, in~\cite{Gro53-dvo-conj} he conjectured an important phenomenon\footnote{This phenomenon was situated within the Ribe program by Bourgain, Figiel and Milman~\cite{BFM86}, and as such it eventually had ramifications to a well-studied (algorithmic) form of metric dimension reduction through its use to "compress" a finite metric space into a data structure called an {\em approximate distance oracle}~\cite{TZ05}. To date, the only known way to construct such a data structure with constant query time (and even, by now, conjecturally sharp approximation factor~\cite{Che14}) is via the nonlinear Dvoretzky theorem of~\cite{MN07-ramsey}, and thus through the Ribe program. For lack of space, we will not discuss this direction here; see the survey~\cite{Nao12}.} that was later proved by Dvoretzky~\cite{Dvo60} (see the discussion in~\cite{Sch17}), and his contributions in~\cite{Gro53} were transformative (e.g.~\cite{LP68,DFS08,KN12,Pis12}). The above definition set the stage for decades of (still ongoing) work on the local theory of Banach spaces which had  major impact on a wide range of mathematical  areas.

The above "softening" of Banach's "linear dimension" is called today {\em finite representability}, following the terminology  of James~\cite{Jam72} (and his important contributions on this topic). Given $\alpha\in [1,\infty)$, a Banach space $X$ is said to be $\alpha$-finitely representable in a Banach space $Y$ if for any $\beta>\alpha$, any finite dimensional subspace of $X$ is $\beta$-isomorphic to a subspace of $Y$ (in the notation of the above quote, $\beta=1+M$); $X$ is  (crudely) finitely representable in $Y$  if there is some $\alpha\in [1,\infty)$ such that $X$ is $\alpha$-finitely representable in $Y$. This means that the finite dimensional subspaces of $X$ are not very different from subspaces of $Y$; if each of $X$ and $Y$ is finitely representable in the other, then this should be viewed as saying that $X$ and $Y$ have the same finite dimensional subspaces (up to a global allowable error that does not depend on the finite dimensional subspace in question). As an important example of the "taming power" of this definition, the {\em principle of local reflexivity} of Lindenstrauss and Rosenthal~\cite{LR69} asserts that even though sometimes $X^{**}\neq X$, it is always true that $X^{**}$ is $1$-finitely representable in $X$. Thus, while in infinite dimensions $X^{**}$ can be much larger than $X$, passing to the bidual cannot produce substantially new finite dimensional structures. The aforementioned Dvoretzky theorem~\cite{Dvo60} asserts that $\ell_2$ is $1$-finitely representable in any infinite dimensional Banach space.  As another example of a landmark theorem on finite representability, Maurey and Pisier strengthened~\cite{MP-type-cotype} Theorem~\ref{thm:MP cotype} by showing that $\ell_{q_X}$ is $1$-finitely representable in any infinite dimensional Banach space $X$.

Isomorphic local linear properties of Banach spaces are defined to be those properties that are preserved under finite representability. As an example, one should keep in mind  finitary inequalities such as the cotype condition~\eqref{def:cotype}. The formal statement of Ribe's rigidity theorem~\cite{Ribe76} is
\begin{theorem}\label{thm:ribe} Uniformly homeomorphic Banach spaces $X$ and $Y$ are finitely representable in each other.
\end{theorem}

The "immediate translation" of finite representability in the  above quoted text from~\cite{Bourgain-trees} is to define that for $\alpha\in [1,\infty)$  a metric space $\MM$ is $\alpha$-finitely representable in a metric space $\NN$ if  $\cc_\NN(\sub)\le \alpha$ for {\em every} finite subset $\sub\subset \MM$. By doing so one does not induce any terminological conflict, because one can show that a Banach space $X$ is (linearly) $\alpha$-finitely representable in a Banach space $Y$ if and only if $X$ is $\alpha$-finitely representable in $Y$ when $X$ and $Y$ are viewed as metric spaces. This statement follows from "soft" reasoning that is explained in~\cite{GNS12} (relying on a $w^*$-differentiation argument of Heinrich and Mankiewicz~\cite{HM82} as well as properties of ultrapowers of Banach spaces and the aforementioned principle of local reflexivity), though it also follows from Ribe's original proof of Theorem~\ref{thm:ribe} in~\cite{Ribe76}, and a  different (quantitative) approach to this statement was obtained in~\cite{Bou87}.

\subsubsection{Universality and dichotomies} Say that a metric space $\MM$ is (finitarily)  universal if there is $\alpha\ge 1$ such that $\cc_\MM(\mathscr{F})\le \alpha$ for {\em every} finite metric space $\mathscr{F}$. By~\cite{Mat92-ramsey,MN08-cotype}, this requirement holds for some $\alpha\ge 1$ if and only if it holds for $\alpha=1$ (almost-isometric embeddings), so the notion of universality turns out   to be insensitive to the underlying distortion bound. Since for every $n\in \N$, any $n$-point metric space $(\mathscr{F}=\{x_1,\ldots,x_n\},d_\mathscr{F})$ is isometric to a subset of $\ell_\infty$ via the embedding  $(x\in \mathscr{F})\mapsto (d_\mathscr{F}(x,x_i))_{i=1}^n$ (Fr\'echet embedding~\cite{Fre10}),  a different way to state the notion of universality is to say that $\MM$ is universal if $\ell_\infty$ is finitely representable in $\MM$.

Determining whether a given metric space is universal is a subtle matter. By Theorem~\ref{thm:MP cotype}, for a Banach space $X$ this is the same as asking to determine whether $q_X=\infty$. Such questions include major difficult issues in functional analysis that have been studied for a long time; as notable examples, see the works~\cite{Pel77,Bou84} on the (non)universality of the dual of the disc algebra, and the characterization~\cite{BM85} of Sidon subsets of the dual of a compact Abelian group $G$ in terms of the universality of their span is the space of continuous functions on $G$. Here are three famous concrete situations in which it is  unknown if a certain specific space is universal.

\begin{question}[Pisier's dichotomy problem] For each $n\in \N$ let $X_n$ be an arbitrary linear subspace of $\ell_\infty^n$ satisfying
\begin{equation}\label{eq:Pisier dichotomy}
\limsup_{n\to \infty}\frac{\dim(X_n)}{\log n}=\infty.
\end{equation}
Pisier conjectured~\cite{Pis81} that~\eqref{eq:Pisier dichotomy} forces the $\ell_2$ (Pythagorean)  direct sum $(X_1\oplus X_2\oplus\ldots)_2$ to be universal. By duality, a positive answer to this question is  equivalent to the following appealing statement on the geometry of polytopes. For $n\in \N$, suppose that $K\subset \R^n$ is an origin-symmetric polytope with $e^{o(n)}$ faces. Then, for each $\d>0$ there is $k=k(n,\d)\in \n$ with $\lim_{n\to \infty} k(n,\d)=\infty$, a subspace $F=F(n,\d)$ of $\R^n$ with $\dim(F)=k$ and a parallelepiped $Q\subset F$ (thus, $Q$ is an image of $[-1,1]^{k}$ under an invertible  linear transformation) such that $Q\subset K\cap F\subset (1+\d)Q$. Hence, a positive answer to Pisier's dichotomy conjecture  implies that every centrally symmetric polytope with $e^{o(n)}$ faces has a central section of dimension $k$ going to $\infty$ (as a function of the specific $o(n)$ dependence in the underlying assumption), which is $(1+\d)$-close to a polytope (a parallelepiped) with only $O(k)$ faces. The use of "dichotomy" in the name of this conjecture is due to the fact that this conclusion does not hold with $o(n)$ replaced by $O(n)$, as seen by considering polytopes that approximate the Euclidean ball. More generally, by the "isomorphic version" of the Dvoretzky theorem due to Milman and Schechtman~\cite{MS95}, for every  sequence of normed spaces $\{Y_n\}_{n=1}^\infty$ with $\dim(Y_n)=n$ (not only $Y_n=\ell_\infty^n$, which is the case of interest above), and every $k(n)\in\n$ with $k(n)=O(\log n)$, there is a subspace $X_n\subset Y_n$ with $\dim(X_n)=k(n)$ such that the space $(X_1\oplus X_2\oplus\ldots)_2$ is isomorphic to a Hilbert space, and hence in particular it is not universal. The best-known bound in Pisier's dichotomy conjecture appears in the forthcoming work of Schechtman and Tomczak-Jaegermann~\cite{ST18}, where it is shown that the desired conclusion does indeed hold true  if~\eqref{eq:Pisier dichotomy} is replaced by the stronger assumption $\limsup_{n\to\infty} \dim(X_n)/((\log n)^2(\log\log n)^2)=\infty$; this is achieved~\cite{ST18} by building on ideas of Bourgain~\cite{Bou84-pisier}, who obtained the same conclusion if $\limsup_{n\to\infty} \dim(X_n)/(\log n)^4>0$. Thus, due to~\cite{ST18} the above statement  about almost-parallelepiped  central sections of centrally symmetric polytopes does hold true if the initial polytope is assumed to have $\exp(o(\sqrt{n}/\log n))$ faces.
\end{question}

Prior to stating the next question on universality (which, apart from its intrinsic interest, plays a role in the ensuing discussion on metric dimension reduction), we need to very briefly recall some basic notation and terminology from optimal transport (see e.g.~\cite{AGS08,Vil09}). Suppose that $(\MM,d_\MM)$ is a separable complete metric space and fix $p\in [1,\infty)$. Denote by $\mathsf{P}_1(\MM)$ the set of all Borel probability measures $\mu$ on $\MM$ of finite $p$'th moment, i.e., those Borel probability measure $\mu$ on $\MM$ for which $\int_\MM d_\MM(x,y)^p\ud \mu(y)<\infty$ for all $x\in \MM$. A probability measure $\pi\in \mathsf{P}_p(\MM\times \MM)$ is a called a coupling of $\mu,\nu\in \mathsf{P}_p(\MM)$ if $\mu(A)=\pi(A\times\MM)$ and $\nu(A)=\pi(\MM\times A)$ for every Borel measurable subset $A\subset \MM$. The Wasserstein-$p$ distance between $\mu,\nu\in \mathsf{P}_p(\MM)$, denoted $\W_p(\mu,\nu)$, is defined to be the infimum of $(\iint_{\MM\times \MM} d_\MM(x,y)^p\ud \pi(x,y))^{1/p}$ over all couplings $\pi$ of $\mu,\nu$. Below,  $\mathsf{P}_p(\MM)$ is always assumed to be endowed with the metric $\W_p$. The following question is from~\cite{Bourgain-trees}.

\begin{question}[Bourgain's universality problem]\label{Q:bourgain universality} Is $\mathsf{P}_1(\R^2)$ universal? This formulation may  seem different from the way it is asked in~\cite{Bourgain-trees}, but, as explained~\cite[Section~1.5]{ANN15}, it is equivalent to it. More generally, is $\mathsf{P}_1(\R^k)$ universal for some integer $k\ge 2$ (it is simple to see that $\mathsf{P}_1(\R)$ is not universal)? In~\cite{Bourgain-trees} it was proved that $\mathsf{P}_1(\ell_1)$ is universal (see also the exposition in~\cite{Ost13}). So, it is important here that the underlying space is finite dimensional, though to the best of our knowledge it is also unknown whether  $\mathsf{P}_1(\ell_2)$ is universal, or, for that matter, if $\mathsf{P}_1(\ell_p)$ is universal for any  $p\in (1,\infty)$. See~\cite{ANN15} for a (sharp) universality property of $\mathsf{P}_p(\R^3)$ if $p\in (1,2]$.
\end{question}

For the following open question about universality (which will also play a role in the subsequent discussion on metric dimension reduction), recall the notion~\cite{Gro52,Gro55} of {\em projective tensor product} of Banach spaces. Given two Banach spaces $(X,\|\cdot\|_X)$ and $(Y,\|\cdot\|_Y)$, their projective tensor product $X\tp Y$ is the completion of their algebraic tensor product  $X\otimes Y$ under the norm whose unit ball is the convex hull of the simple tensors of vectors of norm at most $1$, i.e., the convex hull of the set $\{x\otimes y\in X\otimes Y:\ \|x\|_X,\|y\|_Y\le 1\}$. For example, $\ell_1\tp X$ can be naturally identified with $\ell_1(X)$, and $\ell_2\tp \ell_2$ can be naturally identified with Schatten--von Neumann trace class $\S_1$ (recall that for $p\in [1,\infty]$, the Schatten--von Neumann trace class $\S_p$  is the Banach space~\cite{vN37}  of those compact linear operators $T:\ell_2\to \ell_2$ for which $\|T\|_{\S_p}=(\sum_{j=1}^\infty\sss_j(T)^p)^{1/p}<\infty$, where $\{\sss_j(T)\}_{j=1}$  are the singular values of $T$); see the monograph~\cite{Rya02} for much more on tensor products of Banach spaces.

It is a longstanding endeavour in Banach space theory to understand which properties of Banach spaces are preserved under projective tensor products; see~\cite{DFS03,BNR12} and the references therein for more on this research direction. Deep work of Pisier~\cite{Pis83} shows that there exist two Banach spaces $X,Y$ that are not universal (even of cotype $2$) such that $X\tp Y$ is universal. The following question was posed by Pisier in~\cite{Pis-92-clapem}.

\begin{question}[universality of projective tensor products] Suppose that $p\in (1,2)$. Is $\ell_p\tp \ell_2$ universal? We restricted the range of $p$ here because it is simple to check that $\ell_1\tp \ell_2\cong\ell_1(\ell_2)$ is not universal, Tomczak-Jaegermann~\cite{Tom74} proved that $\ell_2\otimes \ell_2\cong \S_1$ is not universal, and Pisier proved~\cite{Pis92} that $\ell_p\tp \ell_q$ is not universal when $p,q\in [2,\infty)$. It was also asked in~\cite{Pis-92-clapem} if $\ell_2\tp\ell_2\tp\ell_2$ is universal. The best currently available result in this direction (which will be used below) is that, using the local theory of Banach spaces and recent work on locally decodable codes, it was shown in~\cite{BNR12} that if $a,b,c\in (1,\infty)$ satisfy $\frac{1}{a}+\frac{1}{b}+\frac{1}{c}\le 1$, then $\ell_a\tp\ell_b\tp\ell_c$ is universal.
\end{question}

The following theorem is a union of several results of~\cite{MN08-cotype}.

\begin{theorem}\label{thm:union MN} The following conditions are equivalent for a metric space $(\MM,d)$.
\begin{itemize}
\item $\MM$ is not universal.
\item There exists $q=q(\MM)\in (0,\infty)$  with the following property. For every $n\in \N$ there is $m=m(n,\MM,q)\in \N$ such that any collection of points $\{x_w\}_{w\in \Z_{2m}^n}$ in $\MM$ satisfies the inequality
\begin{equation}\label{eq:metric cotype}
 \sum_{i=1}^n \sum_{w\in \Z_{2m}^n} \frac{d(x_{w+me_i},x_w)^2}{m^2}\lesssim_{X,q} \frac{n^{1-\frac{2}{q}}}{3^n}\sum_{\e\in \{-1,0,1\}^n}\sum_{w\in \Z_{2m}^n}d(x_{w+\e},x)^2.
\end{equation}
Here $e_1,\ldots,e_n$ are the standard basis of $\R^n$ and addition (in the indices) is modulo $2m$.
\item There is $\theta(\MM)\in (0,\infty)$ with the following property.  For arbitrarily large $n\in \N$ there exists an $n$-point metric space $(\mathcal{B}_n,d_n)$ such that
    \begin{equation}
    \cc_\MM(\mathcal{B}_n)\gtrsim_\MM (\log n)^{\theta(\MM)}.
    \end{equation}
\end{itemize}
Moreover, if we assume that $\MM$ is a Banach space rather than an arbitrary metric space, then for any $q\in (0,\infty]$ the validity of~\eqref{eq:metric cotype} (as stated, i.e., for each $n\in \N$ there is $m\in \N$ for which~\eqref{eq:metric cotype} holds for every configuration $\{x_w\}_{w\in \Z_{2m}^n}$ of points in $\MM$) is equivalent to~\eqref{def:cotype}, and hence, in particular,  the infimum over those $q\in (0,\infty]$ for which~\eqref{eq:metric cotype} holds true is equal to the infimal cotype $q_\MM$ of $\MM$.
\end{theorem}

The final sentence of Theorem~\ref{thm:union MN} is an example of a successful step in the Ribe program, because it reformulates the local linear invariant~\eqref{def:cotype} in purely metric terms, namely as the quadratic geometric inequality~\eqref{eq:metric cotype} that imposes a  restriction on the behavior of pairwise distances within any configuration of $(2m)^n$ points $\{x_w\}_{w\in \Z_{2m}^n}$ (indexed by the discrete torus $\Z_{2m}^n$) in the given Banach space.

With this at hand, one can consider~\eqref{eq:metric cotype} to be a property of a  metric space, while initially~\eqref{def:cotype}  made sense only for a normed space. As in the case of~\eqref{def:cotype}, if $q=\infty$, then~\eqref{eq:metric cotype} holds in any metric space $(\MM,d_\MM)$ (for any $m\in \Z$,  with the implicit constant in~\eqref{eq:metric cotype}  being universal);  by its general nature, such a statement must of course be nothing more than a formal consequence of the triangle inequality, as carried  out in~\cite{MN08-cotype}. So, the validity of~\eqref{eq:metric cotype} for $q<\infty$ could be viewed as an asymptotic (randomized) enhancement of the triangle inequality in $(\MM,d_\MM)$; by considering the the canonical realization of  $\Z_{2m}^n$ in $\C^n$, namely the points $\{(\exp(\pi iw_1/m),\ldots, \exp(\pi iw_n/m))\}_{w\in \Z_{{2m}}^n}$, equipped with the metric inherited from $\ell_\infty^n(\C)$, one checks that not every metric space satisfies this requirement. The equivalence of the first two bullet points in Theorem~\ref{thm:union MN} shows that once one knows that a metric space  is not universal, one deduces the validity of such an enhancement of the triangle inequality. This is an analogue of Theorem~\ref{thm:MP cotype} of Maurey and Pisier for general metric spaces.

The equivalence of the first and third bullet points in Theorem~\ref{thm:union MN} yields the following dichotomy. If one finds a finite metric space $\mathscr{F}$ such that $\cc_{\mathscr{F}}(\MM)>1$, then there are arbitrary large finite metric spaces whose minimal distortion in $\MM$ is at least a fixed positive power (depending on $\MM$) of the logarithm their cardinality. Hence, for example, if every $n$-point metric space embeds into $\MM$ with distortion $O(\log\log n)$, then actually for any $\d>0$, every finite metric space embeds into $\MM$ with distortion $1+\d$. See~\cite{Men09,MN11-arxiv,MN13,ANN15} for more on metric dichotomies of this nature, as well as a quite delicate counterexample~\cite{MN13} for a natural variant for trees (originally asked by C. Fefferman). It remains a  mystery~\cite{MN08-cotype} if the power of the logarithm $\theta(\MM)$ in Theorem~\ref{thm:union MN} could be bounded from below by a universal positive constant, as formulated in the following open question.

\begin{question}[metric cotype dichotomy problem] Is there a universal constant $\theta>0$ such that in Theorem~\ref{thm:union MN}  one could take $\theta(\MM)>\theta$. All examples that have been computed thus far leave the possibility that even $\theta(\MM)\ge 1$, which would be sharp (for $\MM=\ell_2$) by Bourgain's embedding theorem~\cite{Bou85}.  Note, however, that in~\cite{ANN15} it is asked whether for the Wasserstein space $\mathsf{P}_p(\R^3)$ we have $\liminf_{p\to 1} \theta(\mathsf{P}_p(\R^3))=0$. If this were true, then it would resolve  the metric cotype dichotomy problem negatively. It would be interesting to understand the bi-Lipschitz structure of these spaces of measures on $\R^3$ regardless of this context, due to their independent importance.
\end{question}

Theorem~\ref{thm:union MN} is a good illustration of a "vanilla" accomplishment of the Ribe program, since it obtains a metric reformulation of a key isomorphic linear property of metric spaces, and also proves statements about general metric spaces which are inspired by the analogies with the linear theory that the Ribe program is aiming for. However, even in this particular setting of metric cotype, Theorem~\ref{thm:union MN} is only a part of the full picture, as it has  additional purely metric ramifications. Most of these  rely on a delicate issue that has been suppressed in the above statement of Theorem~\ref{thm:union MN}, namely that of understanding the asymptotic behavior of $m=m(n,\MM,q)$  in~\eqref{eq:metric cotype}. This matter is not yet fully resolved even when $\MM$ is a Banach space~\cite{MN08-cotype,GNS11}, and generally such questions seem to be quite challenging (see~\cite{MN07-scaled,GN10,Nao16-riesz} for related issues). Thus far, whenever this question was answered for specific (classes of) metric spaces, it led to interesting geometric applications; e.g.~its resolution for certain Banach spaces in~\cite{MN08-cotype} was used in~\cite{Nao12-quasi} to answer a longstanding question~\cite{Vai99} about quasisymmetric embeddings, and its resolution for  Alexandrov spaces of (global) nonpositive curvature~\cite{Ale51} (see e.g.~\cite{BGP92,Stu03} for the relevant background) in the forthcoming work~\cite{EMN17}  is used there to answer a longstanding question about the coarse geometry of such Alexandrov spaces.

\begin{comment}

\begin{theorem} There exists a metric space $(\mathscr{B},d_\mathscr{B})$ that does not admit a coarse embedding into any Alexandorv space of global nonpositive curvature $(\MM,d_\MM)$, i.e., if $\omega,\Omega:[0,\infty)\to [0,\infty)$ are any two nondecreasing moduli with $\lim_{s\to \infty}\omega(s)=\infty$, then there does not exist $f:\mathscr{B}\to \MM$ satisfying
\begin{equation}\label{eq:coarse defined in theorem}
\forall\, x,y\in \mathscr{B},\qquad \omega\big(d_\mathscr{B}(x,y)\big)\le d_\MM\big(f(x),f(y)\big)\le \Omega\big(d_\mathscr{B}(x,y)\big).
\end{equation}
\end{theorem}
\end{comment}

%\subsubsection{Markov convexity}

\subsection{Metric dimension reduction}\label{sec:metric dim reduction intro} By its nature, many aspects of the local theory of Banach spaces involve describing phenomena that rely on dimension-dependent estimates. In the context of the Ribe program, the goal is to formulate/conjecture analogous phenomena for metric spaces, which is traditionally governed by asking Banach space-inspired questions about a finite metric space $(\MM,d_\MM)$ in which $\log |\MM|$ serves as a replacement for the  dimension. This analogy arises naturally also in the context of the bi-Lipschitz embedding problem into $\R^k$ (Problem~\ref{Q:bilip Rk}); see Remark~\ref{rem:doub log n} below. Early successful instances of this analogy can be found in the work of Marcus and Pisier~\cite{MP84}, as well as the aforementioned work of Johnson and Lindentrauss~\cite{JL84}. However, it should be stated at the outset that over the years it became clear that while making this analogy is the right way to get "on track" toward the discovery of fundamental metric phenomena, from the perspective of the Ribe program the reality is much more nuanced and, at times, even unexpected and surprising.

\begin{comment}

This endeavour/hope is well-described in the following quote~\cite{Mat96}.

\blockquote{{ \small \em This investigation started in the context of the local Banach space theory, where
the general idea was to obtain some analogs for general metric spaces of notions and
results dealing with the structure of finite dimensional subspaces of Banach spaces.
The distortion of a mapping should play the role of the norm of a linear operator,
and the quantity $\log n$, where $n$ is the number of points in a metric space, would
serve as an analog of the dimension of a normed space. Parts of this programme
have been carried out by [...]

\begin{flushright}{\small \em J.~Matou\v{s}ek, 1996.}\end{flushright}}}
\end{comment}

Johnson and Lindenstrauss asked~\cite[Problem~3]{JL84}  whether  every finite metric space $\MM$ embeds with distortion $O(1)$ into some normed space $X_\MM$ (which is allowed to depend on $\MM$) of dimension $\dim(X_\MM)\lesssim \log |\MM|$. In addition to arising from the above background, this question is  motivated by~\cite[Problem~4]{JL84}, which asks if  the Euclidean distortion of every finite metric space $\MM$  satisfies $\cc_2(\MM)\lesssim \sqrt{\log |\MM|}$. If so, this would have served as a very satisfactory metric  analogue of John's theorem~\cite{Joh48}, which asserts that any finite dimensional normed space $X$ is $\sqrt{\dim(X)}$-isomorphic to a subspace of $\ell_2$. Of course, John's theorem shows that a positive answer to the former question~\cite[Problem~3]{JL84} formally implies a positive answer to the latter question~\cite[Problem~4]{JL84}.

The aforementioned Johnson--Lindenstrauss lemma~\cite{JL84} (JL lemma, in short) shows that, at least for finite subsets of  a Hilbert space, the answer to the above stated~\cite[Problem~3]{JL84} is positive.
\begin{theorem}[JL lemma]\label{thm:JL alpha} For each $n\in \N$ and $\alpha\in (1,\infty)$, there is $k\in \n$ with  $k\lesssim_\alpha \log n$ such that any $n$-point subset of $\ell_2$ embeds into $\ell_2^k$ with distortion $\alpha$.
\end{theorem}
We postpone discussion of this fundamental geometric fact to Section~\ref{sec:JL} below, where it is examined in detail and its proof is presented. Beyond Hilbert spaces, there is only one other example (and variants thereof) of a Banach space for which it is currently known that~\cite[Problem~3]{JL84} has a positive answer for any of its finite subsets, as shown in the following theorem from~\cite{JN10}.
\begin{theorem}\label{thm:tsirelson} There is a Banach space $\mathscr{T}^{(2)}$ which is not isomorphic to a Hilbert space yet it has the following property. For every finite subset $\sub\subset \mathscr{T}^{(2)}$ there is $k\in \n$ with $k\lesssim\log |\sub|$ and a $k$-dimensional linear subspace $F$ of $\mathscr{T}^{(2)}$ such that $\sub$ embeds into $F$ with $O(1)$ distortion.
\end{theorem}
The space $\mathscr{T}^{(2)}$ of Theorem~\ref{thm:tsirelson}  is not very quick to describe, so we refer to~\cite{JN10} for the details. It suffices to say here  that this space is the $2$-convexification of the classical Tsirelson space~\cite{Cir74,FJ74}, and that the proof that it satisfies the stated dimension reduction result is obtained in~\cite{JN10} via a concatenation of several (substantial) structural results in the literature; see Section~4 in~\cite{JN10} for a discussion of variants of this construction, as well as related open questions.  The space $\mathscr{T}^{(2)}$ of Theorem~\ref{thm:tsirelson} is not isomorphic to a Hilbert space, but barely so: it is explained in~\cite{JN10} that for every $n\in \N$ there exists an $n$-dimensional  subspace $F_n$ of $\mathscr{T}^{(2)}$ with $\cc_2(F_n)\ge e^{c\mathsf{Ack}^{-1}(n)}$, where $c>0$ is a universal constant and $\mathsf{Ack}^{-1}(\cdot)$ is the inverse of the Ackermann function from  computability theory (see e.g.~\cite[Appendix~B]{AKNSS08}). So, indeed $\lim_{n\to \infty} \cc_2(F_n)=\infty$, but at a  tremendously slow rate.

Remarkably, despite major scrutiny for over $3$ decades, it remains unknown if~\cite[Problem~3]{JL84}  has a positive answer for subsets of {\em any} non-universal classical Banach space. In particular, the following  question is open.
\begin{question}\label{Q:lp into some space}
Suppose that $p\in [1,\infty)\setminus\{2\}$. Are there $\alpha=\alpha(p),\beta=\beta(p)\in [1,\infty)$ such that for any $n\in \N$, every $n$-point subset of $\ell_p$ embeds with distortion $\alpha$ into some $k$-dimensional normed space with $k\le \beta\log n$?
\end{question}
It is even open if in Question~\ref{Q:lp into some space} one could obtain a  bound of $k=o(n)$ for any fixed $p\in [1,\infty)\setminus\{2\}$. Using John's theorem as above, a positive answer to Question~\ref{Q:lp into some space} would imply that $\cc_2(\sub)\lesssim_p \sqrt{\log |\sub|}$ for any finite subset $\sub$ of $\ell_p$. At present, such an embedding statement is not known for {\em any} $p\in [1,\infty)\setminus\{2\}$, though for $p\in [1,2]$ it is known~\cite{ALN08} that any $n$-point subset of $\ell_p$ embeds into $\ell_2$ with distortion $(\log n)^{1/2+o(1)}$; it would be interesting to obtain any $o(\log n)$ bound here for any fixed $p\in (2,\infty)$, which would be a  "nontrivial" asymptotic behavior in light of the following general theorem~\cite{Bou85}.
\begin{theorem}[Bourgain's embedding theorem]\label{thm:bourgain embedding} $\cc_2(\MM)\lesssim \log |\MM|$ for every finite metric space $\MM$.
\end{theorem}

The above questions from~\cite{JL84}  were the motivation for the influential work~\cite{Bou85}, where Theorem~\ref{thm:bourgain embedding} was proved. Using a probabilistic construction and the JL lemma, it was shown in~\cite{Bou85} that Theorem~\ref{thm:bourgain embedding} is almost sharp in the sense that there are arbitrarily large $n$-point metric spaces $\MM_n$ for which $\cc_2(\MM_n)\gtrsim (\log n)/\log\log n$. By John's theorem, for every $\alpha\ge 1$, if $X$ is a finite dimensional normed space and $\cc_X(\MM_n)\le \alpha$, then $\cc_2(\MM_n)\le \alpha\sqrt{\dim(X)}$. Therefore the above lower bound on $\cc_2(\MM_n)$ implies that  $\dim(X)\gtrsim (\log n)^2/(\alpha^2(\log\log n)^2)$.

The  achievement of~\cite{Bou85} is thus twofold. Firstly, it discovered Theorem~\ref{thm:bourgain embedding} (via the introduction of an influential randomized embedding method), which is the "correct"  metric version of John's theorem in the Ribe program. The reality turned out to be more nuanced in the sense that the answer  is not quite as good as the $O(\sqrt{\log n})$ that was predicted in~\cite{JL84}, but the $O(\log n)$ of Theorem~\ref{thm:bourgain embedding} is still a  strong and useful phenomenon that was discovered through the analogy that the Ribe program provided. Secondly, we saw above that~\cite[Problem~3]{JL84} was {\em disproved} in~\cite{Bou85}, though  the "bad news" that follows from~\cite{Bou85} is only mildly worse than  the $O(\log n)$ dimension bound that~\cite[Problem~3]{JL84} predicted, namely a dimension lower bound that grows quite slowly, not faster than $(\log n)^{O(1)}$. Curiously, the very availability of strong dimension reduction in $\ell_2$ through the JL lemma is what was harnessed in~\cite{Bou85} to deduce that any "host normed space" that contains $\MM_n$ with $O(1)$ distortion must have dimension at least of order $(\log n/\log\log n)^2\gg \log n$. Naturally, in light of these developments, the question of understanding what is the correct asymptotic behavior of the smallest $k(n)\in \N$ such that any $n$-point metric space embeds with distortion $O(1)$ into a $k(n)$-dimensional normed space was raised  in~\cite{Bou85}.

In order to proceed, it would be convenient to introduce some notation and terminology.
\begin{definition}[metric dimension reduction modulus]\label{def:modulus bilip} Fix $n\in \N$ and $\alpha\in [1,\infty)$. Suppose that $(X,\|\cdot\|_X)$ is a normed space. Denote by $\kk_n^\alpha(X)$ the minimum $k\in \N$ such that for {\em any} $\sub\subset X$ with $|\sub|=n$ there exists a $k$-dimensional linear subspace $F_\sub$ of $X$ into which $\sub$ embeds with distortion $\alpha$.
\end{definition}

The quantity $\kk_n^\alpha(\ell_\infty)$ was introduced by Bourgain~\cite{Bou85} under the notation $\psi_\alpha(n)= \kk_n^\alpha(\ell_\infty)$; see also~\cite{AR92,Mat96} where this different notation persists, though for the sake of uniformity of the ensuing discussion we prefer not to use it here because we will treat $X\neq\ell_\infty$ extensively. \cite{Bou85} focused for concreteness  on the  arbitrary value $\alpha=2$, and asked for the asymptotic behavior $\kk_n^2(\ell_\infty)$ as $n\to \infty$.

 An $n$-point subset of $X=\ell_\infty$ is nothing more than a general $n$-point metric space, via the aforementioned isometric Fr\'echet embedding. In the same vein, a $k$-dimensional linear subspace of $\ell_\infty$ is nothing more than a general $k$-dimensional normed space $(F,\|\cdot\|_F)$ via the linear isometric embedding $(x\in F)\mapsto (x_i^*(x))_{i=1}^\infty$, where $\{x_i^*\}_{i=1}^\infty$ is an arbitrary sequence of linear functionals on $F$ that are dense in the unit sphere of the dual space $F^*$. Thus, the quantity $\kk_n^\alpha(\ell_\infty)$ is the smallest $k\in \N$ such that every $n$-point metric space $(\MM,d_\MM)$ can be realized with distortion at most $\alpha$ as a subset of $(\R^k,\|\cdot\|_\MM)$ for some norm $\|\cdot\|_\MM:\R^k\to \R^k$ on $\R^k$ (which, importantly, is allowed to be adapted to the initial metric space $\MM$), i.e., $\cc_{(\R^k,\|\cdot\|_\MM)}(\MM,d_\MM)\le \alpha$. This is precisely the quantity that~\cite[Problem~3]{JL84} asks about, and above we have seen that~\cite{Bou85} gives the lower bound
 \begin{equation}\label{eq:bourgain displayed log log}
\kk_n^\alpha(\ell_\infty)\gtrsim  \left(\frac{\log n}{\alpha\log\log n}\right)^2.
\end{equation}

 Theorem~\ref{thm:summary} below is a summary of the main nontrivial\footnote{Trivially $\kk_n^\alpha(X)\le n-1$, by  considering in Definition~\ref{def:modulus bilip} the subspace $F_\sub=\spn(\sub-x_0)$ for any fixed $x_0\in \sub$.} bounds on the modulus $\kk_n^\alpha(X)$ that are currently known for specific Banach spaces $X$. Since  such a "combined statement" contains a large amount information and covers a lot of the literature on this topic, we suggest reading it in tandem with the subsequent discussion, which includes further clarifications and explanation of the history of the respective results. A "take home" message from the statements below is that despite major efforts by many researchers, apart from information on metric dimension reduction for $\ell_2$, the space  $\mathscr{T}^{(2)}$ of Theorem~\ref{thm:tsirelson}, $\ell_\infty$, $\ell_1$ and $\S_1$, nothing is known for other spaces (even for $\ell_1$ and $\S_1$ more remains to be done, notably with respect to bounding $\kk_n^\alpha(\ell_1),\kk_n^\alpha(\S_1)$ from above).

 \begin{theorem}[summary of the currently known upper and lower bounds on metric dimension reduction]\label{thm:summary} There exist universal constants $c,C>0$ such that the following assertions hold true for every integer $n\ge 20$.
 \begin{enumerate}
\item In the Hilbertian setting, we have the sharp bounds
\begin{equation}\label{eq:JL in theporem}
\forall\, \alpha\ge 1+\frac{1}{\sqrt[3]{n}},\qquad  \kk_n^\alpha(\ell_2)\asymp \frac{\log n}{\log(1+(\alpha-1)^2)}\asymp \max\left\{\frac{\log n}{(\alpha-1)^2},\frac{\log n}{\log \alpha}\right\}.
\end{equation}
\item For the space $\mathscr{T}^{(2)}$ of Theorem~\ref{thm:tsirelson}, there exists $\alpha_0\in [1,\infty)$ for which $\kk_n^{\alpha_0}(\mathscr{T}^{(2)})\asymp \log n$.
\item For $\ell_\infty$, namely in the setting of~\cite[Problem~3]{JL84}, we have
\begin{equation}\label{eq:alpha<2}
\forall\, \alpha\in [1,2),\qquad \kk_n^\alpha(\ell_\infty)\asymp n,
\end{equation}
and
\begin{equation}\label{eq:linfty bounds in theorem}
\forall\, \alpha\ge 2,\qquad n^{\frac{c}{\alpha}}+\frac{\log n}{\log\left(\frac{\log n}{\alpha}+\frac{\alpha\log\log n}{\log n}\right)}\lesssim \kk_n^\alpha(\ell_\infty)\lesssim \frac{n^{\frac{C}{\alpha}}\log n}{\log\left(1+\frac{\alpha}{\log n}\right)}.
\end{equation}
\item For $\ell_1$, we have
\begin{equation}\label{eq:ell1 in theorem}
\forall\, \alpha\ge 1,\qquad n^{\frac{c}{\alpha^2}}+\frac{\log n}{\log(\alpha+1)}\lesssim \kk_{n}^\alpha(\ell_1)\lesssim \frac{n}{\alpha}.
\end{equation}
Moreover, if  $\alpha\ge 2C\sqrt{\log n}\log\log n$, then we have the better upper bound
\begin{equation}\label{eq:use ALN}
 \kk_n^\alpha(\ell_1)\lesssim \frac{\log n}{\log\left(\frac{\alpha}{C\sqrt{\log n}\log\log n}\right)}.
\end{equation}
\item For the Schatten--von Neumann trace class $\S_1$, we have
\begin{equation}\label{eq:S1 in theorem}
\forall\, \alpha\ge 1,\qquad \kk_n^\alpha(\S_1)\gtrsim n^{\frac{c}{\alpha^2}}+\frac{\log n}{\log(\alpha+1)}.
\end{equation}
 \end{enumerate}
 \end{theorem}

The bound $\kk_n^\alpha(\ell_2)\lesssim (\log n)/\log(1+(\alpha-1)^2)$ in~\eqref{eq:JL in theporem} restates Theorem~\ref{thm:JL alpha} (the JL lemma) with the implicit dependence on $\alpha$ now stated explicitly; it actually holds for every $\alpha> 1$, as follows from the original proof in~\cite{JL84} and explained in Section~\ref{sec:JL} below. The restriction $\alpha\ge 1+1/\sqrt[3]{n}$ in~\eqref{eq:JL in theporem}  pertains only to the corresponding lower bound on $\kk_n^\alpha(\ell_2)$, which exhibits different behaviors in the low-distortion and high-distortion regimes.

Despite scrutiny of many researchers over the past 3 decades, only very recently the dependence on $\alpha$ in the JL lemma when $\alpha$ is arbitrarily close to $1$ but independent of $n$ was proved to be sharp by Larsen and Nelson~\cite{GN17} (see~\cite{Alo03,Alo09,GN16} for earlier results in this direction, as well as the subsequent work~\cite{AK17}). This is so even when $\alpha$ is allowed to tend to $1$ with $n$, and even in a somewhat larger range than the requirement $\alpha\ge 1+1/\sqrt[3]{n}$  in~\eqref{eq:JL in theporem} (see~\cite{GN17} for the details), though there remains a small range of values of $\alpha$ ($n$-dependent, very close to $1$) for which it isn't currently known what is the behavior of $\kk_n^\alpha(\ell_2)$. The  present article is focused on embeddings that permit large errors, and in particular in ways to prove impossibility results even if  large errors are allowed. For this reason, we will not describe here the ideas of the proof in~\cite{GN17} that pertains to the almost-isometric regime.

For, say, $\alpha\ge 2$, it is much simpler to see that the $\kk_n^\alpha(\ell_2)\gtrsim (\log n)/\log\alpha$, in even greater generality that also explains the appearance of the term $(\log n)/\log(\alpha+1)$ in~\eqref{eq:ell1 in theorem} and~\eqref{eq:S1 in theorem}. One could naturally generalize Definition~\ref{def:modulus bilip} so as to introduce the following notation for {\em relative metric dimension reduction moduli}. Let $\mathcal{F}$ be a family of metric spaces and $\YY$ be a family of normed spaces. For $n\in \N$ and $\alpha\in [1,\infty)$, denote by $\kk_n^\alpha(\mathcal{F},\YY)$ the minimum $k\in \N$ such that for every $\MM\in \mathcal{F}$ with $|\MM|=n$ there exists $Y\in \YY$ with $\dim(Y)=k$ such that $\cc_Y(\MM)\le \alpha$. When $\mathcal{F}$ is the collection of all the finite subsets of a fixed Banach space $X$, and $\YY$ is the collection of all the finite-dimensional linear subspaces of a fixed Banach space $Y$, we  use the simpler notation $\kk_n^\alpha(\mathcal{F},\YY)=\kk_n^\alpha(X,Y)$. Thus, the modulus $\kk_n^\alpha(X)$ of Definition~\ref{def:modulus bilip} coincides with $\kk_n^\alpha(X,X)$. Also, under this notation Question~\eqref{Q:lp into some space} asks if for $p\in [1,\infty)\setminus \{2\}$ we have $\kk_n^\alpha(\ell_p,\ell_\infty)\lesssim_p \log n$ for some $1\le \alpha\lesssim_p 1$. The study of the modulus $\kk_n^\alpha(\mathcal{F},\YY)$  is essentially a completely unexplored area, partially because even our understanding of the "vanilla" dimension reduction modulus $\kk_n^\alpha(X)$  is currently very limited. By a short volumetric argument that is presented  in Section~\ref{sec:JL} below, every infinite dimensional Banach space $X$ satisfies
\begin{equation}\label{lower brunnminkowksi intro}
\forall (n,\alpha)\in \N\times [1,\infty),\qquad \kk_n^\alpha(X,\ell_\infty)\ge \frac{\log n}{\log(\alpha+1)}.
\end{equation}
Hence also $\kk_n^\alpha(X)\ge (\log n)/\log(\alpha+1)$, since~\eqref{lower brunnminkowksi intro} rules out embeddings into any normed space of dimension less than $(\log n)/\log(\alpha+1)$, rather than only into such spaces that are also subspaces of $X$.

Using an elegant Fourier-analytic argument, Arias-de-Reyna and Rodr\'iguez-Piazza proved in~\cite{AR92} that for every $\alpha\in [1,2)$ we have $\kk_n^\alpha(\ell_\infty)\gtrsim (2-\alpha)n$. This was slightly improved by Matou\v{s}ek~\cite{Mat96} to $\kk_n^\alpha(\ell_\infty)\gtrsim n$, i.e., he showed that the constant multiple of $n$ actually remains bounded below by a positive constant as $\alpha\to 2^-$ (curiously, the asymptotic behavior of $\kk_n^2(\ell_\infty)$ remains unknown). These results establish~\eqref{eq:alpha<2}. So, for sufficiently small distortions one cannot hope to embed every $n$-point metric space into some normed space of dimension $o(n)$. For larger distortions (our main interest), it was conjectured in~\cite{AR92} that $\kk_n^\alpha(\ell_\infty)\lesssim (\log n)^{O(1)}$ if $\alpha>2$.

The bounds in~\eqref{eq:linfty bounds in theorem} refute this conjecture of~\cite{AR92}, since they include the lower bound $\kk_n^\alpha(\ell_\infty)\gtrsim n^{c/\alpha}$, which is a landmark achievement of Matou\v{s}ek~\cite{Mat96} (obtained a decade after Bourgain asked about the asymptotics here and over a decade after Johnson and Lindenstrauss posed the question whether $\kk_n^\alpha(\ell_\infty)\lesssim_\alpha\log n$). It is, of course, an exponential improvement over Bourgain's bound~\eqref{eq:bourgain displayed log log}. Actually, in the intervening period Linial, London and Rabinovich~\cite{LLR} removed the iterated logarithm in the lower bound of~\cite{Bou85} by showing that Theorem~\ref{thm:bourgain embedding} (Bourgain's embedding) is sharp up to the value of the implicit universal constant. By the same reasoning as above (using John's theorem), this also removed the iterated logarithm from the denominator in~\eqref{eq:bourgain displayed log log}, i.e., \cite{LLR} established that $\kk_n^\alpha(\ell_\infty)\gtrsim (\log n)^2/\alpha^2$. This was the best-known bound prior to~\cite{Mat96}.

Beyond proving a fundamental geometric theorem, which, as seen in~\eqref{eq:linfty bounds in theorem}, is  optimal up to the constant in the exponent, this work of Matou\v{s}ek is important because it injected a refreshing approach from real algebraic geometry into this area, which was previously governed by considerations from analysis, geometry, probability and combinatorics. Section~\ref{sec:matousek} covers this outstanding contribution in detail, and obtains the following stronger  statement that wasn't previously noticed in the literature but follows from an adaptation of Matou\v{s}ek's ideas.

\begin{theorem}[impossibility of coarse dimension reduction]\label{thm:coarse matousek} There is a universal constant $c\in (0,\infty)$ with the following property. Suppose that $\omega,\Omega:[0,\infty)\to [0,\infty)$ are increasing functions that satisfy $\omega(s)\le \Omega(s)$ for all $s\in [0,\infty)$, as well as $\lim_{s\to\infty}\omega(s)=\infty$. Define
\begin{equation}\label{eq:def beta}
\beta(\omega,\Omega)\eqdef \sup_{s\in (0,\infty)} \frac{s}{\omega^{-1}\big(2\Omega(s)\big)}\in (0,1).
\end{equation}
For arbitrarily large $n\in \N$ there is a metric space $(\MM,d_{\MM})=\left(\MM(n,\omega,\Omega),d_{\MM(n,\omega,\Omega)}\right)$
with $|\MM|=3n$ such that for any normed space $(X,\|\cdot\|_X)$, if there exists $f:\MM\to X$ which satisfies
\begin{equation}\label{eq:coarse condition}
\forall\,x,y\in \MM,\qquad \omega\big(d_\MM(x,y)\big)\le \|f(x)-f(y)\|_X\le \Omega\big(d_\MM(x,y)\big),
\end{equation}
then necessarily %$\dim(X)\gtrsim n^{c\beta(\omega,\Omega)}$.
\begin{equation}\label{eq:dim lower coarse}
\dim(X)\gtrsim n^{c\beta(\omega,\Omega)}.
\end{equation}
\end{theorem}
A mapping that satisfies~\eqref{eq:coarse condition} is called a {\em coarse embedding} (with moduli $\omega,\Omega$), as introduced in Gromov's seminal work~\cite{Gro93} and studied extensively ever since, with a variety of interesting applications (see the monographs~\cite{Roe03,NY12,Ost13} and the references therein). The bi-Lipschitz requirement~\eqref{eq:def embedding} corresponds to $\omega(s)=\tau s$ and $\Omega(s)=\alpha\tau s$ in~\eqref{eq:coarse condition}, in which case~\eqref{eq:dim lower coarse} becomes Matou\v{s}ek's aforementioned lower bound on $\kk_n^\alpha(\ell_\infty)$. Theorem~\ref{thm:coarse matousek} asserts that there exist arbitrarily large finite metric spaces that cannot be embedded even with a very weak (coarse) guarantee into any low-dimensional normed space, with the dimension of the host space being forced to be at least a power of their cardinality, which is exponentially larger than the logarithmic behavior that one would predict from the natural ball-covering requirement that is induced by low-dimensionality (see the discussion of the doubling condition in Section~\ref{sec:doubling}, as well as the proof of~\eqref{lower brunnminkowksi intro} in Section~\ref{sec:JL}).

\begin{remark}\label{rem:snowflake} Consider the following special case of  Theorem~\ref{thm:coarse matousek}. Fix $\theta\in (0,1]$ and let $(\MM,d_\MM)$ be a metric space. It is straightforward to check that $d_\MM^\theta:\MM\times \MM\to [0,\infty)$ is also a metric on $\MM$. The metric space $(\MM,d_\MM^\theta)$ is commonly called the {\em $\theta$-snowflake} of $\MM$ (in reference to the von Koch snowflake curve; see e.g.~\cite{DS97}) and it is denoted $\MM^\theta$. Given $\alpha\ge 1$, the statement that $\MM^\theta$ embeds with distortion $\alpha$ into a normed space $(X,\|\cdot\|_X)$ is the same as the requirement~\eqref{eq:coarse condition} with $\omega(s)=s^\theta$ and $\Omega(s)=\alpha s^\theta$. Hence, by Theorem~\ref{thm:coarse matousek} there exist arbitrarily large $n$-point metric spaces $\MM_n=\MM_n(\alpha,\theta)$ such that if $\MM_n^\theta$ embeds with distortion $\alpha$ into some $k$-dimensional normed space, then $k\ge n^{c/(2\alpha)^{1/\theta}}$. Conversely, Remark~\ref{rem:snowflake of ell infty} below shows that for every $n\in \N$ and $\alpha>1$, the $\theta$-snowflake of any $n$-point metric space embeds with distortion $\alpha$ into a normed space $X$ with $\dim(X)\lesssim_{\alpha,\theta} n^{C/\alpha^{1/\theta}}$. So, the bound~\eqref{eq:dim lower coarse} of Theorem~\ref{thm:coarse matousek} is quite sharp even for embeddings that are not bi-Lipschitz, though we did not investigate the extent of its sharpness for more general moduli $\omega,\Omega:[0,\infty)\to [0,\infty)$.
\end{remark}

At this juncture, it is natural to complement the (coarse) strengthening in Theorem~\ref{thm:coarse matousek} of Matou\v{s}ek's bound $\kk_n^\alpha(\ell_\infty)\ge n^{c/\alpha}$ by stating the following different type of strengthening, which we recently obtained in~\cite{Nao17}.

\begin{theorem}[impossibility of average dimension reduction]\label{thm:average distortion} There is a universal constant $c\in (0,\infty)$ with the following property. For arbitrarily large $n\in \N$ there is an $n$-point metric space $(\MM,d_{\MM})$
such that for  any normed space $(X,\|\cdot\|_X)$ and any $\alpha\in [1,\infty)$, if there exists $f:\MM\to X$ which satisfies $\|f(x)-f(y)\|_X\le\alpha d_\MM(x,y)$ for all $x,y\in X$, yet $\frac{1}{n^2}\sum_{x,y\in X} \|f(x)-f(y)\|_X\ge \frac{1}{n^2}\sum_{x,y\in X} d_\MM(x,y)$, then necessarily $\dim(X)\ge n^{c/\alpha}$.
\end{theorem}

An $n$-point metric space $\MM$ as in Theorem~\ref{thm:average distortion} is intrinsically high dimensional even on average, in the sense that if one wishes to assign in an $\alpha$-Lipschitz manner to each point in $\MM$ a vector in some normed space $X$ such that the average distance in the image is the same as the average distance in $\MM$, then this forces the abient dimension to satisfy $\dim(X)\ge n^{c/\alpha}$. Prior to Theorem~\ref{thm:average distortion}, the best-known bound here was $\dim(X)\gtrsim (\log n)^2/\alpha^2$, namely the aforementioned lower bound $\kk_n^\alpha(\ell_\infty)\gtrsim (\log n)^2/\alpha^2$ of Linial, London and Rabinovich~\cite{LLR} actually treated the above "average distortion" requirement rather than only the (pairwise) bi-Lipschitz requirement.

\begin{remark} The significance of Theorem~\ref{thm:average distortion} will be discussed further in Section~\ref{sec:average} below; see also~\cite{ANNRW17,ANNRW18}. In Section~\ref{sec:average} we will present a new proof of Theorem~\ref{thm:average distortion}  that is different from (though inspired by) its proof in~\cite{Nao17}. It suffices to say here that the proof of Theorem~\ref{thm:average distortion} is conceptually different from Matou\v{s}ek's approach~\cite{Mat96}. Namely, in contrast to the algebraic/topological argument of~\cite{Mat96}, the proof of Theorem~\ref{thm:average distortion} relies on the theory of nonlinear spectral gaps, which is also an outgrowth of the Ribe program; doing justice to this theory and its ramifications is beyond the scope of the present article (see~\cite{MN14} and the references therein), but the basics are recalled in Section~\ref{sec:average}. Importantly, the proof of Theorem~\ref{thm:average distortion}  obtains a criterion for determining if a given metric space $\MM$ satisfies its conclusion, namely $\MM$ can be taken to be the shortest-path metric of any bounded degree graph with a spectral gap.  This information is harnessed in the forthcoming work~\cite{ANNRW18} to imply that finite-dimensional normed spaces have a structural proprty (a new type of hierarchical partitioning scheme) which has implications to the design of efficient data structures for approximate nearest neighbor search, demonstrating that the omnipresent "curse of dimensionality" is to some extent absent from this fundamental algorithmic task.
\end{remark}

In the intervening period between Bourgain's work~\cite{Bou85} and Matousek's solution~\cite{Mat96}, the question of determining the asymptotic behavior of $\kk_n^\alpha(\ell_\infty)$ was pursued by Johnson, Lindenstrauss and Schechtman, who proved in~\cite{JLS87} that $\kk_n^\alpha(\ell_\infty)\lesssim_\alpha n^{C/\alpha}$ for some universal constant $C>0$.
%, thus the form of Matou\v{s}ek's lower bound matched the upper bound that was available in the literature at the time.
They demonstrated this by constructing for every $n$-point metric space $\MM$ a normed space $X_\MM$, which they (probabilistically) tailored to the given metric space $\MM$, with $\dim(X_\MM)\lesssim_\alpha n^{C/\alpha}$ and such that $\MM$ embeds into $X_\MM$ with distortion $\alpha$. Subsequently, Matou\v{s}ek showed~\cite{Mat92} via a different argument that one could actually work here with $X_\MM=\ell_\infty^k$ for $k\in \N$ satisfying $k\lesssim_\alpha n^{C/\alpha}$, i.e., in order to obtain this type of upper bound on the asymptotic behavior of $\kk_n^\alpha(\ell_\infty)$ one does not need to adapt the target normed space to the metric space $\MM$ that is being embedded. The implicit dependence on $\alpha$ here, as well as the constant $C$ in the exponent, were further improved in~\cite{Mat96}. For $\alpha=O((\log n)/\log\log n)$, the upper bound on $\kk_n^\alpha(\ell_\infty)$ that appears in~\eqref{eq:linfty bounds in theorem} is that of~\cite{Mat96}, and for the remaining values of  $\alpha$ it is due to a more recent improvement over~\cite{Mat96} by Abraham, Bartal and Neiman~\cite{ABN11} (specifically, the upper bound in~\eqref{eq:linfty bounds in theorem} is a combination of Theorem~5 and Theorem~6 of~\cite{ABN11}).

\begin{remark}\label{rem:snowflake of ell infty} An advantage of the fact~\cite{Mat92} that one could take $X_\MM=\ell_\infty^k$ rather than the more general  normed space of~\cite{JLS87} is that it quickly implies the optimality of the lower bound from Remark~\ref{rem:snowflake} on dimension reduction of snowflakes. Fix $n\in \N$, $\alpha>1$ and $\theta\in (0,1]$. Denote $\d=\min\{\sqrt{\alpha}-1,1\}$, so that $\alpha\asymp \alpha/(1+\d)>1$. By~\cite{Mat92}, given an $n$-point metric space $(\MM,d_\MM)$ there is an integer $k\lesssim n^{c/\alpha^{1/\theta}}$ and $f=(f_1,\ldots,f_k):\MM\to \R^k$ such that
$$
\forall\, x,y\in \MM,\qquad d_\MM(x,y)\le \|f(x)-f(y)\|_{\ell_\infty^k}\le \left(\frac{\alpha}{1+\d}\right)^{\frac{1}{\theta}}d_\MM(x,y).
$$
Hence (here it becomes useful that we are dealing with the $\ell_\infty^k$ norm, as it commutes with powering),
$$
\forall\, x,y\in \MM,\qquad d_\MM(x,y)^\theta\le \max_{i\in \k} |f_i(x)-f_i(y)|^\theta \le \frac{\alpha}{1+\d}d_\MM(x,y)^\theta.
$$
By works of Kahane~\cite{Kah81} and Talagrand~\cite{Tal92-helix}, there is $m=m(\d,\theta)$ and a mapping (a quasi-helix) $h:\R\to \R^m$ such that $|s-t|^\theta\le\|h(s)-h(t)\|_{\ell_\infty^m}\le (1+\d)|s-t|^\theta$ for all $s,t\in \R$. The mapping $$(x\in \MM)\mapsto \bigoplus_{i=1}^k h\circ f_i(x)\in \bigoplus_{i=1}^k \ell_\infty^m$$ is a distortion-$\alpha$ embedding of the $\theta$-snowflake $(\MM,d_\MM^\theta)$ into a normed space of dimension $mk\lesssim_{\alpha,\theta}  n^{c/\alpha^{1/\theta}}$. The implicit dependence on $\alpha,\theta$ that~\cite{Kah81,Tal92-helix} imply here is quite good, but likely not sharp  as $\alpha\to 1^+$ when $\theta\neq\frac12$.
\end{remark}

Since the expressions in~\eqref{eq:linfty bounds in theorem} are somewhat involved, it is beneficial to restate them on a case-by-case basis as follows. For sufficiently large $\alpha$, we have a bound\footnote{One can alternatively justify the upper bound in~\eqref{eq:sharp for large alpha} (for sufficiently large $n$) by first using Theorem~\ref{thm:bourgain embedding} (Bourgain's embedding theorem) to embed an $n$-point metric space $\MM$ into $\ell_2$ with distortion $A\log n$ for some universal constant $A\ge 1$, and then using Theorem~\ref{thm:JL alpha} (the JL lemma) with the dependence on the distortion as stated in~\eqref{eq:JL in theporem} to reduce the dimension of the image of $\MM$ under Bourgain's embedding to $O((\log n)/\log(\alpha/(A\log n)))$ while incurring a further distortion of $\alpha/(A\log n)$, thus making the overall distortion be at most $\alpha$. The right hand side  of~\eqref{eq:sharp for large alpha} is therefore  in fact an upper bound on $\kk_n^\alpha(\ell_\infty,\ell_2)$; see also Corollary~\ref{cor:nontrivial type}.} that is sharp up to universal constant factors.
\begin{equation}\label{eq:sharp for large alpha}
\alpha\ge (\log n)\log\log n\implies \kk_n^\alpha(\ell_\infty)\asymp \frac{\log n}{\log\left(\frac{\alpha}{\log n}\right)}.
\end{equation}
For a range of smaller values of $\alpha$, including those $\alpha$ that do not tend to $\infty$ with $n$, we have
\begin{equation}\label{eq:sharp for small alpha}
1\le \alpha\le \frac{\log n}{\log\log n}\implies n^{\frac{c}{\alpha}}\lesssim \kk_n^\alpha(\ell_\infty)\lesssim n^{\frac{C}{\alpha}}.
\end{equation}
\eqref{eq:sharp for small alpha} satisfactorily shows that the asymptotic behavior of $\kk_n^\alpha(\ell_\infty)$ is of  power-type, but it is not as sharp as~\eqref{eq:sharp for large alpha}. We suspect that determining the  correct exponent of $n$ in the power-type dependence of $\kk_n^\alpha(\ell_\infty)$ would be challenging (there is indication~\cite{Mat96,Mat02}, partially assuming a positive answer to a difficult conjecture of Erd\H{o}s~\cite{Erd64,Bol78}, that this exponent has infinitely many jump discontinuities as a function of $\alpha$). In an intermediate range $(\log n)/\log\log n\lesssim \alpha\lesssim (\log n)\log\log n$ the bounds~\eqref{eq:linfty bounds in theorem} are less satisfactory. The case $\alpha\asymp \log n$, corresponding to the distortion in Bourgain's embedding theorem, is especially intriguing, with~\eqref{eq:linfty bounds in theorem} becoming
\begin{equation}\label{eq:log log log}
\frac{\log n}{\log\log \log n}\lesssim \kk_n^{\Theta(\log n)}(\ell_\infty)\lesssim \log n.
\end{equation}
The first inequality in~\eqref{eq:log log log} has not been stated in the literature, and we justify it in Section~\ref{sec:average} below. A more natural lower bound here would be a constant multiple of $(\log n)/\log\log n$, as this corresponds to the volumetric restriction~\eqref{lower brunnminkowksi intro}, and moreover by the upper bound~\eqref{eq:JL in theporem} in the JL lemma we know that any $n$-point subset of a Hilbert space does in fact embed with distortion $\log n$ into $\ell_2^k$ with $k\lesssim (\log n)/\log\log n$. The triple logarithm in~\eqref{eq:log log log} is therefore quite intriguing/surprising, thus leading to the following open question.

\begin{question}\label{Q:log n} Given an integer $n\ge 2$, what is the asymptotic behavior of the smallest $k=k_n\in \N$ for which any $n$-point metric space $\MM$ embeds with distortion $O(\log n)$ into some $k$-dimensional normed space $X_\MM$.
\end{question}

There is a dearth of available upper bounds on $\kk_n^\alpha(\cdot)$, i.e., positive results establishing that metric dimension reduction is  possible. This is especially striking in the case of $\kk_n^\alpha(\ell_1)$, due to the importance of $\ell_1$ from the perspective of pure mathematics and algorithms. The upper bound on $\kk_n^\alpha(\ell_1)$ in the large distortion regime~\eqref{eq:use ALN} follows from combining the Euclidean embedding of~\cite{ALN08} with the JL lemma.  The only general dimension reduction result in $\ell_1$ that lowers the dimension below the trivial bound $\kk_n^\alpha(\ell_1)\le n-1$ is the
forthcoming work~\cite{AN17}, where the estimate $\kk_n^\alpha(\ell_1)\lesssim n/\alpha$ in~\eqref{eq:ell1 in theorem} is obtained; even this modest statement requires effort (among other things, it relies on the sparsification method of Batson, Spielman and Srivastava~\cite{BSS12}).

The bound $\kk_n^\alpha(\ell_1)\ge n^{c/\alpha^2}$  in~\eqref{eq:ell1 in theorem} is a remarkable theorem of Brinkman and Charikar~\cite{BC05} which answered a question that  was at the time open for many years. To avoid any possible confusion, it is important to note that~\cite{BC05} actually exhibits an $n$-point subset $\sub_{\mathsf{BC}}$ of $\ell_1$ for which it is shown in~\cite{BC05} that if $\sub_{\mathsf{BC}}$ embeds with distortion $\alpha$ into $\ell_1^k$, then necessarily $k\ge n^{c/\alpha^2}$. On the face of it, this seems weaker than~\eqref{eq:ell1 in theorem}, because the lower bound on $\kk_n^\alpha(\ell_1)$ in~\eqref{eq:ell1 in theorem} requires showing  that if $\sub_{\mathsf{BC}}$ embeds into an {\em arbitrary} finite-dimensional linear subspace $F$ of $\ell_1$, then necessarily $\dim(F)\ge n^{c/\alpha^2}$. However, Talagrand proved~\cite{Tal90} that in this setting for every $\beta> 1$ the subspace $F$ embeds with distortion $\beta$ into $\ell_1^k$, where $k\lesssim_\beta \dim(F)\log \dim(F)$. From this, an application of the above stated result of~\cite{BC05}  gives that $\dim(F)\log \dim(F)\gtrsim n^{c/\alpha^2}$, and so the lower bound in~\eqref{eq:ell1 in theorem} follows from the formulation in~\cite{BC05}. Satisfactory analogues of the above theorem of Talagrand are known~\cite{Sch87,BLM89,Tal95} (see also the survey~\cite{JS01} for more on this subtle issue) when $\ell_1$ is replaced by $\ell_p$ for some $p\in (1,\infty)$, but such reductions to "canonical" linear subspaces are not available elsewhere,    so the above reasoning is a rare "luxury" and in general one must treat arbitrary low-dimensional linear subspaces of the Banach space in question.

The above difficulty was overcome for  $\mathsf{S}_1$ in~\cite{NPS18}, where~\eqref{eq:S1 in theorem} was proven. The similarity of the lower bounds in~\eqref{eq:ell1 in theorem} and~\eqref{eq:S1 in theorem} is not coincidental.  One can view the Brinkman--Charikar example $\sub_{\mathsf{BC}}\subset \ell_1$ also as a collection of diagonal matrices in $\S_1$, and~\cite{NPS18}  treats this very same subset by strengthening the assertion of~\cite{BC05}  that $\sub_{\mathsf{BC}}$  does not well-embed into low-dimensional subspaces of $\mathsf{S_1}$ which consist entirely of diagonal matrices, to the same assertion for low-dimensional subspaces of $\mathsf{S}_1$ which are now allowed to consist of any matrices whatsoever. Using our notation for relative dimension reduction moduli, this gives the stronger assertion $\kk_n^\alpha(\ell_1,\mathsf{S}_1)\ge n^{c/\alpha^2}$.

A geometric challenge of the above discussion is that, even after one isolates a candidate $n$-point subset $\sub$ of  $\ell_1$ that is suspected not to be realizable with $O(1)$ distortion in low-dimensions (finding such a suspected intrinsically high-dimensional set  is of course a major challenge in itself), one needs to devise a way to somehow argue that if one could find a configuration of $n$ points in a low-dimensional subspace $F$ of $\ell_1$ (or $\mathsf{S}_1$) whose pairwise distances are within a fixed, but potentially very large, factor $\alpha\ge 1$ of the corresponding pairwise distances within $\sub$ itself, then this would force the ambient dimension $\dim(F)$ to be very large. In~\cite{BC05} this was achieved via a clever proof that relies on linear programming; see also~\cite{ACNNH11} for a variant of this linear programming approach in the almost isometric regime $\alpha\to 1^+$. In~\cite{Reg13} a different proof of the Brinkman--Charikar theorem was found, based on information-theoretic reasoning. Another entirely different geometric method to prove that theorem  was devised in~\cite{LN04}; see also~\cite{LMN05,JS09} for more applications of  the approach of~\cite{LN04}.

Very recently, a further geometric approach was obtained in~\cite{NPS18}, where it was used to derive a stronger statement that, as shown in~\cite{NPS18}, cannot follow from the method of~\cite{LN04} (the statement is that the $n$-point subset $\sub_{\mathsf{BC}}\subset \ell_1$ is not even an $\alpha$-Lipschitz quotient of any subset of a low dimensional subspace of $\mathsf{S}_1$; see~\cite{NPS18} for the relevant definition an a complete discussion). The approach of~\cite{NPS18} relies on an invariant that arose in the Ribe program and is called {\em Markov convexity}. Fix $q>0$. Let $\{\chi_t\}_{t\in \Z}$ be a Markov chain on a state space $\Omega$. Given an integer $k\ge 0$, denote by $\{\widetilde{\chi}_t(k)\}_{t\in \Z}$ the process that equals $\chi_t$ for time $t\le k$, and evolves independently of $\chi_t$ (with respect to the same transition probabilities) for time $t>k$. Following~\cite{LNP09}, the Markov $q$-convexity constant of a metric space $(\MM,d_\MM)$, denoted $\Pi_q(\MM)$, is the infimum over those $\Pi\in [0,\infty]$ such that for every Markov chain $\{\chi_t\}_{t\in \Z}$ on a state space $\Omega$ and every $f:\Omega\to \MM$ we have
\begin{equation*}\label{eq:def markov convexity}
\bigg(\sum_{k=1}^\infty\sum_{t\in \Z}\frac{1}{2^{qk}}\E\left[d_\MM \big(f\big(\widetilde{\chi}_t(t-2^k)\big),f(\chi_t)\big)^q\right]\bigg)^{\frac1{q}}\le \Pi\bigg(\sum_{t\in \Z} \E\Big[d_\MM\big(f(\chi_t),f(\chi_{t-1})\big)^q\Big]\bigg)^{\frac1{q}}.
\end{equation*}
By~\cite{LNP09,MN13},  a Banach space $X$ satisfies $\Pi_q(X)<\infty$ if and only if it has an equivalent norm $|||\cdot|||:X\to [0,\infty)$ whose modulus of uniform convexity has power type $q$, i.e., $|||x+y|||\le 2-\Omega_X(|||x-y|||^q)$ for every $x,y\in X$ with $|||x|||=|||y|||=1$. This completes the step in the Ribe program which corresponds to the local linear property  "$X$ admits an equivalent norm whose modulus of uniform convexity has power type $q$," and it is a refinement of the aforementioned characterization of superreflexivity in~\cite{Bourgain-trees} (which by deep results of~\cite{Enf72,Pisier-martingales} corresponds to the cruder local linear property "there is a finite $q\ge 2$ for which $X$ admits an equivalent norm whose modulus of uniform convexity has power type $q$"). By~\cite{MN13,EMN17-bifurcation}, the Brinkman--Charikar subset $\sub_{\mathsf{BC}}\subset \ell_1$ (as well as a variant of it due to Laakso~\cite{Laakso} which has~\cite{LMN05} the same non-embeddability property into low-dimensional subspaces of $\ell_1$) satisfies  $\Pi_q(\sub_{\mathsf{BC}})\gtrsim (\log n)^{1/q}$ for every $q\ge 2$ (recall that in our notation $|\sub_{\mathsf{BC}}|=n$). At the same time, it is proved in~\cite{NPS18} that $\Pi_2(F)\lesssim \sqrt{\log \dim(F)}$ for every finite dimensional subset of $\mathsf{S}_1$. It remains to contrast these asymptotic behaviors (for $q=2$) to deduce that if $\cc_F(\sub_{\mathsf{BC}})\le \alpha$, then necessarily $\dim(F)\ge n^{c/\alpha^2}$.

Prior to the forthcoming work~\cite{NY17-3dim},  the set $\sub_{\mathsf{BC}}$ (and variants thereof of a similar nature) was the only known example that demonstrates that there is no $\ell_1$ analogue of the JL-lemma. The following theorem is from~\cite{NY17-3dim}.

\begin{theorem}\label{thm:heisenberg} There is a universal constant $c\in (0,\infty)$ with the following property. For arbitrarily large $n\in \N$ there exists an $n$-point $O(1)$-doubling subset $\mathscr{H}_n$ of $\ell_1$ satisfying  $\cc_4(\mathscr{H}_n)\lesssim 1$, such that for every $\alpha\in [1,\infty)$ and every finite-dimensional linear subspace $F$ of $\S_1$, if $\mathscr{H}_n$ embeds into $F$ with distortion $\alpha$, then necessarily
\begin{equation}\label{eq:heisenberg lower}
\dim(F)\ge \exp\Big(\frac{c}{\alpha^2}\sqrt{\log n}\Big).
\end{equation}
\end{theorem}
 See Section~\ref{sec:doubling} for the (standard) terminology "doubling" that is used in Theorem~\ref{thm:heisenberg}. While~\eqref{eq:heisenberg lower} is  weaker than the  lower bound of Brinkman--Charikar in terms of the dependence on $n$, it nevertheless rules out metric dimension reduction in $\ell_1$ (or $\mathsf{S}_1$) in which the target dimension is, say, a power of $\log n$. The example $\mathscr{H}_n$ of Theorem~\ref{thm:heisenberg} embeds with distortion $O(1)$ into $\ell_4$, and hence in particular $\sup_{n\in \N}\Pi_4(\mathscr{H}_n)\lesssim \Pi_4(\ell_4)<\infty$, by~\cite{LNP09}. This makes $\mathscr{H}_n$ be qualitatively different from all the previously known examples which exhibit the impossibility of metric dimension reduction in $\ell_1$, and as such its existence has further ramifications that answer longstanding questions; see~\cite{NY17-3dim} for a detailed discussion. The proof of Theorem~\ref{thm:heisenberg} is markedly different from (and more involved than)  previous proofs of impossibility of dimension reduction in $\ell_1$, as it relies on new geometric input (a subtle property of the $3$-dimensional Heisenberg group which fails for the $5$-dimensional Heisenberg group) that is obtained in~\cite{NY17-3dim}, in combination with results from~\cite{ANT13,LN14-poincare,LN14,HN16}; full details appear in~\cite{NY17-3dim}.

\subsection{Spaces admitting bi-Lipschitz and average metric dimension reduction}\label{sec:spaces} Say that an infinite dimensional Banach space $(X,\|\cdot\|_X)$ {\em admits  metric dimension reduction} if there is $\alpha=\alpha_X\in [1,\infty)$ such that
\begin{equation*}\label{eq:def admits}
\lim_{n\to \infty} \frac{\log \kk_n^\alpha(X)}{\log n}=0.
\end{equation*}
In other words, the requirement here is that for some $\alpha=\alpha_X\in [1,\infty)$ and every $n\in \N$, any $n$-point subset $\sub\subset X$ embeds with (bi-Lipschitz) distortion $\alpha$ into some linear subspace $F$  of $X$ with $\dim(F)=n^{o_X(1)}$.

Analogously, we say that $(X,\|\cdot\|_X)$  {\em admits  average metric dimension reduction} if there is $\alpha=\alpha_X\in [1,\infty)$ such that for any $n\in \N$ there is $k_n=n^{o_X(1)}$, i.e., $\lim_{n\to\infty} (\log k_n)/\log n=0$, such that for any $n$-point subset $\sub$ of $X$ there is a linear subspace $F$  of $X$ with $\dim(F)=k_n$ and a mapping $f:\sub\to F$ which is $\alpha$-Lipschitz, i.e., $\|f(x)-f(y)\|_X\le \alpha \|x-y\|_X$ for all all $x,y\in \sub$, yet
\begin{equation}\label{eq:av dist requirement}
\frac{1}{n^2}\sum_{x\in \sub}\sum_{y\in \sub}\|f(x)-f(y)\|_X\ge \frac{1}{n^2}\sum_{x\in \sub}\sum_{y\in \sub}\|x-y\|_X.
\end{equation}

Our choice here of the  behavior $n^{o_X(1)}$ for the target dimension is partially motivated by the available results, based on which this type of asymptotic behavior appears to be a benchmark.
%Also, algorithms on $k$-dimensional normed spaces with an approximation factor of $O(\log k)$ yield $o(\log n)$ if and only if  $k=n^{o(1)}$, which in many cases is the "nontrivial range" of parameters since $O(\log n)$ would follow from Bourgain's embedding theorem (Theorem~\ref{thm:bourgain embedding}); this occurs e.g.~in~\cite{ANNRW18}.
We stress, however, that since the repertoire of spaces that are known to admit metric dimension reduction is currently very limited, finding any new setting in which one could prove that reducing dimension to $o_X(n)$ is possible would be a highly sought after achievement. In the same vein, finding new spaces for which one could prove a metric dimension reduction lower bound that   tends to $\infty$ faster than $\log n$  (impossibility of a JL-style guarantee) would be very interesting.

\begin{remark} In the above definition of spaces admitting average metric dimension reduction we imposed the requirement~\eqref{eq:av dist requirement} following the terminology that was introduced by Rabinovich in~\cite{Rab08}, and due to the algorithmic usefulness of this notion of embedding. However, one could also consider natural variants such as  $(\frac{1}{n^2}\sum_{x\in \sub}\sum_{y\in \sub}\|f(x)-f(y)\|_X^p)^{1/p}\ge (\frac{1}{n^2}\sum_{x\in \sub}\sum_{y\in \sub}\|x-y\|_X^q)^{1/q}$ for any $p,q\in (0,\infty]$, and much of the ensuing discussion holds mutatis mutandis in this setting as well.
\end{remark}

The only Banach spaces that are currently known to admit metric dimension reduction are Hilbert spaces and the space $\mathscr{T}^{(2)}$ of Theorem~\ref{thm:tsirelson} (and variants thereof). These examples allow for the possibility that if $(X,\|\cdot\|_X)$ admits metric dimension reduction, i.e., $\kk_n^{O_X(1)}(X)=n^{o_X(1)}$, then actually $\kk_n^{O_X(1)}(X)=O_X(\log n)$. Such a dichotomy would of course be remarkable, but there is insufficient evidence to conjecture that this is so.

The available repertoire of spaces that admit average metric dimension reduction is larger, since if $p\in [2,\infty)$, then  $\ell_p$ and even $\mathsf{S}_p$ satisfy the assumption of the following theorem, by~\cite{Maz29} and~\cite{Ric15}, respectively.

\begin{theorem}\label{thm:average in lp} Let $(X,\|\cdot\|_X)$ be an infinite dimensional Banach space with unit ball $B_X=\{x\in X:\ \|x\|_X\le 1\}$. Suppose that there is a Hilbert space $(H,\|\cdot\|_H)$ and a one-to-one mapping $f:B_X\to H$ such that $f$ is Lipschitz and $f^{-1}:f(B_X)\to X$ is uniformly continuous. Then $X$ admits average metric dimension reduction. In fact, this holds for embeddings into a subspace of logarithmic dimension, i.e., there is $\alpha=\alpha_X\in [1,\infty)$ such that for any $n\in \N$ and any $n$-point subset $\sub$ of $X$ there is a linear subspace $F$ of $X$ with $\dim(F)\lesssim \log n$ and a mapping $f:\sub\to F$ which satisfies both~\eqref{eq:av dist requirement} and $\|f(x)-f(y)\|_X\le \alpha\|x-y\|_X$ for all $x,y\in \sub$.
\end{theorem}

\begin{proof}This statement is implicit in~\cite{Nao14}. By combining~\cite[Proposition~7.5]{Nao14} and~\cite[Lemma~7.6]{Nao14} there is a $O_X(1)$-Lipschitz mapping $f:\sub\to \ell_2$ which satisfies $\frac{1}{n^2}\sum_{x\in \sub}\sum_{y\in \sub} \|f(x)-f(y)\|_{2}\ge \frac{1}{n^2}\sum_{x\in \sub}\sum_{y\in \sub} \|x-y\|_X$. By the JL lemma we may assume that $f$ actually takes values in $\ell_2^k$ for some $k\lesssim \log n$. Since $X$ is infinite dimensional, Dvoretzky's theorem~\cite{Dvo60} ensures that $\ell_2^k$ is $2$-isomorphic to a $k$-dimensional subspace $F$ of $X$.
\end{proof}

\begin{remark} By~\cite{Maz29}, for $p\in [2,\infty)$ the assumption of Theorem~\ref{thm:average in lp} holds for $X=\ell_p$. An inspection of the proofs in~\cite{Nao14} reveals that the dependence of the Lipschitz constant $\alpha=\alpha_p$ on $p$ that Theorem~\ref{thm:average in lp} provides for $X=\ell_p$ grows to $\infty$ exponentially with $p$. As argued in~\cite[Section~5.1]{Nao14} (using metric cotype), this exponential behavior is unavoidable using the above proof. However, in this special case a more sophisticated argument of~\cite{Nao14} yields $\alpha_p\lesssim p^{5/2}$; see equation (7.40) in~\cite{Nao14}. Motivated by~\cite[Corollary~1.6]{Nao14}, we conjecture that this could be improved to $\alpha_p\lesssim p$, and there is some indication (see~\cite[Lemma~1.11]{Nao14}) that this would be sharp.
\end{remark}

Prior to~\cite{Nao17}, it was not known if there exists a Banach space which fails to admit average metric dimension reduction. Now we know (Theorem~\ref{thm:average distortion}) that $\ell_\infty$ fails to admit average metric dimension reduction, and therefore also any universal Banach space fails to admit average metric dimension reduction. A fortiori, the same is true also for (non-average) metric dimension reduction, but this statement follows from the older  work~\cite{Mat96}. Failure of average metric dimension reduction is not known for {\em any} non-universal (finite cotype) Banach space, and it would be very interesting to provide such an example. By~\cite{BC05,NPS18} we know that $\ell_1$ and $\mathsf{S}_1$ fail to admit  metric dimension reduction, but this is not known for average distortion, thus leading to the following question.

\begin{question}\label{Q:l1 average} Does $\ell_1$ admit average metric dimension reduction? Does $\ell_p$ have this property for any $p\in [1,2)$?
\end{question}
All of the available examples of $n$-point subsets of $\ell_1$ for which  the $\ell_1$ analogue of the JL lemma fails (namely if $k=O(\log n)$, then they do not embed with $O(1)$ distortion into $\ell_1^{k}$) actually embed into the real line $\R$ with $O(1)$ average distortion; this follows from~\cite{Rab08}. Specifically, the examples in~\cite{BC05,LMN05} are the shortest-path metric on planar graphs, and the example in Theorem~\ref{thm:heisenberg} is $O(1)$-doubling, and both of these classes of metric spaces are covered by~\cite{Rab08}; see also~\cite[Section~7]{Nao14} for generalizations. Thus, the various known proofs which demonstrate that the available examples cannot be embedded into a low dimensional subspace of $\ell_1$ argue that any such low-dimensional embedding must highly distort  some distance, but this is not so for a typical distance. A negative answer to Question~\ref{Q:l1 average} would therefore require a substantially new type of construction which exhibits a much more "diffuse" intrinsic high-dimensionality despite it being a subset of $\ell_1$.  In the reverse direction, a positive answer to Question~\ref{Q:l1 average}, beyond its intrinsic geometric/structural interest, could have algorithmic applications.

\subsubsection{Lack of stability under projective tensor products} Prior to the recent work~\cite{NPS18}, it was unknown whether the property of admitting metric dimension reduction is preserved under projective tensor products.
\begin{corollary}\label{coro:S1} There exist Banach spaces $X,Y$ that admit metric dimension reduction yet $X\tp Y$ does not.
\end{corollary}
Since $\S_1$ is isometric to $\ell_2\tp\ell_2$ and~\cite{NPS18} establishes that $\S_1$ fails to admit metric dimension reduction, together with the JL lemma this implies  Corollary~\ref{coro:S1} (we can thus even have $X=Y$ and $\kk_n^\alpha(n)\lesssim_\alpha \log n$ for all $\alpha>1$).

Since we do not know whether $\S_1$ admits average metric dimension reduction (the above comments pertaining to Question~\ref{Q:l1 average} are valid also for $\S_1$), the analogue of Corollary~\ref{coro:S1}  for average metric dimension reduction was previously unknown. Here we note  the following statement, whose proof is a somewhat curious argument.

\begin{theorem}\label{thm:tensor average} There exist Banach spaces $X,Y$ that admit average metric dimension reduction yet $X\tp Y$ does not. Moreover, for every $p\in (2,\infty)$ we can take here $X=\ell_p$.
\end{theorem}

\begin{proof} By~\cite{BNR12} (which relies on major input from the theory of locally decodable codes~\cite{Efr09}  and an important inequality of Pisier~\cite{Pis80}), the $3$-fold product $\ell_3\tp\ell_3\tp\ell_3$ is universal. So, by the recent work~\cite{Nao17} (Theorem~\ref{thm:average distortion}), $\ell_3\tp\ell_3\tp\ell_3$ does not admit average metric dimension reduction. At the same time, by Theorem~\ref{thm:average in lp} we know that $\ell_3$ admits average metric dimension reduction. So, if $\ell_3\tp\ell_3$ fails to admit average metric dimension reduction, then we can take $X=Y=\ell_3$ in Theorem~\ref{thm:tensor average}. Otherwise,  $\ell_3\tp\ell_3$ does admit average metric dimension reduction, in which case we can take $X=\ell_3$ and $Y=\ell_3\tp\ell_3$. Thus, in either of the above two cases, the first assertion of Theorem~\ref{thm:tensor average} holds true. The second assertion of Theorem~\ref{thm:tensor average} follows by repeating this argument using the fact~\cite{BNR12} that $\ell_p\tp \ell_p\tp  \ell_{q}$ is universal if $2/p+1/q\le 1$, or equivalently $q\ge p/(p-2)$. If we choose, say, $q=\max\{2,p/(p-2)\}$, then by Theorem~\ref{thm:average in lp} we know that both $\ell_p$ and $\ell_q$ admit average metric dimension reduction, while $\ell_p\tp \ell_p\tp  \ell_{q}$ does not. So, the second assertion of Theorem~\ref{thm:tensor average} holds for either $Y=\ell_p$ or $Y=\ell_p\tp\ell_q$.
\end{proof}

The proof of Theorem~\ref{thm:tensor average} establishes  that at least one of the  pairs $(X=\ell_3,Y=\ell_3)$ or $(X=\ell_3,Y=\ell_3\tp\ell_3)$ satisfies its conclusion, but it gives no indication which of these two options occurs. This naturally leads to

\begin{question}\label{Q:l3} Does $\ell_3\tp\ell_3$ admit average metric dimension reduction?
\end{question}
A positive answer to Question~\ref{Q:l3} would yield a new space that admits average metric dimension reduction. In order to claim that $\ell_3\tp\ell_3$ is indeed new in this context, one must show that it does not satisfy the assumption of Theorem~\ref{thm:average in lp}. This is so because $\S_1$ (hence also $\ell_1$) is finitely representable in $\ell_3\tp \ell_3$; see e.g.~\cite[page~61]{DFS03}. The fact that no Banach space in which $\ell_1$ is finitely representably satisfies the assumption of Theorem~\ref{thm:average in lp} follows by combining~\cite[Lemma~1.12]{Nao14}, \cite[Proposition~7.5]{Nao14}, and~\cite[Lemma~7.6]{Nao14}. This also shows that a positive answer to Question~\ref{Q:l3} would imply that any $n$-point subset of $\ell_1$ (or $\S_1$) embeds with $O(1)$ average distortion into some normed space (a  subspace of $\ell_3\tp\ell_3$) of dimension $n^{o(1)}$, which is a statement in the spirit of Question~\ref{Q:l1 average}. If the answer to Question~\ref{Q:l3}  were negative, then $\ell_3\tp\ell_3$  would be the first example of a  non-universal space which fails to admit average metric dimension reduction, because Pisier proved~\cite{Pis92,Pis-92-clapem} that $\ell_3\tp\ell_3$ is not universal.

Another question that arises naturally from Theorem~\ref{thm:tensor average}  is whether its conclusion holds true also for $p=2$.

\begin{question} Is there a Banach space $Y$ that admits average metric dimension reduction yet $\ell_2\tp Y$ does not?
\end{question}

\subsubsection{Wasserstein spaces} Let $(\MM,d_\MM)$ be a metric space  and $p\in [1,\infty)$. The Wasserstein space $\mathsf{P}_p(\MM)$ is not a Banach space, but there is a natural version of the metric dimension reduction question in this context.

\begin{question}\label{Q:wasserstein}
 Fix $\alpha>1$, $n\in \N$ and $\mu_1,\ldots,\mu_n\in \mathsf{P}_p(\MM)$. What is the asymptotic behavior of the smallest $k\in \N$ for which there is $\mathcal{S}\subset \MM$ with $|\mathcal{S}|\le k$ such that $(\{\mu_1,\ldots,\mu_n\},\W_p)$ embeds with distortion $\alpha$ into $\mathsf{P}_p(\mathcal{S})$?
\end{question}
%Analogously to the above discussion, one can  formulate mutatis mutandis "average versions" of Question~\ref{Q:wasserstein}.

Spaces of measures with the Wasserstein metric $\W_p$ are of major importance in pure and applied
 mathematics,  as well as in computer science (mainly for $p=1$, where they are used in graphics and vision, but also for other values of $p$; see e.g.~the discussion in~\cite{ANN16-ICALP}). However, their bi-Lipschitz structure is poorly understood, especially so in the above context of metric dimension reduction. If $k$ were small in Question~\ref{Q:wasserstein}, then this would give a way to "compress" collections of measures using measures with small support while approximately preserving  Wasserstein distances. In the context of, say, image retrieval (mainly $\MM=\n^2\subset \R^2$ and $p=1$), this could be viewed as obtaining representations of images using  a small number of "pixels."

Charikar~\cite{Cha02} and Indyk--Thaper~\cite{IT03} proved that if $\MM$ is a finite metric space, then $\mathsf{P}_1(\MM)$ embeds into $\ell_1$ with distortion $O(\log |\MM|)$. Hence, if the answer to Question~\ref{Q:wasserstein} were $k=n^{o(1)}$ for some $\alpha=O(1)$, then it would follow that any $n$-point subset of $\mathsf{P}_1(\MM)$ embeds into $\ell_1$ with distortion $o(\log n)$, i.e., better distortion than the general bound that is provided by Bourgain's embedding theorem (actually the $\ell_1$-variant of that theorem, which is also known to be sharp in general~\cite{LLR}). This shows that one cannot hope to answer  Question~\ref{Q:wasserstein} with $k=n^{o(1)}$ and $\alpha=O(1)$ without imposing geometric  restrictions on the underlying metric space $\MM$, since if $(\MM=\{x_1,\ldots,x_n\},d_\MM)$ is a metric space for which $\cc_1(\MM)\asymp \log n$, then we can take $\mu_1,\ldots,\mu_n$ to be the point masses $\d_{x_1},\ldots,\d_{x_n}$, so that $(\{\mu_1,\ldots,\mu_n\},\W_1)$ is isometric to $(\MM,d_\MM)$. The pertinent issue is therefore to study Question~\ref{Q:wasserstein} when the $\MM$ is "nice." For example, sufficiently good bounds here for $\MM=\R^2$ would be relevant to Question~\ref{Q:bourgain universality}, but at this juncture such a potential approach  to Question~\ref{Q:bourgain universality} is quite speculative.

The above "problematic" example relied inherently on the fact that the underlying metric space $\MM$ is itself far from being embeddable in $\ell_1$, but the difficulty persists even when $\MM=\ell_1$. Indeed, we recalled in Question~\ref{Q:bourgain universality} that Bourgain proved~~\cite{Bourgain-trees} that $\mathsf{P}_1(\ell_1)$ is universal, and hence the spaces of either Theorem~\ref{thm:coarse matousek}  or Theorem~\ref{thm:average distortion}  embed into $\mathsf{P}_1(\ell_1)$  with $O(1)$ distortion. So, for arbitrarily large $n\in \N$ we can find probability measures $\mu_1,\ldots,\mu_n$ on $\ell_1$ (actually on a sufficiently high dimensional Hamming cube $\{0,1\}^N$) such that $(\{\mu_1,\ldots,\mu_n\},\W_1)$ does not admit a good embedding into any normed space of dimension $n^{o(1)}$. This rules out an answer of $k=n^{o(1)}$  to Question~\ref{Q:wasserstein} (even for average distortion) for $(\mu_1,\ldots,\mu_n,\W_1)$, because in the setting of  Question~\ref{Q:wasserstein}, while  $\mathsf{P}_1(\mathcal{S})$ is not a normed space, it embeds isometrically into a normed space of dimension $|\mathcal{S}|-1$ (the dual of the mean-zero Lipschitz functions on $(\mathcal{S},d_\MM)$; see e.g.~\cite{NS07,Vil09} for an explanation of this standard fact). In the case of average distortion, one could see this using a different approach of Khot and the author~\cite{KN06} which constructs a collection of $n=e^{O(d)}$ probability measures on the Hamming cube $\{0,1\}^d$ that satisfy the conclusion of Theorem~\ref{thm:average distortion}, as explained in~\cite[Remark~5]{Nao17}. This shows that even though these $n$ probability measures reside on a Hamming cube of dimension $O(\log n)$, one cannot realize their Wasserstein-1 geometry  with $O(1)$ distortion (even on average) in any normed space of dimension $n^{o(1)}$, let alone in $\mathsf{P}_1(\mathcal{S})$ with $|\mathcal{S}|=n^{o(1)}$.

It is therefore natural to investigate Question~\ref{Q:wasserstein} when $\MM$ is low-dimensional. When $p=1$, this remains  an (important) uncharted terrain. When $p>1$ and $\MM=\R^3$, partial information on Question~\ref{Q:wasserstein} follows from~\cite{ANN15}. To see this, focus for concreteness on the case $p=2$. Fix $\alpha\ge 1$ and $n\in \N$. Suppose that $(\NN,d_\NN)$ is an $n$-point metric space for which the conclusion of Theorem~\ref{thm:coarse matousek}  holds true with $\omega(t)=\sqrt{t}$ and $\Omega(t)=2\alpha\sqrt{t}$. By~\cite{ANN15}, the metric space $(\NN,\sqrt{d_\NN})$ embeds with distortion $2$ into $\mathsf{P}_2(\R^3)$, where $\R^3$ is equipped with the standard Euclidean metric. Hence, if the image under this embedding    of $\NN$ in $\mathsf{P}_2(\R^3)$ embedded into some $k$-dimensional normed space with distortion $\alpha$, then by Theorem~\ref{thm:coarse matousek}  necessarily $k\ge n^{c/\alpha^2}$ for some universal constant $c$. This does not address Question~\ref{Q:wasserstein} as stated, because to the best of our knowledge it is not known whether $\mathsf{P}_2(\mathcal{S})$ embeds with $O(1)$ distortion into some "low-dimensional" normed space for every "small" $\mathcal{S}\subset \R^3$ (the relation between "small" and "low-dimensional"  remains to be studied). In the case of average distortion, repeat this argument with $\NN$ now being the metric space  of Theorem~\ref{thm:average distortion}.  By Remark~\ref{rem:for average snowflake} below, if the image in $\mathsf{P}_2(\R^3)$ of $(\NN,\sqrt{d_\NN})$ embedded with average distortion $\alpha$ into some $k$-dimensional normed space, then necessarily $k\ge\exp(\frac{c}{\alpha}\sqrt{\log n})$.

\section{Finite subsets of Hilbert space}\label{sec:JL}

The article~\cite{JL84} of Johnson and Lindenstrauss is devoted to proving a  theorem on the extension of Lipschitz functions from finite subsets of metric spaces.\footnote{Stating this theorem here would be an unnecessary digression, but we highly recommend examining the accessible geometric result of~\cite{JL84}; see~\cite{NR17} for  a review of the current state of the art on Lipschitz extension from finite subsets.} Over the ensuing decades, the classic~\cite{JL84} attained widespread prominence outside the rich literature on the {\em Lipschitz extension problem}, due to  two components of~\cite{JL84} that had major conceptual significance and influence, but are technically simpler than the  proof of its main theorem.

The first of these components is the JL lemma, which we already stated in the Introduction. Despite its wide acclaim and applicability, this result is commonly called a "lemma" rather than a "theorem" because within the context of~\cite{JL84} it was just that, i.e., a relatively simple step toward the proof of the main theorem of~\cite{JL84}.

The second of these components is a section of~\cite{JL84} that is devoted to formulating open problems in the context of the Ribe program; we already described a couple of the questions that were raised there, but it contains more questions that proved to be remarkably insightful and had major impact on subsequent research (see e.g.~\cite{Bal92,NPSS06}). Despite its importance,  the impact of~\cite{JL84} on the Ribe program will not be pursued further in the present article, but we will next proceed to study the JL lemma in detail (including some new observations).

Recalling Theorem~\ref{thm:JL alpha}, the JL lemma~\cite{JL84} asserts that for every integer $n\ge 2$ and (distortion/error tolerance) $\alpha\in (1,\infty)$, if  $x_1,\ldots,x_n$ are distinct vectors in a Hilbert space $(H,\|\cdot\|_H)$, then there exists (a target dimension) $k\in \n$  and a new $n$-tuple of $k$-dimensional vectors $y_1,\ldots,y_n\in \R^k$ such that
\begin{equation}\label{eq:JL conclusion1}
k\lesssim_\alpha \log n,
\end{equation}
and the assignment $x_i\mapsto y_i$, viewed as a mapping into $\ell_2^k$, has distortion at most $\alpha$, i.e.,
\begin{equation}\label{eq:JL conclusion2}
\forall\, i,j\in\n,\qquad\|x_i-x_j\|_H\le \|y_i-y_j\|_{\ell_2^k}\le \alpha\|x_i-x_j\|_H.
\end{equation}

It is instructive to  take note of the "compression" that this statement entails. By tracking the numerical value of the target dimension $k$ that the proof in Section~\ref{sec:JL opt} below yields (see Remark~\ref{rem:constants}), one concludes that given an arbitrary collection of, say, a billion vectors of length a billion (i.e., $1\,000\,000\,000$ elements of $\R^{1\,000\,000\,000}$), one can find a billion vectors of length $329$ (i.e., elements of $\R^{329}$), all of whose pairwise distances are within a factor  $2$ of the corresponding pairwise distances among the initial configuration of billion-dimensional
vectors. Furthermore, if one wishes to maintain the pairwise distances of those billion vectors within a somewhat larger constant factor, say, a factor of $10$ or $450$, then one could do so in dimension $37$ or $9$, respectively.

The logarithmic dependence on $n$ in~\eqref{eq:JL conclusion1} is optimal, up to the value of the implicit ($\alpha$-dependent) constant factor. This is so even when one considers the special case when $x_1,\ldots,x_n\in H$ are the vertices of the standard $(n-1)$-simplex, i.e., $\|x_i-x_j\|_H=1$ for all distinct $i,j\in \n$, and even when one allows the Euclidean norm in~\eqref{eq:JL conclusion2} to be  replaced by any norm $\|\cdot\|:\R^k\to [0,\infty)$, namely if instead of~\eqref{eq:JL conclusion2} we have $1\le \|y_i-y_j\|\le \alpha$ for all distinct $i,j\in \n$. Indeed, denote the unit ball of $\|\cdot\|$  by $B=\{z\in \R^k:\ \|z\|\le 1\}$ and let $\vol_k(\cdot)$ be the Lebesgue measure on $\R^k$. If $i,j\in \n$ are distinct, then by the triangle inequality the assumed lower bound $\|y_i-y_j\|\ge 1$ implies that the interiors of $y_i+\frac12 B$ and $y_j+\frac12 B$ are disjoint. Hence, if we denote $A=\bigcup_{i=1}^n(y_i+\frac12 B)$, then $\vol_k(A)=\sum_{i=1}^n \vol_k(y_i+\frac12 B)=\frac{n}{2^k}\vol_k(B)$. At the same time, for every $u,v\in A$ there are $i,j\in \n$ for which $u\in y_i+\frac12 B$ and $v\in y_j+\frac12 B$, so by another application of the triangle inequality we have $\|u-v\|\le \|y_i-y_j\|+1\le \alpha+1$. This implies that $A-A\subset (\alpha+1)B$. Hence,
\begin{equation*}\label{eq:use BM}
(\alpha+1)\sqrt[k]{\vol_k(B)}=\sqrt[k]{\vol_k((\alpha+1)B)}\ge \sqrt[k]{\vol_k(A-A)}\ge 2\sqrt[k]{\vol_k(A)} =\sqrt[k]{n\vol_k(B)},
\end{equation*}
where the penultimate step uses the Brunn--Minkowski inequality~\cite{Sch14}.  This simplifies to give
\begin{equation}\label{eq:k lower log general norm}
k\ge \frac{\log n}{\log(\alpha+1)}.
\end{equation}
By~\cite{Bra99,Dek00},  the vertices of $(n-1)$-simplex embed isometrically into any infinite dimensional Banach space, so we have thus justified the bound~\eqref{lower brunnminkowksi intro}, and hence in particular the first lower bound on $\kk_n^\alpha(\ell_2)$ in~\eqref{eq:JL in theporem}. As we already explained, the second lower bound (for the almost-isometric regime) on $\kk_n^\alpha(\ell_2)$ in~\eqref{eq:JL in theporem} is due to the very recent work~\cite{GN17}. The upper bound on $\kk_n^\alpha(\ell_2)$ in~\eqref{eq:JL in theporem}, namely that in~\eqref{eq:JL conclusion1} we can take
\begin{equation}\label{eq:JL const best known}
k\lesssim \frac{\log n}{\log\big(1+(\alpha-1)^2\big)}\asymp \max\left\{\frac{\log n}{(\alpha-1)^2},\frac{\log n}{\log \alpha}\right\},
\end{equation}
follows from the original proof of the JL lemma in~\cite{JL84}. A justification of~\eqref{eq:JL const best known} appears in Section~\ref{sec:JL opt} below.

\begin{question}[dimension reduction for the vertices of the simplex]Fix $\d\in (0,\frac12)$. What is the order of magnitude (up to universal constant factors) of the smallest $\mathfrak{S}(\d)\in (0,\infty)$ such that for every  $n\in \N$ there is $k\in \N$ with $k\le \mathfrak{S}(\d)\log n$ and $y_1,\ldots,y_{n}\in \R^k$ that satisfy $1\le \|y_i-y_j\|_{2}\le 1+\d$ for all distinct $i,j\in \n$? By~\eqref{eq:JL const best known} we have $\mathfrak{S}(\d)\lesssim 1/\d^{2}$. The best-known lower bound here is $\mathfrak{S}(\d)\gtrsim 1/(\d^2\log(1/\d))$, due to Alon~\cite{Alo03}.
\end{question}

\begin{remark} The upper bound~\eqref{eq:JL const best known} treats the target dimension in the JL lemma  for an {\em arbitrary} subset of a Hilbert space. The lower bound~\eqref{eq:k lower log general norm} was derived in the special case of the vertices of the regular simplex, but it is also more general as it is valid for embeddings of these vertices into an {\em arbitrary} $k$-dimensional norm. In this (both special, and more general) setting, the bound~\eqref{eq:k lower log general norm} is quite sharp for large $\alpha$. Indeed, by~\cite{ABV98} (see also~\cite[Corollary~2.4]{OR16}), for each $n\in \N$ and  $\alpha>\sqrt{2}$, if we write
$
k=\lceil(\log(4n))/\log(\alpha^2/(2\sqrt{\alpha^2-1}))\rceil,
$
then for {\em every} norm $\|\cdot\|$ on $\R^k$   there exist $y_1,\ldots,y_n\in \R^k$ satisfying $1\le \|y_i-y_j\|\le \alpha$ for distinct $i,j\in \n$. See~\cite[Theorem~4.3]{FL94} for an earlier result in this direction. See also~\cite{ABV98}  and the references therein (as well as~\cite[Problem~2.5]{OR16})  for partial results towards understanding the analogous issue (which is a longstanding open question) in the small distortion regime $\alpha\in (1,\sqrt{2}]$.
\end{remark}

\subsection{Optimality of re-scaled random projections}\label{sec:JL opt} To set the stage for the proof of the JL lemma, note that by translation-invariance we may assume without loss of generality that one of the vectors $\{x_i\}_{i=1}^n$ vanishes, and then by replacing the Hilbert space $H$ with the span of $\{x_i\}_{i=1}^n$, we may further assume that $H=\R^{n-1}$.
 %We shall also assume that $n\ge 4$, and fix $k\in \{1,\ldots,n-3\}$ that will be determined later.

Let $\proj_{\R^k}\in \M_{k\times (n-1)}(\R)$ be the $k$ by $n-1$ matrix of the orthogonal projection from $\R^{n-1}$ onto $\R^k$, i.e., $\proj_{\R^k} z=(z_1,\ldots,z_k)\in\R^k$ is the first $k$ coordinates of $z=(z_1,\ldots,z_{n-1})\in \R^{n-1}$. One could attempt to simply truncate the  vectors vectors $x_1,\ldots,x_n$ so as to obtain $k$-dimensional vectors, namely to consider the vectors  $\{y_i=\proj_{\R^k} x_i\}_{i=1}^n$ in~\eqref{eq:JL conclusion2}.  This naive (and heavy-handed)  way of forcing low-dimensionality can obviously fail miserably, e.g.~we could have $\proj_kx_i=0$ for all $i\in \n$. Such a simplistic idea performs poorly because it makes two arbitrary and unnatural choices, namely it does not take advantage of rotation-invariance and scale-invariance. To remedy this, let $\OO_{n-1}\subset \M_{n-1}(\R)$ denote the group of $n-1$ by $n-1$ orthogonal matrices, and fix (a scaling factor) $\sigma\in (0,\infty)$. Let $\OO\in \OO_{n-1}$ be a random orthogonal matrix distributed according to the Haar probability measure on $\OO_{n-1}$. In~\cite{JL84} it was shown that if $k$ is sufficiently large (yet satisfying~\eqref{eq:JL conclusion1}),  then for an appropriate $\sigma>0$ with positive probability~\eqref{eq:JL conclusion2} holds  for the following  random vectors.
\begin{equation}\label{eq:choice of yi}
\{y_i =\sigma \proj_{\R^k}\OO x_i\}_{i=1}^n\subset \R^k.
\end{equation}

We will do more than merely explain why the randomly projected vectors in~\eqref{eq:choice of yi} satisfy the desired conclusion~\eqref{eq:JL conclusion2} of the JL lemma with positive probability. We shall  next demonstrate that  such a procedure is the {\em best possible} (in  a certain sense that will be made precise) among all the possible choices of random assignments of $x_1,\ldots,x_n$ to $y_1,\ldots,y_n$ via multiplication by a random matrix in $M_{k\times(n-1)}(\R)$, provided that we optimize so as to use the best scaling factor $\sigma\in (0,\infty)$ in~\eqref{eq:choice of yi}.

Let $\mu$ be any Borel probability measure on $\M_{k\times (n-1)}(\R)$, i.e., $\mu$ represents an arbitrary (reasonably measurable) distribution over $k\times (n-1)$ random matrices $\mA\in \M_{k\times (n-1)}(\R)$. For $\alpha\in (1,\infty)$ define
\begin{equation}\label{eqLdef frak p mu}
 \mathfrak{p}_\mu^\alpha\eqdef \inf_{z\in \mathbf{S}^{n-2}} \mu\Big[\big\{\mA\in \M_{k\times (n-1)}(\R):\  1\le \|\mA z\|_{\ell_2^k}\le \alpha\big\} \Big],
\end{equation}
where $\mathbf{S}^{n-2}=\{z\in \R^{n-1}:\ \|z\|_{\ell_2^{n-1}}=1\}$ denotes the unit Euclidean sphere in $\R^{n-1}$. Then
\begin{align}\label{eq:uninion JL}
\nonumber \mu \bigg[&\bigcap_{i,j\in \n}\Big\{\mA\in \M_{k\times (n-1)}(\R):\|x_i-x_j\|_{\ell_2^{n-1}}\le \|\mA x_i-\mA x_j\|_{\ell_2^k}\le \alpha\|x_i-x_j\|_{\ell_2^{n-1}}\Big\}\bigg]\\ \nonumber
&= 1- \mu\bigg[\bigcup_{i=1}^n\bigcup_{j=i+1}^n \bigg(\M_{k\times (n-1)}(\R)\setminus \Big\{\mA\in \M_{k\times (n-1)}(\R):1 \le \Big\|\mA \frac{x_i- x_j}{\|x_i-x_j\|_{\ell_2^{n-1}}}\Big\|_{\ell_2^k}\le \alpha\Big\}\bigg)\bigg]\\ \nonumber
&\ge 1- \sum_{i=1}^n\sum_{j=i+1}^n \bigg(1-\mu \bigg[\Big\{\mA\in \M_{k\times (n-1)}(\R):1 \le \Big\|\mA \frac{x_i- x_j}{\|x_i-x_j\|_{\ell_2^{n-1}}}\Big\|_{\ell_2^k}\le \alpha\Big\}\bigg]\bigg)\\
&\ge 1- \binom{n}{2}\big(1-\mathfrak{p}_\mu^\alpha\big).
\end{align}
Hence, the random vectors $\{y_i=\mA x_i\}_{i=1}^n$ will satisfy~\eqref{eq:JL conclusion2} with positive probability if $\mathfrak{p}_\mu^\alpha>1-\frac{2}{n(n-1)}$.

In order to  succeed to embed the largest possible number of vectors into $\R^k$ via the above randomized procedure while using the estimate~\eqref{eq:uninion JL}, it is in our best interest to work with a probability measure $\mu$ on $\M_{k\times (n-1)}(\R)$ for which $\mathfrak{p}_\mu^\alpha$ is as large as possible. To this end, define
\begin{equation}\label{eq:def pnk}
\mathfrak{p}^\alpha_{n,k}\eqdef \sup \Big\{\mathfrak{p}_\mu^\alpha:\ \mu\ \mathrm{is\ a \ Borel\ probability\ measure\ on\ } \M_{k\times (n-1)}(\R)\Big\}.
\end{equation}
Then, the conclusion~\eqref{eq:JL conclusion2} of the JL lemma will be valid provided $k\in \n$ satisfies
\begin{equation}\label{eq:JL condition random}
\mathfrak{p}^\alpha_{n,k}>1-\frac{2}{n(n-1)}.
\end{equation}
The following proposition asserts that the supremum in the definition~\eqref{eq:def pnk} of $\mathfrak{p}_{n,k}^\alpha$ is attained at a distribution over random matrices that has the aforementioned structure~\eqref{eq:choice of yi}.

\begin{proposition}[multiples of random orthogonal projections are JL-optimal]\label{prop:sigma nk} Fix $\alpha\in (1,\infty)$, an integer $n\ge 4$ and $k\in \{1,\ldots,n-3\}$. Let $\mu=\mu_{n,k}^\alpha$ be   the probability distribution on $\M_{k\times (n-1)}(\R)$ of the random matrix
\begin{equation}\label{eq:optimal rescaled projection}
\sqrt{\frac{\alpha^{\frac{2n-6}{n-k-3}}-1}{\alpha^{\frac{2k}{n-k-3}}-1}}\cdot \proj_{\R^k}\OO,
\end{equation}
that is obtained by choosing $\OO\in \OO_{n-1}$ according to the normalized Haar measure on $\OO_{n-1}$. Then $\mathfrak{p}_\mu^\alpha=\mathfrak{p}_{n,k}^\alpha$.
\end{proposition}

Obviously~\eqref{eq:optimal rescaled projection} is {\em not} a multiple of a uniformly random rank $k$ orthogonal projection $\proj:\R^{n-1}\to \R^{n-1}$ (chosen according to the normalized Haar measure on the appropriate Grassmannian). To obtain such a distribution, one should multiply the matrix in~\eqref{eq:optimal rescaled projection} on the left by $\OO^*$. That additional rotation does not influence the Euclidean length of the image, and hence it does not affect the quantity~\eqref{eqLdef frak p mu}. For this reason and for simplicity of notation, we prefer to work with~\eqref{eq:optimal rescaled projection} rather than random projections as was done in~\cite{JL84}.

\begin{proof}[Proof of Proposition~\ref{prop:sigma nk}] Given $\mA\in \M_{k\times (n-1)}(\R)$, denote its singular values by $\sss_1(\mA)\ge\ldots\ge \sss_k(\mA)$, i.e., they are  the eigenvalues (with multiplicity) of the symmetric matrix $\sqrt{\mA\mA^*}\in \M_k(\R)$. Then,
\begin{equation}\label{eq:haar identity}
\mathfrak{H}^{\OO_{n-1}} \Big[\big\{\OO\in \OO_{n-1}:\ 1\le \|\mA\OO z\|_{\ell_2^k}\le \alpha\big\}\Big]= \int_{\mathbf{S}^{k-1}}\psi_{n,k}^\alpha\bigg(\Big(\sum_{i=1}^k\sss_i(\mA)^2 \omega_i^2\Big)^{\frac12}\bigg)\ud \mathfrak{H}^{\mathbf{S}^{k-1}}(\omega),
\end{equation}
where $\mathfrak{H}^{\OO_{n-1}}$ and $\mathfrak{H}^{\mathbf{S}^{k-1}}$ are the Haar probability measures on the orthogonal group $\OO_{n-1}$ and the unit Euclidean sphere $\mathbf{S}^{k-1}$, respectively, and the function $\psi_{n,k}^\alpha:[0,\infty)\to \R$ is defined by
\begin{equation}\label{eq:def psi nk modified}
\forall\, \sigma\in [0,\infty),\qquad \psi_{n,k}^\alpha(\sigma)\eqdef
\frac{2\pi^{\frac{k}{2}}}{\Gamma\big(\frac{k}{2}\big)}\int^{\max\left\{1,\sigma\right\}}_{\max\left\{1,\frac{\sigma}{\alpha}\right\}}\frac{(s^2-1)^{\frac{n-k-3}{2}}}{s^{n-2}}\ud s.
\end{equation}
To verify the identity~\eqref{eq:haar identity}, consider  the  singular value decomposition
\begin{equation}\label{eq:write SVD}\mA=
\mathsf{U}\begin{pmatrix} \sss_1(\mA) & 0 & \dots& \dots&0 \\
  0 & \sss_2(\mA)& \ddots& \ddots & \vdots\\
  \vdots & \ddots & \ddots  & \ddots & \vdots\\
            \vdots & \ddots & \ddots& \ddots &0\\
              0 & \dots & \dots &0&\sss_k(\mA)
                       \end{pmatrix}\proj_{\R^k}\mathsf{V},\end{equation}
where $\mathsf{U}\in \OO_k$ and $\mathsf{V}\in \OO_{n-1}$. If $\OO\in \OO_{n-1}$ is distributed according to $\mathfrak{H}^{\OO_{n-1}}$, then by the left-invariance of $\mathfrak{H}^{\OO_{n-1}}$ we know that $\mathsf{V}\OO$ is distributed according to $\mathfrak{H}^{\OO_{n-1}}$. By rotation-invariance and uniqueness of Haar measure on $\mathbf{S}^{n-2}$ (e.g.~\cite{MS}), it follows that for every $z\in \mathbf{S}^{n-1}$ the random vector $\mathsf{V}\OO z$ is distributed according to the normalized Haar measure on $\mathbf{S}^{n-2}$. So, $\proj_{\R^k} \mathsf{V}\OO z$ is distributed on the Euclidean unit ball of $\R^k$, with density
\begin{equation}\label{eq:projected density}
\big(u\in \R^k\big)\mapsto \frac{\Gamma\big(\frac{n-1}{2}\big)}{\pi^{\frac{k}{2}}\Gamma\big(\frac{n-1-k}{2}\big)} \left(1-\|u\|_{\ell_2^k}^2\right)^{\frac{n-k-3}{2}}\1_{\big\{\|u\|_{\ell_2^k}\le 1\big\}}.
\end{equation}
See~\cite{Sta82} for a proof  of this distributional identity (or~\cite[Corollary~4]{BGMN05}  for a more general derivation); in codimension $2$, namely $k=n-3$, this is a higher-dimensional analogue of Archimedes' theorem  that the projection to $\R$ of the uniform surface area measure on the unit Euclidean sphere in $\R^{3}$ is the Lebesgue measure on $[-1,1]$. Recalling~\eqref{eq:write SVD}, it follows from this discussion that the Euclidean norm of $\mA\OO z$ has the same distribution as $(\sum_{i=1}^k \sss_i(\mA)^2u_i^2)^{1/2}$, where $u=(u_1,\ldots,u_k)\in \R^k$ is distributed according to the density~\eqref{eq:projected density}. The identity~\eqref{eq:haar identity} now follows by integration in polar coordinates $(\omega,r)\in \mathbf{S}^{k-1}\times [0,\infty)$, followed by the change of variable $s=1/r$.

Next,  $\psi_{n,k}^\alpha$ vanishes on $[0,1]$, increases on $[1,\alpha]$, and is smooth on $[\alpha,\infty)$. The integrand in~\eqref{eq:def psi nk modified} is at most $s^{-k-1}$, so $\lim_{\sigma\to \infty}\psi_{n,k}^\alpha(\sigma)=0$. By directly differentiating~\eqref{eq:def psi nk modified} and simplifying the resulting expression, one sees that if $\sigma\in [\alpha,\infty)$, then $(\psi_{n,k}^\alpha)'(\sigma)=0$ if and only if $\sigma=\sigma_{\max}(n,k,\alpha)$, where
\begin{equation}\label{eq:def sigma max}
\sigma_{\max}(n,k,\alpha)\eqdef \sqrt{\frac{\alpha^{\frac{2n-6}{n-k-3}}-1}{\alpha^{\frac{2k}{n-k-3}}-1}}.
\end{equation}
Therefore, the global maximum of $\psi_{n,k}^\alpha$ is attained at $\sigma_{\max}(n,k,\alpha)$, and by~\eqref{eq:haar identity} we have
\begin{equation}\label{eq:use identity agaoin for maximizer}
\forall\, \mA\in \M_{k\times (n-1)}(\R),\qquad \mathfrak{H}^{\OO_{n-1}} \Big[\big\{\OO\in \OO_{n-1}:\ 1\le \|\mA\OO z\|_{\ell_2^k}\le \alpha\big\}\Big]\le \psi_{n,k}^\alpha\big(\sigma_{\max}(n,k,\alpha)\big)=\mathfrak{p}_\mu^\alpha.
\end{equation}
The final step of~\eqref{eq:use identity agaoin for maximizer} is another application~\eqref{eq:haar identity}, this time in the special case $\mA=\sigma_{\max}(n,k,\alpha)\proj_{\R^k}$, while recalling~\eqref{eqLdef frak p mu} and~\eqref{eq:def sigma max}, and that $\mu$ is the distribution of the random matrix appearing in~\eqref{eq:optimal rescaled projection}.

To conclude the proof of Proposition~\eqref{prop:sigma nk}, take any Borel probability measure $\nu$ on $\M_{k\times (n-1)}(\R)$ and integrate~\eqref{eq:use identity agaoin for maximizer} while using Fubini's theorem to obtain the estimate
\begin{multline*}
\mathfrak{p}_\mu^\alpha\ge\int_{\M_{k\times (n-1)}(\R)} \mathfrak{H}^{\OO_{n-1}}\Big[\big\{\OO\in \OO_{n-1}:\ 1\le \|\mA\OO z_0\|_{\ell_2^k}\le \alpha\big\}\Big]\ud\nu(\mA)\\ =
\int_{\OO_{n-1}} \nu\Big[\big\{\mA\in \M_{k\times (n-1)}(\R):\ 1\le \|\mA\OO z_0\|_{\ell_2^k}\le \alpha\big\}\Big] \ud\mathfrak{H}^{\OO_{n-1}}(\OO)\stackrel{\eqref{eqLdef frak p mu}}{\ge} \int_{\OO_{n-1}} \mathfrak{p}_\nu^\alpha \ud\mathfrak{H}^{\OO_{n-1}}(\OO)=\mathfrak{p}_\nu^\alpha.
\end{multline*}
So, the maximum of $\mathfrak{p}_\nu^\alpha$ over the Borel probability  measures $\nu$ on $\M_{k\times (n-1)}(\R)$ is attained at $\mu$.\end{proof}

\begin{remark} Recalling~\eqref{eq:JL condition random}, due to~\eqref{eq:use identity agaoin for maximizer} the conclusion~\eqref{eq:JL conclusion2} of the JL lemma holds if $k$ satisfies
\begin{equation}\label{eq:beta JL}
\frac{2\pi^{\frac{k}{2}}}{\Gamma\big(\frac{k}{2}\big)}\int^{\sigma_{\max}(n,k,\alpha)}_{\frac{1}{\alpha}
\sigma_{\max}(n,k,\alpha)}\frac{(s^2-1)^{\frac{n-k-3}{2}}}{s^{n-2}}\ud s>1-\frac{2}{n(n-1)},
\end{equation}
where $\sigma_{\max}(n,k,\alpha)$ is given in~\eqref{eq:def sigma max}. This is the best-known bound on $k$ in the JL lemma, which, due to Proposition~\ref{prop:sigma nk}, is the best-possible bound that is obtainable through the reasoning~\eqref{eq:uninion JL}. In particular, the asymptotic estimate~\eqref{eq:JL const best known} follows from~\eqref{eq:beta JL} via straightforward elementary calculus.
\end{remark}

\begin{remark} The JL lemma was reproved many times; see~\cite{FM88,Gor88,IM99,AV99,DG03,Ach03,KM05,IN07,Mat08,AC09,DG09,KW11,AL13,DG14,KN14,BDN15,Dir16}, though we make no claim that this is a comprehensive list of references. There were several motivations for these further investigations, ranging from the desire to obtain an overall  better understanding of the JL phenomenon, to obtain better bounds, and to obtain distributions on random matrices $\mA$ as in~\eqref{eq:uninion JL} with certain additional properties that are favorable from the computational perspective, such as ease of simulation, use of fewer random bits, sparsity, and the ability to evaluate the mapping $(z\in \R^{n-1})\mapsto \mA z$ quickly (akin to the fast Fourier transform). This  body of work represents ongoing efforts by computer scientists and applied mathematicians to further develop improved "JL transforms," driven by their usefulness as a tool for data-compression. We will not survey these ideas here, partially because we established that using random projections yields the best-possible bound on the target dimension $k$ (moreover,  this procedure is natural and simple). We speculate that working with the Haar measure on the orthogonal group $\OO_{n-1}$ as in~\eqref{eq:choice of yi} could have benefits that address the above computational issues, but leave this as an interesting open-ended direction for further research. A specific conjecture towards this goal appears in~\cite[page~320]{AC09}, and we suspect that the more recent work~\cite{BG12} on the spectral gap of Hecke operators of orthogonal Cayley graphs should be relevant in this context as well (e.g. for derandomization and fast implementation of~\eqref{eq:choice of yi}; see~\cite{BHH16,KM15} for steps in this direction).
\end{remark}

\begin{remark}\label{rem:constants}  In the  literature there is often a preference to use random matrices with independent entries in the context of the JL lemma, partially because they are simple to generate, though see the works~\cite{Ste80,Gen00,Mez07} on generating elements of the orthogonal group $\OO_{n-1}$ that are distributed according to its Haar measure. In particular, the best bound on $k$ in~\eqref{eq:JL conclusion1} that was previously available in the literature~\cite{DG03} arose from applying~\eqref{eq:uninion JL} when $\mA$ is replaced by the random matrix $\sigma\mathsf{G}$, where $\sigma=1/\sqrt{k}$ and the entries of $\mathsf{G}=(\mathsf{g}_{ij})\in M_{k\times (n-1)}(\R)$ are independent standard Gaussian random variables. We can, however, optimize over the scaling factor $\sigma$ in this setting as well, in analogy to the above optimization over the scaling factor in~\eqref{eq:choice of yi}, despite the fact that we know that working with the Gaussian matrix $\mathsf{G}$ is inferior to using a random rotation. A short calculation reveals that the optimal scaling factor is now $\sqrt{(\alpha^2-1)/(2k\log \alpha)}$, i.e., the best possible re-scaled Gaussian matrix for the purpose of reasoning as in~\eqref{eq:uninion JL} is not $\frac{1}{\sqrt{k}}\mathsf{G}$ but rather the random matrix
\begin{equation}\label{eq:rescaled gaussian}
\mathsf{G}_k^\alpha\eqdef \sqrt{\frac{\alpha^2-1}{2k\log \alpha}}\cdot\mathsf{G}.
\end{equation}
For this optimal multiple of a Gaussian matrix, one computes that for every $z\in \mathbf{S}^{n-2}$ we have
\begin{multline}\label{eq:gaussian optimizer identity and estimate}
1-\Pr \Big[1\le \|\mathsf{G}_k^\alpha z\|_{\ell_2^k}\le \alpha\Big]=\frac{2k^{\frac{k}{2}}}{\Gamma\big(\frac{k}{2}\big)}\int_{\log \alpha}^\infty \left(\frac{\beta}{e^{2\beta}-1}\right)^{\frac{k}{2}}\exp\bigg(-\frac{k\beta}{e^{2\beta}-1}\bigg)\ud \beta\\ <
\frac{4k^{\frac{k}{2}-1}}{\Gamma\big(\frac{k}{2}\big)}\left(\frac{\alpha^2-1}{\log \alpha}\alpha^{\frac{2}{\alpha^2-1}}\right)^{-\frac{k}{2}}\frac{(\alpha^2-1)^2\log\alpha }{2\alpha^4\log\alpha+2\alpha^2-\alpha^4-4\alpha^2(\log\alpha)^2-2\log\alpha-1}.
\end{multline}
The first step in~\eqref{eq:gaussian optimizer identity and estimate} follows from a straightforward computation using the fact that the squared Euclidean length of $\mathsf{G}_k^\alpha z$ is distributed according to a multiple of the $\chi^2$ distribution with $k$ degrees of freedom (see e.g.~\cite{Dur10}), i.e., one can write the leftmost term of~\eqref{eq:gaussian optimizer identity and estimate} explicitly as a definite integral, and then check that it indeed equals the middle term of~\eqref{eq:gaussian optimizer identity and estimate}, e.g., by verifying the the derivatives with respect to $\alpha$ of both expressions coincide. The final estimate in~\eqref{eq:gaussian optimizer identity and estimate} can be justified via a modicum of straightforward calculus. We deduce from this that the conclusion~\eqref{eq:JL conclusion2} of the JL lemma is holds with positive probability if for each $i\in \n$ we take $y_i$ to be the image of $x_i$ under the re-scaled Gaussian matrix in~\eqref{eq:rescaled gaussian}, provided that $k$ is sufficiently large so as to ensure that
\begin{equation}\label{eq:k condition improved dg}
\frac{\Gamma\big(\frac{k}{2}\big)}{k^{\frac{k}{2}-1}}\left(\frac{\alpha^2-1}{\log \alpha}\alpha^{\frac{2}{\alpha^2-1}}\right)^{\frac{k}{2}}\ge \frac{2n^2(\alpha^2-1)^2\log\alpha }{2\alpha^4\log\alpha+2\alpha^2-\alpha^4-4\alpha^2(\log\alpha)^2-2\log\alpha-1}.
\end{equation}
The values that we stated for the target dimension $k$ in the JL lemma with a billion vectors were obtained by using~\eqref{eq:k condition improved dg}, though even better bounds arise from an evaluation of the integral in~\eqref{eq:gaussian optimizer identity and estimate} numerically, which is what we recommend to do for particular settings of the parameters. As $\alpha\to 1$, the above bounds improve over those of~\cite{DG03} only in the second-order terms. For larger $\alpha$ these bounds yield  substantial improvements that might matter in practice, e.g.~for embedding a billion vectors with distortion $2$, the target dimension that is required using the best-available estimate in the literature~\cite{DG03} is $k=768$, while~\eqref{eq:k condition improved dg} shows that $k=329$ suffices.
\end{remark}

\section{Infinite subsets of Hilbert space}\label{sec:doubling}

The JL lemma provides a  quite complete understanding of the metric dimension reduction problem for finite subsets of Hilbert space. For infinite subsets of Hilbert space, the research splits into two strands. The first is to understand those subsets $\sub\subset \R^n$ for which certain random matrices in $\M_{k\times n}(\R)$ (e.g.~random projections, or matrices whose entries are i.i.d. independent sub-Gaussian random variables) yield with positive probability an embedding of $\sub$ into $\R^k$ of a certain pre-specified distortion; this was pursued in~\cite{Gor88,KM05,IN07,MPT07,MT08,BDN15,Dir16,PDG17}, yielding a satisfactory answer which  relies on  multi-scale  chaining criteria~\cite{Tal14,Nel16} .

The second (and older) research strand focuses on the mere {\em existence} of a low-dimensional embedding rather than on the success of the specific embedding approach of (all the known proofs of) the JL lemma. Specifically, given a subset $\sub$ of a Hilbert space and $\alpha\in [1,\infty)$, could one understand when does $\sub$ admit an embedding with distortion $\alpha$ into $\ell_2^k$ for some $k\in \N$? If one ignores the dependence on the distortion $\alpha$, then this qualitative question coincides with Problem~\ref{Q:bilip Rk} (the bi-Lipschitz embedding problem into $\R^k$), since if a metric space $(\MM,d_\MM)$ satisfies $\inf_{k\in \N}\cc_{\R^k}(\MM)<\infty$, then in particular it admits a bi-Lipschitz embedding into a Hilbert space.

We shall next describe an obvious necessary condition for bi-Lipschitz embeddability into $\R^k$ for some $k\in\N$. In what follows, all balls in a metric space $(\MM,d_\MM)$ will be closed balls, i.e., for $x\in \MM$ and $r\in [0,\infty)$ we write $B_\MM(x,r)=\{y\in \MM:\ d_\MM(x,y)\le r\}$. Given $K\in [2,\infty)$, a metric space   $(\MM,d_\MM)$ is said to be $K$-{\em doubling} (e.g.~\cite{Bou28,CW71}) if every ball in $\MM$ (centered anywhere in $\MM$ and of any radius) can be covered by at most $K$ balls of half its radius, i.e.,  for every $x\in \MM$ and $r\in [0,\infty)$ there is $m\in \N$ with $m\le K$ and $y_1,\ldots,y_m\in \MM$ such that $B_\MM(x,r)\subset B_\MM(y_1,\frac12 r)\cup\ldots\cup B_\MM(y_m,\frac12 r)$. A metric space is doubling if it is $K$-doubling for some $K\in [2,\infty)$.

Fix $k\in \N$ and $\alpha\ge 1$. If a metric space  $(\MM,d_\MM)$ embeds with distortion $\alpha$ into a normed space $(\R^k,\|\cdot\|)$, then $\MM$ is $(4\alpha+1)^k$-doubling.  Indeed, fix $x\in \MM$ and $r>0$. Let $\{z_1,\ldots,z_n\}\subset B_\MM(x,r)$ be a maximal subset (with respect to inclusion)  of  $B_\MM(x,r)$ satisfying $d_\MM(z_i,z_j)>\frac12 r$ for  distinct $i,j\in \n$. The maximality of $\{z_1,\ldots,z_n\}$  ensures that for any $w\in B_\MM(x,r)\setminus \{z_1,\ldots,z_n\}$ we have $\min_{i\in \n} d_\MM(w,z_i)\le \frac12 r$, i.e.,  $B_\MM(x,r)\subset B_\MM(z_1,\frac12 r)\cup\ldots\cup B_\MM(z_n,\frac12 r)$.  We are assuming that there is an embedding $f:\MM\to \R^k$ that satisfies $d_\MM(u,v)\le \|f(u)-f(v)\|\le \alpha d_\MM(u,v)$ for all $u,v\in \MM$.  So, for distinct $i,j\in \n$ we have $\frac{r}{2}<d_\MM(z_i,z_j)\le \|f(z_i)-f(z_j)\|\le \alpha d_\MM(z_i,z_j)\le \alpha (d_\MM(z_i,x)+ d_\MM(x,z_j))\le2\alpha r$.
The reasoning that led to~\eqref{eq:k lower log general norm} with $y_1=\frac{2}{r}f(z_1),\ldots,y_n=\frac{2}{r}f(z_n)$ and $\alpha$ replaced by $4\alpha$ gives $k\ge (\log n)/\log(4\alpha+1)$, i.e., $n\le (4\alpha+1)^k$.

\begin{remark}\label{rem:doub log n} In Section~\ref{sec:metric dim reduction intro} we recalled that in the context of the Ribe program $\log |\MM|$ was the initial (in hindsight somewhat naive, though still very useful) replacement for the "dimension" of a finite metric space $\MM$. This arises naturally also from the above discussion. Indeed, $\MM$ is trivially $|\MM|$-doubling (simply cover each ball in $\MM$ by singletons), and this is the best bound that one could give on the doubling constant of $\MM$ in terms of $|\MM|$. So, from the perspective of the doubling property, the natural restriction on $k\in \N$ for which there exists an embedding of $\MM$ into some $k$-dimensional normed space with $O(1)$ distortion  is that $k\gtrsim \log |\MM|$.
\end{remark}

Using terminology that was recalled in Remark~\ref{rem:snowflake}, the definition of the doubling property directly implies that for every $\theta\in (0,1)$ a metric space $\MM$ is doubling  if and only if its $\theta$-snowflake $\MM^\theta$ is doubling. With this in mind, Theorem~\ref{thm:assouad} below is a very important classical achievement of Assouad~\cite{Ass83}.

\begin{theorem}\label{thm:assouad}
The following assertions are equivalent for every metric space $(\MM,d_\MM)$.
\begin{itemize}
\item $\MM$ is doubling.
\item For every $\theta\in (0,1)$ there is $k\in \N$ such that $\MM^\theta$ admits a bi-Lipschitz embedding into $\R^k$.
\item Some snowflake of $\MM$ admits a bi-Lipschitz embedding into $\R^k$ for some $k\in \N$.
\end{itemize}
\end{theorem}

Theorem~\ref{thm:assouad}  is a qualitative statement, but its proof in~\cite{Ass83} shows that for every $K\in [2,\infty)$ and $\theta\in (0,1)$, there are $\alpha(K,\theta)\in [1,\infty)$ and $k(K,\theta)\in \N$ such that if $\MM$ is $K$-doubling, then $\MM^\theta$ embeds into $\R^{k(K,\theta)}$ with distortion $\alpha(K,\theta)$; the argument of~\cite{Ass83} inherently gives that as $\theta\to 1$, i.e., as the $\theta$-snowflake $\MM^\theta$ approaches the initial metric space $\MM$, we have  $\alpha(K,\theta)\to \infty$ and $k(K,\theta)\to \infty$. A meaningful study of the best-possible asymptotic behavior  of the distortion $\alpha(K,\theta)$ here would require specifying which norm on $\R^k$ is being considered. Characterizing the quantitative dependence in terms of geometric properties of the target norm on $\R^k$  has not been carried out yet (it isn't even clear what should the pertinent geometric properties be), though see~\cite{HM06} for an almost isometric version when one considers the $\ell_\infty$ norm on $\R^k$ (with the dimension $k$ tending to $\infty$ as the distortion approaches $1$); see also~\cite{GK15} for a further partial step in this direction. In~\cite{NN12} it was shown that for $\theta\in [\frac12,1)$ one could take $k(K,\theta)\le k(K)$ to be bounded by a constant that depends only on $K$; the proof of this fact in~\cite{NN12} relies on a probabilistic construction, but in~\cite{DS13} a clever and instructive deterministic proof of this phenomenon was found (though, yielding asymptotically worse estimates on $\alpha(K,\theta),k(K)$ than those of~\cite{NN12}).

Assouad's theorem  is a satisfactory characterization of the doubling property in terms of embeddability into finite-dimensional Euclidean space.  However, it is a "near miss" as an answer to Problem~\ref{Q:bilip Rk}: the same statement with $\theta=1$ would have been a wonderful resolution of the bi-Lipschitz embedding problem into $\R^k$, showing that a simple intrinsic ball covering property is equivalent to bi-Lipschitz embeddability into some $\R^k$. It is important to note that while the snowflaking procedure does in some sense "tend to" the initial metric space as $\theta\to 1$, for $\theta<1$ it deforms  the initial metric space substantially (e.g.~such a $\theta$-snowflake  does not contain any non-constant rectifiable curve). So, while Assouad's theorem with the stated snowflaking is useful (examples of nice applications appear in~\cite{BS00,HM06}), its failure to address the bi-Lipschitz category is a major drawback.

Alas, more than a decade after the publication of Assouad's theorem, it was shown in~\cite{Sem96} (relying a on a rigidity theorem of~\cite{Pan89}) that Assouad's theorem does not hold with $\theta=1$, namely there exists a doubling metric space that does not admit a bi-Lipschitz embedding into $\R^k$ for any $k\in \N$. From the qualitative perspective, we now know that the case $\theta=1$ of Assouad's theorem fails badly in the sense that there exists a doubling metric space (the continuous $3$-dimensional Heisenberg group, equipped with the Carnot--Carath\'eodory metric) that does not admit a bi-Lipschitz embedding into any Banach space with the Radon--Nikod\'ym property~\cite{LN06,CK06} (in particular, it does not admit a bi-Lipschitz embedding into any reflexive or separable dual Banach space, let alone a finite dimensional Banach space), into any $L_1(\mu)$ space~\cite{CK10}, or into any Alexandrov space of curvature bounded above or below~\cite{Pau01} (a further strengthening appears in the forthcoming work~\cite{AN17}). From the quantitative perspective, by now we know  that balls in the discrete $5$-dimensional Heisenberg group equipped with the word metric (which is doubling) have the asymptotically worst-possible bi-Lipschitz distortion (as a function of their cardinality) in  uniformly convex Banach spaces~\cite{LN14-poincare} (see also~\cite{ANT13}) and $L_1(\mu)$ spaces~\cite{NY17,NY17-versus}; interestingly, the latter assertion is not true for the $3$-dimensional Heisenberg group~\cite{NY17-3dim}, while the former assertion does hold true for the $3$-dimensional Heisenberg group~\cite{LN14-poincare}.

All of the known "bad examples" (including, in addition to the Heisenberg group, those that were subsequently found in~\cite{Laakso,Laa02,BP99,Che99}) which show that the doubling property is not the sought-after answer to Problem~\ref{Q:bilip Rk} do not even embed into an infinite-dimensional Hilbert space. This leads to the following natural and intriguing question that was stated by Lang and Plaut in~\cite{LP01}.

\begin{question}\label{Q:LP}
Does every doubling subset of a Hilbert admit a bi-Lipschitz embedding into $\R^k$ for some $k\in \N$?
\end{question}
As stated, Question~\ref{Q:LP} is qualitative, but by a compactness argument (see~\cite[Section~4]{NN12})  if its answer were positive, then for every $K\in [2,\infty)$ there would exist $d_K\in \N$ and $\alpha_K\in [1,\infty)$ such that any $K$-doubling subset of a Hilbert space would embed into $\ell_2^{d_K}$ with distortion $\alpha_K$. If Question~\ref{Q:LP} had a positive answer, then it would be very interesting to determine the asymptotic behavior of $d_K$ and $\alpha_K$ as $K\to \infty$. A positive answer to Question~\ref{Q:LP}  would be a solution of Problem~\ref{Q:bilip Rk}, though the intrinsic criterion that it would provide would be quite complicated, namely it would say that a metric space $(\MM,d_\MM)$   admits a bi-Lipschitz embedding into $\R^k$ for some $k\in \N$ if and only if it is doubling and satisfies the family of quadratic distance inequalities~\eqref{eq:K LLR}. More importantly, it seems that any positive answer to Question~\ref{Q:LP}  would devise a procedure that starts with a subset in a very high-dimensional Euclidean space and, if that subset is $O(1)$-doubling, produce a bi-Lipschitz embedding into $\R^{O(1)}$; such a procedure, if possible, would be a quintessential metric dimension reduction result that is bound to be of major importance. It should be noted that, as proved in~\cite[Remark~4.1]{IN07},  any such general procedure cannot be an embedding into low-dimensions via a linear mapping as in the JL lemma, i.e., Question~\ref{Q:LP} calls for a genuinely nonlinear dimension reduction technique.\footnote{On its own, the established necessity  of obtaining a genuinely   nonlinear embedding method into low dimensions should not discourage attempts to answer Question~\ref{Q:LP}, because  some rigorous nonlinear dimension reduction methods have been devised in the literature; see e.g.~\cite{Ass83,Sem99,CS02,GKL03,BLMN04,BM04,KLMN05,LS05,BKL07,LNP09,CGT10,ABN11,BRS11,GT11,NN12,DS13,LdMM13,NR13,GK15,BG16,Nei16,OR16,ANN17}. However, all of these approaches  seem far from addressing Question~\ref{Q:LP}.}

Despite the above reasons why a positive answer to Question~\ref{Q:LP} would be very worthwhile, we suspect that Question~\ref{Q:LP} has a negative answer. A specific doubling subset of a Hilbert space which is a potential counterexample to Question~\ref{Q:LP}  was constructed  in~\cite[Question~3]{NN12}, but to date it remains unknown whether or not this subset admits a bi-Lipschitz embedding into $\R^{O(1)}$. If the answer to Question~\ref{Q:LP} is indeed negative, then the next challenge would be to formulate a candidate conjectural characterization for resolving the bi-Lipschitz embedding problem into $\R^k$.

The analogue of Question~\ref{Q:LP} is known to fail in some non-Hilbertian settings. Specifically, it follows from~\cite{LN14,NY17,NY17-versus} that for every $p\in (2,\infty)$ there exists a doubling subset $\mathscr{D}_p$ of $L_p(\R)$ that does not admit a bi-Lipschitz embedding into any $L_q(\mu)$ space for any $q\in [1,p)$. So, in particular there is no bi-Lipschitz embedding of $\mathscr{D}_p$ into any finite-dimensional normed space, and a fortiori there is no such embedding into any finite-dimensional subspace of $L_p(\R)$. Note that in~\cite{LN14} this statement is made for embeddings of $\mathscr{D}_p$ into $L_q(\mu)$ in the reflexive range $q\in (1,p)$, and the case $q=1$ is treated in~\cite{LN14} only when $p\ge p_0$ for some universal constant $p_0>2$. The fact that $\mathscr{D}_p$ does not admit a bi-Lipschitz embedding into any $L_1(\mu)$ space follows by combining the argument of~\cite{LN14} with the more recent result\footnote{When~\cite{LN14} was written, only a weaker bound of~\cite{CKN11} was known.} of~\cite{NY17,NY17-versus} when the underlying group in the construction of~\cite{LN14} is the $5$-dimensional Heisenberg group; interestingly we now know~\cite{NY17-3dim} that if one carries out the construction of~\cite{LN14} for the $3$-dimensional Heisenberg group, then the reasoning of~\cite{LN14}  would yield the above conclusion only when $p>4$. A different example of a doubling subset of $L_p(\R)$ that fails to embed bi-Lipschitzly into $\ell_p^k$ for any $k\in \N$ was found in~\cite{BGN15}. In $L_1(\R)$, there is an even stronger counterexample~\cite[Remark~1.4]{LN14}: By~\cite{GNRS04}, the spaces considered in~\cite{Laakso,Laa02} yields a doubling subset of $L_1(\R)$ that by~\cite{CK09-RNP} (see also~\cite{Ost11}) does not admit a bi-Lipschitz embedding into any Banach space with the Radon--Nikod\'ym property~\cite{LN06,CK06}, hence it does not admit a bi-Lipschitz embedding into any reflexive or separable dual Banach space. The potential validity of the above statement for $p\in (1,2)$ remains an intriguing open problem, and the case $p=2$ is of course Question~\ref{Q:LP}.

\section{Matou\v{s}ek's random metrics, Milnor--Thom, and coarse dimension reduction}\label{sec:matousek}

Fix two moduli $\omega,\Omega:[0,\infty)\to [0,\infty)$ as in Theorem~\ref{thm:coarse matousek}, i.e., they are increasing functions and $\omega\le \Omega$ point-wise. For a metric space $(\MM,d_\MM)$ define $\dim_{(\omega,\Omega)}(\MM,d_\MM)$ to be the smallest dimension $k\in \N$ for which there exists a $k$-dimensional normed space $(X,\|\cdot\|_X)=(X(\MM),\|\cdot\|_{X(\MM)})$ and a mapping $f:\MM\to X$ that satisfies~\eqref{eq:coarse condition}. If no such $k\in \N$ exists, then write   $\dim_{(\omega,\Omega)}(\MM,d_\MM)=\infty$. For $\alpha\in [1,\infty)$, this naturally generalizes the notation $\dim_\alpha(\MM,d_\MM)$ of~\cite{LLR} in the bi-Lipschitz setting, which coincides with  $\dim_{(t,\alpha t)}(\MM,d_\MM)$. 	

Recalling~\eqref{eq:def beta}, the goal of this section is  to show that   $\dim_{(\omega,\Omega)}(\MM,d_\MM)\ge n^{c\beta(\omega,\Omega)}$ for arbitrarily large $n\in \N$,  some universal constant $c\in (0,\infty)$ and some metric space $(\MM,d_\MM)$ with $|\MM|=3n$, thus proving Theorem~\ref{thm:coarse matousek}. We will do so by following Matou\v{s}ek's beautiful  ideas in~\cite{Mat96}, yielding a probabilistic  argument for the existence of such an intrinsically (coarsely) high-dimensional metric space $(\MM,d_\MM)$.

The collections of subsets  of a set $S$ of size $\ell\in \N$ will be denoted below $\binom{S}{\ell}=\{\mathfrak{e}\subset S:\ |\mathfrak{e}|=\ell\}$. Fix $n\in \N$ and a bipartite graph $\mathsf{G}=(\mathsf{L},\mathsf{R},\EE)$ with $|\mathsf{L}|=|\mathsf{R}|=n$. Thus, $\mathsf{L}$ and $\mathsf{R}$ are disjoint $n$-point sets (the "left side" and "right side" of $\mathsf{G}$) and $\EE$ is a subset of $\binom{\mathsf{L}\cup\mathsf{R}}{2}$ consisting only of $\mathfrak{e}\subset \mathsf{L}\cup\mathsf{R}$ such that $|\mathsf{L}\cap \mathfrak{e}|=|\mathsf{R}\cap \mathfrak{e}|=1$. Following Matou\v{s}ek~\cite{Mat96}, any such graph $\mathsf{G}$ can used as follows  as a "template" for obtaining a family  $2^{|\EE|}$ graphs, each of which having $3n$ vertices. For each $\lambda \in \mathsf{L}$ introduce two new elements $\lambda^{\!+},\lambda^{\!-}$. Denote $\mathsf{L}^{\!+}=\{\lambda^{\!+}:\ \lambda\in \mathsf{L}\}$ and $\mathsf{L}^{\!-}=\{\lambda^{\!-}:\ \lambda\in \mathsf{L}\}$. Assume that the sets $\mathsf{L}^{\!+},\mathsf{L}^{\!-},\mathsf{R}$ are  disjoint. For every $\sigma:\EE\to \{-,+\}$ define
\begin{equation}\label{eq:def signed edges}
\EE_\sigma\eqdef \left\{\big\{\lambda^{\sigma(\{\lambda,\rho\})},\rho\big\}:\ (\lambda,\rho)\in \mathsf{L}\times\mathsf{R}\ \wedge\ \{\lambda,\rho\}\in  \EE\right\}\subset \binom{\mathsf{L}^{\!+}\cup\mathsf{L}^{\!-}\cup \mathsf{R}}{2}.
\end{equation}
We thus obtain a bipartite graph $\mathsf{G}_\sigma=(\mathsf{L}^{\!+}\cup\mathsf{L}^{\!-},\mathsf{R},\EE_\sigma)$. By choosing $\sigma:\EE\to \{-,+\}$ uniformly at random, we think of $\mathsf{G}_\sigma$ as a random graph; let $\Pr$ denote the uniform probability measure on the set of all such $\sigma$. In other words, consider $\sigma:\EE\to \{-,+\}$ to be independent tosses of a fair coin, one for each edge of $\mathsf{G}$. Given an outcome of the coin tosses $\sigma$,  each edge $\mathfrak{e}\in \EE$ of $\mathsf{G}$ induces an element of $\EE_\sigma$ as follows. If $\lambda$ is the endpoint of $\mathfrak{e}$ in $\mathsf{L}$ and $\rho$ is the endpoint of $\mathfrak{e}$ in $\mathsf{R}$, then $\EE_\sigma$ contains exactly one of the unordered pairs $\{\lambda^{\!+},\rho\}, \{\lambda^{\!-},\rho\}$ depending on whether $\sigma(\mathfrak{e})=+$ or $\sigma(\mathfrak{e})=-$, respectively; see Figure~\ref{fig:mat graphs} below for a schematic depiction of this construction.
\begin{figure}[h]
\centering
%\fbox{
\begin{minipage}{7.3in}
\centering
%\medskip
\includegraphics[height=3.3in]{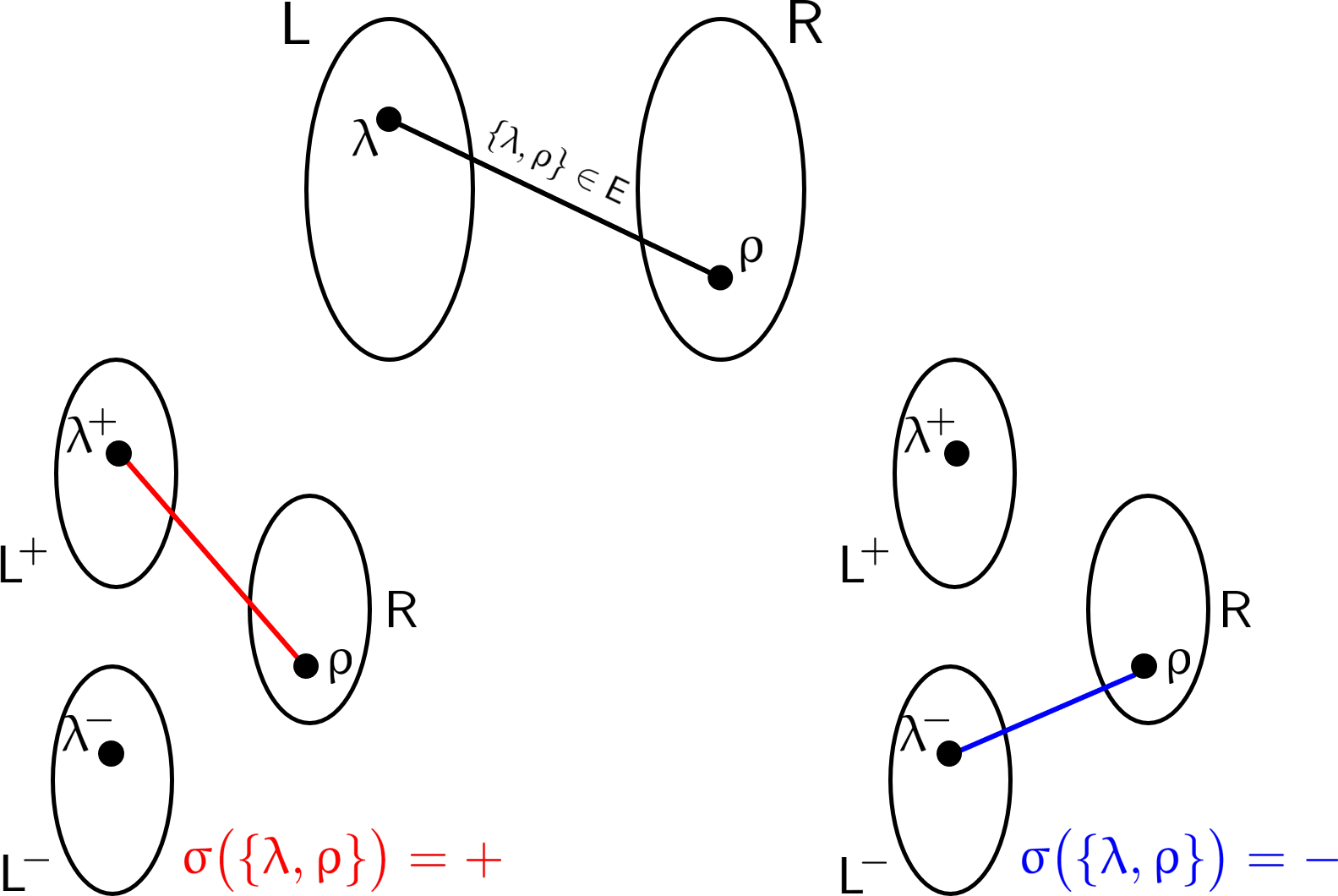}
\smallskip
\caption{The random bipartite graph $\mathsf{G}_\sigma=(\mathsf{L}^{\!+}\cup \mathsf{L}^{\!-},\mathsf{R},\EE_\sigma)$ that is associated to the bipartite  graph $\mathsf{G}=(\mathsf{L},\mathsf{R},\EE)$ and coin flips $\sigma:\mathsf{E}\to \{-,+\}$. Suppose that $(\lambda,\rho)\in \mathsf{L}\times \mathsf{R}$ and $\mathfrak{e}=\{\lambda,\rho\}\in \EE$. If the outcome of the coin that was flipped for the edge $\mathfrak{e}$ is $+$, then include in $\EE_\sigma$ the red edge between $\lambda^{\!+}$ and $\rho$ and do not include an edge between $\lambda^{\!-}$ and $\rho$. If the outcome of the coin that was flipped for the edge $\mathfrak{e}$ is $-$, then include in $\EE_\sigma$ the blue edge between $\lambda^{\!-}$ and $\rho$ and do not include an edge between $\lambda^{\!+}$ and $\rho$.
}
\label{fig:mat graphs}
\end{minipage}
%}
\end{figure}

 Let $d_{\mathsf{G_\sigma}}:(\mathsf{L}^{\!+}\cup\mathsf{L}^{\!-}\cup \mathsf{R})\times (\mathsf{L}^{\!+}\cup\mathsf{L}^{\!-}\cup \mathsf{R})\to [0,\infty]$ be the shortest-path metric corresponding to $\mathsf{G}_\sigma$, with the convention that $d_{\mathsf{G}_\sigma}(x,y)=\infty$ if $x,y\in \mathsf{L}^{\!+}\cup\mathsf{L}^{\!-}\cup \mathsf{R}$ belong to different connected components of $\mathsf{G}_\sigma$. We record for convenience of later use the following very simple observation.

 \begin{claim}\label{claim:cycle compress} Fix $\lambda\in \mathsf{L}$ and $\sigma:\EE\to \{-,+\}$. Suppose that $k\eqdef d_{\mathsf{G}_\sigma}(\lambda^{\!+},\lambda^{\!-})<\infty$.  Then the original "template graph" $\mathsf{G}$ contains a cycle of length at most  $k$.
 \end{claim}

 \begin{proof} Denote by $\pi:\mathsf{L}^{\!+}\cup \mathsf{L}^{\!-}\cup\mathsf{R}\to \mathsf{L}\cup\mathsf{R}$ the canonical "projection," i.e., $\pi$ is the identity mapping on $\mathsf{R}$ and $\pi(\lambda^{\!+})=\pi(\lambda^{\!-})=\lambda$ for every $\lambda\in \mathsf{L}$. The natural induced mapping $\pi:\EE_\sigma\to \EE$ (given by $\pi(\{x,y\})=\{\pi(x),\pi(y)\}$ for each $\{x,y\}\in \EE_\sigma$) is one-to-one, because by construction   $\EE_\sigma$ contains one and only one of the unordered pairs  $\{\mu^+,\rho\},\{\mu^-,\rho\}$ for each $(\mu,\rho)\in \mathsf{L}\times \mathsf{R}$ with $\{\mu,\rho\}\in \EE$.

 Let $\gamma: \{0,\ldots, k\}\to \mathsf{L}^{\!+}\cup \mathsf{L}^{\!-}\cup\mathsf{R}$ be a geodesic in $\mathsf{G}_\sigma$ that joins $\lambda^{\!+}$  and $\lambda^{\!-}$. Thus $\gamma(0)= \lambda^{\!+}$, $\gamma(k)= \lambda^{\!-}$ and $\{\{\gamma(i-1),\gamma(i)\}\}_{i=1}^k$ are distinct edges in $\EE_\sigma$ (they are distinct because $\gamma$ is a shortest path joining $\lambda^{\!+}$  and $\lambda^{\!-}$ in $\mathsf{G}_\sigma$). By the injectivity of $\pi$ on $\mathsf{E}_\sigma$, the unordered pairs $\{\{\pi(\gamma(i-1)),\pi(\gamma(i))\}\}_{i=1}^k$ are distinct edges in $\EE$. So, the subgraph $\mathsf{H}$ of $\mathsf{G}$ that is induced on the vertices $\{\pi(\gamma(i))\}_{i=0}^k$ has at least $k$ edges. But, $\mathsf{H}$ has at most $k$ vertices, since $\pi(\gamma(0))=\pi(\gamma(k))=\lambda$. Hence $\mathsf{H}$ is not a forest, i.e., it contains a cycle of length at most $k$.
 \end{proof}

Even though $d_{\mathsf{G_\sigma}}$ is not necessarily a metric due to its possible infinite values, for every $s,T\in (0,\infty)$ we can re-scale and truncate it so as to obtain a (random) metric $d_\sigma^{s,T}:(\mathsf{L}^{\!+}\cup\mathsf{L}^{\!-}\cup \mathsf{R})\times (\mathsf{L}^{\!+}\cup\mathsf{L}^{\!-}\cup \mathsf{R})\to [0,\infty]$  by defining
\begin{equation}\label{eq:truncated}
\forall\, x,y\in \mathsf{L}^{\!+}\cup\mathsf{L}^{\!-}\cup \mathsf{R},\qquad d_\sigma^{s,T}(x,y)\eqdef \min\big\{sd_{\mathsf{G}_\sigma}(x,y),T\big\}.
\end{equation}

The following lemma shows that if in the above construction $\mathsf{G}$ has many edges and no short cycles, then with overwhelmingly high probability the random metric in~\eqref{eq:truncated} has large coarse metric dimension.

\begin{lemma}\label{lem:prob matousek} There is a universal constant $\eta>0$ with the following property. For every $\omega,\Omega:[0,\infty)\to [0,\infty)$ as above, every $n\in \N$ and every template graph $\mathsf{G}=(\mathsf{L},\mathsf{R},\EE)$ as above, suppose that  $g\in \N$ and $s,T>0$ satisfy
\begin{equation}\label{eq:girth assumtpions}
\frac{\omega^{-1}(2\Omega(s))}{s}<g\le \frac{T}{s},
\end{equation}
and that the shortest cycle in  $\mathsf{G}$ has length at least $g$. Then for every $\d\in (0,\frac13]$ we have
\begin{equation}\label{eq:probabilistic version'}
 \Pr\left[\sigma:\EE\to \{-,+\}:\  \dim_{\omega,\Omega}\Big(\mathsf{L}^{\!+}\cup\mathsf{L}^{\!-}\cup \mathsf{R},d_\sigma^{s,T}\Big)\le \d\eta\frac{|\EE|}{n}\right] < \big(2\d^{\d}\big)^{-|\EE|}.
\end{equation}
In particular, by choosing $\d=\frac13$ in~\eqref{eq:probabilistic version'} we deduce that
\begin{equation}\label{eq:probabilistic version}
\Pr\left[\sigma:\EE\to \{-,+\}:\  \dim_{\omega,\Omega}\Big(\mathsf{L}^{\!+}\cup\mathsf{L}^{\!-}\cup \mathsf{R},d_\sigma^{s,T}\Big)> \frac{\eta|\EE|}{3n}\right] > 1-e^{-\frac15|\EE|}.
\end{equation}
\end{lemma}

Prior to proving Lemma~\ref{lem:prob matousek}, we shall now explain how it implies Theorem~\ref{thm:coarse matousek}.

\begin{proof}[Proof of Theorem~\ref{thm:coarse matousek} assuming Lemma~\ref{lem:prob matousek}]  Recalling~\eqref{eq:def beta}, we can fix $s\in (0,\infty)$ such that
\begin{equation}\label{eq:def girth}
g\eqdef \left\lfloor \frac{\omega^{-1}(2\Omega(s))}{s}  \right\rfloor+1\le \frac{2}{\beta(\omega,\Omega)}.
\end{equation}
There is a universal constant $\kappa\in (0,\infty)$ such that for arbitrarily large $n\in \N$ there exists a bipartite graph $\mathsf{G}=(\mathsf{L},\mathsf{R},\EE)$ with $|\mathsf{L}|=|\mathsf{R}|=n$, girth at least $g$ (i.e., $\mathsf{G}$ does not contain any cycle of length smaller than $g$)  and $|\EE|\ge n^{1+\kappa/g}$. Determining the largest possible value of $\kappa$ here is a well-studied and longstanding open problem in graph theory (see e.g.~the discussions in~\cite{Bol01,Mat02,Ost13}), but for the present purposes any value of $\kappa$ suffices. For the latter (much more modest) requirement, one can obtain $\mathsf{G}$ via a simple probabilistic construction (choosing each of the edges independently at random and deleting an arbitrary edge from each cycle of length at most $g-1$). See~\cite{LUW95} for the best known lower bound on $\kappa$ here (arising from an algebraic construction).

We shall use the above graph $\mathsf{G}$ as the template graph for the random graphs $\{\mathsf{G}_\sigma\}_{\sigma:\EE\to \{-,+\}}$. Our choice of $g$ in~\eqref{eq:def girth} ensures that if we write $T=sg$, then~\eqref{eq:girth assumtpions} holds true and we obtain a distribution over metric spaces $(\mathsf{L}^{\!+}\cup\mathsf{L}^{\!-}\cup \mathsf{R},d_\sigma^{s,T})$ for which the conclusion~\eqref{eq:probabilistic version} of Lemma~\ref{lem:prob matousek} holds true. Hence, by choosing $c=\kappa/2$ and substituting the  bound $|\EE|\ge n^{1+\kappa/g}$ into~\eqref{eq:probabilistic version} while using~\eqref{eq:def girth} we have
\begin{equation*}
\Pr\left[\sigma:\EE\to \{-,+\}:\  \dim_{\omega,\Omega}\Big(\mathsf{L}^{\!+}\cup\mathsf{L}^{\!-}\cup \mathsf{R},d_\sigma^{s,T}\Big) \gtrsim n^{c\beta(\omega,\Omega)}\right] \ge  1-\exp\left(-\frac15n^{1+c\beta(\omega,\Omega)}\right).
\end{equation*}
Consequently, by the definition of $\dim_{\omega,\Omega}(\cdot)$, with probability exponentially close to $1$ the random metric space $(\mathsf{L}^{\!+}\cup\mathsf{L}^{\!-}\cup \mathsf{R},d_\sigma^{s,T})$ satisfies the assertion of Theorem~\ref{thm:coarse matousek}.
\end{proof}

The proof of Lemma~\ref{lem:prob matousek}  relies on the following lemma that was obtained implicitly by Matou\v{s}ek~\cite{Mat96}. Its proof takes as input a clever argument of Alon~\cite{Al86}  which uses the classical bound of Milnor~\cite{Mil64} and Thom~\cite{Tho65} on the number of connected components of a real algebraic variety.

\begin{lemma}\label{lem:rank} Fix $m,n\in \N$ and $\mathsf{E}\subset \n^2$. Suppose that $\mA_1=(a_{ij}^1),\ldots,\mA_m=(a_{ij}^m)\in \M_n(\R)$ are matrices that satisfy $a_{ij}^k\neq 0$ for all $(i,j)\in \mathsf{E}$ and $k\in \m$, and that the sign vectors
\begin{equation}\label{eq:sign patterns}
\big(\sign(a^1_{ij})\big)_{(i,j)\in \mathsf{E}},\big(\sign(a^2_{ij})\big)_{(i,j)\in \mathsf{E}}\ldots,\big(\sign(a^m_{ij})\big)_{(i,j)\in \mathsf{E}}\subset \{-1,1\}^{\mathsf{E}}
\end{equation}
are distinct. Then there exists $k\in \{1,\ldots,m\}$ such that
\begin{equation}\label{eq:rank lower optimized}
\rank(\mA_k)\gtrsim \frac{\log m}{n\log\left(\frac{|\mathsf{E}|}{\log m}\right)}.
\end{equation}
%In particular, $\max_{k\in \m}\rank(\mA_k)\gtrsim \frac{|\mathsf{E}|}{n}$ if all possible sign patterns appear in~\eqref{eq:sign patterns}, i.e., $m=2^{|\mathsf{E}|}$.
\end{lemma}

\begin{proof} Let $\alpha\in \N$ be an auxiliary parameter that will be specified later so as to optimize the ensuing argument. Write $h=\lceil |\EE|/\alpha\rceil$ and fix any partition of $\EE$ into subsets $\JJ_1,\ldots,\JJ_h\neq \emptyset$ (i.e., $\JJ_1,\ldots,\JJ_h\subset\EE$ are pairwise disjoint and $\EE=\JJ_1\cup\ldots \cup\JJ_h$) that satisfy $|\JJ_u|\le \alpha$ for all $u\in \{1,\ldots,h\}$.

Denote
\begin{equation}\label{eq:def mu rho}
\mu\eqdef \min_{\substack{(i,j)\in \EE\\ k\in \m}}\big|a_{ij}^k\big|\qquad\mathrm{and}\qquad r\eqdef \max_{k\in \m} \rank(\mA_k).
\end{equation}
Lemma~\ref{lem:rank} assumes that $\mu>0$, and its goal is to show that $r$ is at least a universal constant multiple of the quantity that appears in the right hand side on~\eqref{eq:rank lower optimized}. The definition of $r$ means that for every $k\in \m$ there exist $n$-by-$r$ and $r$-by-$n$ matrices $\mathsf{B}_k\in \M_{n\times r}(\R)$ and $\mathsf{C}_k\in \M_{r\times n}(\R)$, respectively, such that $\mA_k=\mathsf{B}_k\mathsf{C}_k$. Define vectors $\{\zeta_k=(\zeta_1^k,\ldots,\zeta^k_h)\in \R^h\}_{k=1}^m$ by setting
\begin{equation}\label{eq:square root}
\forall\, (k,u)\in \m\times \{1,\ldots,h\},\qquad \zeta_u^k\eqdef \sqrt{\prod_{(i,j)\in\mathsf{J}_u}(a_{ij}^k)^2-\frac12 \mu^{2|\mathsf{J}_u|}}.
\end{equation}
Observe that the definition of $\mu$ in~\eqref{eq:def mu rho} ensures that the quantity under the square root in~\eqref{eq:square root} is positive, so $\zeta_u\in (0,\infty)$. Define polynomials $\{\mathfrak{p}_{u}:\M_{n\times r}(\R)\times \M_{r\times n}(\R)\times \R^h\to \R\}_{u=1}^h$ by setting
\begin{equation}\label{eq:def pu}
\mathfrak{p}_{u}(X,Y,z)\eqdef \prod_{(i,j)\in\mathsf{J}_u} (XY)_{ij}^2-z_u^2-\frac{1}{2}\mu^{2|\mathsf{J}_u|}= \prod_{(i,j)\in\mathsf{J}_u} \bigg(\sum_{k=1}^n x_{ik}y_{kj}\bigg)^2-z_u^2-\frac{1}{2}\mu^{2|\mathsf{J}_u|},
\end{equation}
for all $u\in \{1,\ldots,h\}$, $X=(x_{is})\in \M_{n\times r}(\R)$, $Y=(y_{sj})\in \M_{r\times n}(\R)$ and $z=(z_i)\in \R^h$. The above notation ensures that $\mathfrak{p}_u(\mathsf{B}_k,\mathsf{C}_k,\zeta_k)=0$ for all $k\in \m$ and $u\in \{1,\ldots,h\}$. In other words, $\{(\mathsf{B}_k,\mathsf{C}_k,\zeta_k)\}_{k=1}^m\subset \mathscr{V}$, where $\mathscr{V}\subset \M_{n\times r}(\R)\times \M_{r\times n}(\R)\times \R^h$ is the variety
\begin{equation}\label{eq:def variety}
\mathscr{V}\eqdef \bigcap_{u=1}^h \left\{(X,Y,z)\in \M_{n\times r}(\R)\times \M_{r\times n}(\R)\times \R^h;\ \mathfrak{p}_{u}(X,Y,z)=0\right\}.
\end{equation}

We claim that $\mathscr{V}$ has at least $m$ connected components. In fact, if $k,\ell\in \m$ are distinct, then $(\mathsf{B}_k,\mathsf{C}_k,\zeta_k)$ and $(\mathsf{B}_\ell,\mathsf{C}_\ell,\zeta_\ell)$ belong to different connected component of $\mathscr{V}$. Indeed, suppose for the sake of obtaining a contradiction that $\mathcal{C}\subset \mathscr{V}$ is a connected subset of $\mathscr{V}$ and $(\mathsf{B}_k,\mathsf{C}_k,\zeta_k),(\mathsf{B}_\ell,\mathsf{C}_\ell,\zeta_\ell)\in \mathcal{C}$. Since $k\neq \ell$, by switching the roles of $k$ and $\ell$ if necessary, the assumption of Lemma~\ref{lem:rank} ensures that there exists $(i,j)\in \EE$ such that $(\mathsf{B}_k\mathsf{C}_k)_{ij}=a_{ij}^k<0<a_{ij}^\ell=(\mathsf{B}_\ell \mathsf{C}_\ell)_{ij}$. So, if we denote $\psi:\mathcal{C}\to \R$  by $\psi(X,Y,z)=(XY)_{ij}$, then $\psi(\mathsf{B}_k,\mathsf{C}_k,\zeta_k)<0<\psi(\mathsf{B}_\ell,\mathsf{C}_\ell,\zeta_\ell)$. Since $\mathcal{C}$ is connected and $\psi$ is continuous, it follows that $\psi(X,Y,z)=0$ for some $(X,Y,z)\in \mathcal{C}$. Let $u\in \{1,\ldots,h\}$ be the index for which $(i,j)\in \JJ_u$. By the definition~\eqref{eq:def pu} of $\mathfrak{p}_u$, the fact that $\psi(X,Y,z)=(XY)_{ij}=0$ implies that $\mathfrak{p}_u(X,Y,z)=-z_u^2-\frac12\mu^{|J_u|}\le -\frac12\mu^{|J_u|}<0$, since $\mu>0$. Hence $(X,Y,z)\notin \mathscr{V}$, in contradiction to our choice of $(X,Y,z)$ as an element of $\mathcal{C}\subset \mathscr{V}$.

Recalling~\eqref{eq:def pu}, for all $u\in \{1,\ldots,h\}$ the degree of $\mathfrak{p}_u$ is  $4|\JJ_u|\le 4\alpha$. So, the variety $\mathscr{V}$ in~\eqref{eq:def variety} is defined using $h$ polynomials of degree at most $4\alpha$ in $2nr+h$ variables. By (a special case of) a theorem of Milnor~\cite{Mil64} and Thom~\cite{Tho65}, the number of connected components of $\mathscr{V}$ is at most $4\alpha(8\alpha-1)^{2nr+h-1}=4\alpha(8\alpha-1)^{2nr+\lceil |\EE|/\alpha\rceil-1}$. Since we already established that this number of connected components is at least $m$, it follows that
\begin{equation}\label{eq:alpha to optimize}
m\le 4\alpha(8\alpha-1)^{2nr+\left\lceil \frac{|\EE|}{\alpha}\right\rceil-1}\iff r\ge \frac{1}{2n}
\left(\frac{\log\left(\frac{m}{4\alpha}\right)}{\log(8\alpha-1)}-\left\lceil \frac{|\EE|}{\alpha}\right\rceil+1\right).
\end{equation}
The value of $\alpha\in \N$ that maximizes  the rightmost quantity in~\eqref{eq:alpha to optimize} satisfies
$$
\alpha\asymp \frac{|\EE|}{\log (2m)}\left(\log\left(\frac{|\EE|}{\log (2m)}\right)\right)^2.
$$
For this $\alpha$ the estimate~\eqref{eq:alpha to optimize} simplifies to imply the desired bound $r\gtrsim \log m/(n\log(|\EE|/\log m))$.
\end{proof}

\begin{proof}[Proof of Lemma~\ref{lem:prob matousek}] For notational convenience, write $\mathsf{L}=\{\lambda_1,\ldots,\lambda_n\}$ and  $\mathsf{R}=\{\rho_1,\ldots,\rho_n\}$, and think of $\EE$ as a subset of $\n^2$ (i.e., $(i,j)\in \EE$  if and only if $\{\lambda_i,\rho_j\}$ is an edge of $\mathsf{G}$).

For every $\Delta\in \N$ denote
\begin{equation}\label{eq:def bad event}
\mathscr{B}_{\!\Delta}\eqdef \left\{\sigma:\EE\to \{-,+\}:\  \dim_{\omega,\Omega}\Big(\mathsf{L}^{\!+}\cup\mathsf{L}^{\!-}\cup \mathsf{R},d_\sigma^{s,T}\Big)\le  \Delta\right\}.
\end{equation}
Then, by the definition of $\dim_{\omega,\Omega}(\cdot)$, if $\sigma\in \mathscr{B}_{\!\Delta}$,  we can fix a normed space $(X_\sigma,\|\cdot\|_{X_\sigma})$ with $\dim(X_\sigma)=\Delta$ and  a mapping $f_\sigma: \mathsf{L}^{\!+}\cup\mathsf{L}^{\!-}\cup \mathsf{R}\to X_\sigma$ that satisfies
\begin{equation}\label{eq:omega Omega truncated}
\forall\, x,y\in \mathsf{L}^{\!+}\cup\mathsf{L}^{\!-}\cup \mathsf{R},\qquad  \omega\big(d_\sigma^{s,T}(x,y)\big)\le \|f_\sigma(x)-f_\sigma(y)\|_{X_\sigma}\le \Omega\big(d_\sigma^{s,T}(x,y)\big).
\end{equation}

Using the Hahn--Banach theorem, for each $i\in \n$  and $\sigma\in \mathscr{B}_{\!\Delta}$ we can fix a linear functional $z_{\sigma,i}^*\in X_\sigma^*$ of unit norm that normalizes the vector  $f_\sigma(\lambda_i^{\!+})-f_\sigma(\lambda_i^{\!-})\in X_\sigma$, i.e.,
\begin{equation}\label{eq:HB-mat}
z^*_{\sigma,i}\big(f_\sigma(\lambda^{\!+}_i)-f_\sigma(\lambda^{\!-}_i)\big)=\big\|f_\sigma(\lambda^{\!+}_i)-f_\sigma(\lambda^{\!+}_i)\big\|_{X_\sigma}
\qquad\mathrm{and}\qquad \|z^*_{\sigma,i}\|_{X_\sigma^*}=\sup_{w\in X_\sigma\setminus \{0\}}\frac{\big|z_{\sigma,i}^*(w)\big|}{\|w\|_{X_\sigma}}=1.
\end{equation}
Using these linear functionals, define an $n\times n$ matrix  $\mA_\sigma=(a^\sigma_{ij})\in \M_n(\R)$ by setting
\begin{equation}\label{eq:our gram}
\forall(i,j)\in \n^2,\qquad a_{ij}^\sigma\eqdef
z_{\sigma,i}^*\big(f_\sigma(\rho_j)\big)-\frac12 z_{\sigma,i}^*\big( f_\sigma(\lambda_i^{\!+})+f_\sigma(\lambda_i^{\!-})\big).
\end{equation}
Observe in passing that the following identity holds true for every $i,j\in \n$ and $\sigma \in \mathscr{B}_{\!\Delta}$.
\begin{equation}\label{eq:AR identity}
\sigma(i,j)a_{ij}^\sigma=\frac12 z^*_{\sigma,i}\big(f_\sigma(\lambda^{\!+}_i)-f_\sigma(\lambda^{\!-}_i)\big) +\sigma(i,j)z_{\sigma,i}^*\Big(f_\sigma(\rho_j)-f_\sigma\Big(\lambda_i^{\sigma(i,j)}\Big)\Big).
\end{equation}
(Simply verify~\eqref{eq:AR identity} for the cases $\sigma(i,j)=+$ and $\sigma(i,j)=-$ separately, using the linearity of $z^*_{\sigma,i}$.)

Since we are assuming in Lemma~\ref{lem:prob matousek} that the shortest cycle in the template graph $\mathsf{G}$ has length at least $g$, it follows from Claim~\ref{claim:cycle compress} that $d_{\mathsf{G}_\sigma}(\lambda_i^{\!+},\lambda_i^{\!-})\ge g$ for all $i\in \n$ and $\sigma:\EE\to \{-,+\}$. So,
\begin{equation}\label{eq:plus minus dist big}
d_\sigma^{s,T}(\lambda_i^{\!+},\lambda_i^{\!-})\stackrel{\eqref{eq:truncated}}{=}\min\left\{sd_{\mathsf{G}_\sigma}(\lambda_i^{\!+},\lambda_i^{\!-}),T\right\}\ge \min\{sg,T\}\stackrel{\eqref{eq:girth assumtpions}}{=}sg.
\end{equation}
Recalling~\eqref{eq:def signed edges}, we have $\{\lambda_i^{\sigma(i,j)},\rho_j\}\in \EE_\sigma$ for all $(i,j)\in \EE$. Hence $d_{\mathsf{G}_\sigma}(\lambda_i^{\sigma(i,j)},\rho_j)=1$ and therefore
\begin{equation}\label{eq:truncated upper edge}
d_\sigma^{s,T}\big(\lambda_i^{\sigma(i,j)},\rho_j\big)\stackrel{\eqref{eq:truncated}}{\le} sd_{\mathsf{G}_\sigma}\big(\lambda_i^{\sigma(i,j)},\rho_j\big)=s.
\end{equation}
Consequently, for every $(i,j)\in \EE$ and $\sigma\in \mathscr{B}_{\!\Delta}$ we have
\begin{eqnarray*}
\sigma(i,j)a_{ij}^\sigma&\stackrel{\eqref{eq:AR identity}}{\ge} & \frac12 z^*_{\sigma,i}\big(f_\sigma(\lambda^{\!+}_i)-f_\sigma(\lambda^{\!-}_i)\big)-\left\|z_{\sigma,i}^*\right\|_{X_\sigma^*}\cdot \Big\|f_\sigma(\rho_j)-f_\sigma\Big(\lambda_i^{\sigma(i,j)}\Big)\Big\|_{X_\sigma}\\&\stackrel{\eqref{eq:HB-mat}}{=}&\frac12 \big\|f_\sigma(\lambda^{\!+}_i)-f_\sigma(\lambda^{\!-}_i)\big\|_{X_\sigma}-\Big\|f_\sigma(\rho_j)-f_\sigma\Big(\lambda_i^{\sigma(i,j)}\Big)\Big\|_{X_\sigma}\\
&\stackrel{\eqref{eq:omega Omega truncated}}{\ge} &\frac12\omega \big(d_\sigma^{s,T}(\lambda_i^{\!+},\lambda_i^{\!-})\big)-\Omega \Big(d_\sigma^{s,T}\big(\lambda_i^{\sigma(i,j)},\rho_j\big)\Big)\\
&\stackrel{\eqref{eq:plus minus dist big}\wedge \eqref{eq:truncated upper edge}}{\ge}& \frac12 \omega(sg)-\Omega(s)\stackrel{\eqref{eq:girth assumtpions}}{>}0.
\end{eqnarray*}
Hence, $a_{ij}^\sigma\neq 0$ and $\sign(a_{ij}^\sigma)=\sigma(i,j)$ for all $(i,j)\in \EE$ and $\sigma\in \mathscr{B}_{\!\Delta}$. This is precisely the setting of Lemma~\ref{lem:rank} (with $m=|\mathscr{B}_{\!\Delta}|$), from which we conclude that there exists $\tau\in \mathscr{B}_{\!\Delta}$ such that
\begin{equation}\label{eq:rank lower use}
 \rank(\mA_\tau)\ge \frac{c\log |\mathscr{B}_{\!\Delta}|}{n\log\left(\frac{|\mathsf{E}|}{\log |\mathscr{B}_{\!\Delta}|}\right)},
\end{equation}
where $c\in (0,\infty)$ is a universal constant. Henceforth, we shall fix a specific $\tau\in \mathscr{B}_{\!\Delta}$ as in~\eqref{eq:rank lower use}.

Since $\tau\in \mathscr{B}_{\!\Delta}$ we have $\dim(X_\tau)=\Delta$, so we can fix a basis $e_{\tau}^1,\ldots,e_{\tau}^\Delta$ of $X_\tau$ and for every $j\in \n$ write $f_\tau(\rho_j)=\gamma_{\tau,j}^1e_{\tau}^1+\ldots+\gamma_{\tau,j}^\Delta e_{\tau}^\Delta$ for some scalars $\gamma_{\tau,j}^1,\ldots,\gamma_{\tau,j}^\Delta\in \R$.  Hence,
$$
(a_{ij}^\tau)_{i=1}^n\stackrel{\eqref{eq:our gram}}{=}\gamma_{\tau,j}^1 \big(z_{\tau,i}^*(e_\tau^1)\big)_{i=1}^n+\ldots+\gamma_{\tau,j}^\Delta \big(z_{\tau,i}^*(e_\tau^\Delta)\big)_{i=1}^n-\frac12\big(z_{\sigma,i}^*\big( f_\tau(\lambda_i^{\!+})+f_\tau(\lambda_i^{\!-})\big)\big)_{i=1}^n.
$$
We have thus expressed the columns of the matrix $\mA_\tau$ as elements of the span of the $\Delta+1$ vectors
$$
\big(z_{\tau,i}^*(e_\tau^1)\big)_{i=1}^n,\big(z_{\tau,i}^*(e_\tau^2)\big)_{i=1}^n,\ldots,\big(z_{\tau,i}^*(e_\tau^\Delta)\big)_{i=1}^n,\big(z_{\sigma,i}^*\big( f_\tau(\lambda_i^{\!+})+f_\tau(\lambda_i^{\!-})\big)\big)_{i=1}^n\in \R^n.
$$
Consequently, the rank of $\mA_\tau$ is at most $\Delta+1\le 2\Delta$. By contrasting this with~\eqref{eq:rank lower use}, we see that
\begin{equation}\label{eq:almost matousek punchline}
\frac{|\EE|}{\log |\mathscr{B}_{\!\Delta}|}\log\left(\frac{|\mathsf{E}|}{\log |\mathscr{B}_{\!\Delta}|}\right)\ge \frac{c|\EE|}{2\Delta n}.
\end{equation}

We shall now conclude by showing that Lemma~\ref{lem:prob matousek}  holds true with $\eta=c/2$. Indeed, fix $\delta\in (0,\frac13]$ and observe that we may assume also that $\d\eta|\EE|/n\ge 1$, since otherwise the left hand side of~\eqref{eq:probabilistic version'} vanishes. Then, by choosing $\Delta=\lfloor \d\eta|\EE|/n\rfloor\in \N$ in the above reasoning it follows from~\eqref{eq:almost matousek punchline} that
$$
\frac{|\EE|}{\log |\mathscr{B}_{\!\Delta}|}\log\left(\frac{|\mathsf{E}|}{\log |\mathscr{B}_{\!\Delta}|}\right)\ge \frac{c}{2\d\eta}=\frac{1}{\d}\ge 3>e.
$$
This implies that $|\mathscr{B}_{\!\Delta}|\le \d^{-\d|\EE|}$. Equivalently, $\Pr[\mathscr{B}_{\!\Delta}]\le (2\d)^{-\d|\EE|}$, which is the desired bound~\eqref{eq:probabilistic version'}.
\end{proof}

\section{Nonlinear spectral gaps and impossibility of average dimension reduction}\label{sec:average}

Fix $n\in \N$ and an irreducible reversible row-stochastic matrix $\mA=(a_{ij})\in \M_n(\R)$. This implies that there is a   unique\footnote{We are assuming irreducibility only for notational convenience, namely so that $\pi$ will be unique and could therefore be suppressed in the ensuing notation. Our arguments work for any stochastic matrix and any probability measure $\pi$ on $\n$ with respect to which $\mA$ is reversible. We suggest focusing initially on the case when $\mA$ is symmetric and $\pi$ is the uniform measure on $\n$, though the general case is useful for treating graphs that are not regular, e.g.~those of Section~\ref{sec:matousek}. See~\cite{LPW17} for the relevant background.} $\AA$-stationary probability measure $\pi=(\pi_1,\ldots,\pi_n)\in [0,1]^n$ on $\n$, namely $\pi\AA=\pi$, and  we have the reversibility condition  $\pi_ia_{ij}=\pi_ja_{ji}$ for all $i,j\in \n$. Then $\AA$ is a self-adjoint contraction on $L_2(\pi)$, and we denote by $1=\lambda_1(\AA)\ge \lambda_2(\AA)\ge\ldots\ge \lambda_n(\AA)\ge -1$ the decreasing rearrangement of its eigenvalues.

Given a metric space $(\MM,d_\MM)$ and $p>0$, define $\gamma(\AA,d_\MM^p)$ to be the infimum over those $\gamma>0$ such that
\begin{equation}\label{eq:def nonlinear gap}
\forall\, x_1,\ldots,x_n\in \MM,\qquad \sum_{i=1}^n\sum_{j=1}^n \pi_i\pi_j d_\MM(x_i,x_j)^p\le \gamma\sum_{i=1}^n\sum_{j=1}^n \pi_ia_{ij} d_\MM(x_i,x_j)^p.
\end{equation}
This definition is implicit in~\cite{Gro03}, and appeared explicitly in~\cite{NS11}; see~\cite{MN14,Nao14,MN15} for a detailed treatment. It suffices to note here that if $(H,\|\cdot\|_H)$ is a Hilbert space and $p=2$, then by expanding the squares one directly sees that $\gamma(\AA,\|\cdot\|_H^2)=1/(1-\lambda_2(\AA))$ is the reciprocal of the {\em spectral gap} of $\AA$. In general, we think of $\gamma(\AA,d_\MM^p)$ as measuring the magnitude of the nonlinear spectral gap of $\AA$ with respect to the kernel $d_\MM^p:\MM\times \MM\to [0,\infty)$.

Using the notation that was recalled in Section~\ref{sec:bilip}, the definition~\eqref{eq:def nonlinear gap} immediately implies that nonlinear spectral gaps are bi-Lipschitz invariants in the sense that $\gamma(\AA,d_\MM^p)\le \cc_\NN(\MM)^p\gamma(\AA,d_\NN^p)$ for every two metric spaces $(\MM,d_\MM)$ and $(\NN,d_\NN)$, every matrix $\AA$ as above and every $p>0$. In particular, if $(H,\|\cdot\|_H)$ is a Hilbert space into which $(\MM,d_\MM)$ admits a bi-Lipschitz embedding, then we have the following general (trivial) bound.
\begin{equation}\label{eq:trivial spectral gap bound}
\sqrt{\gamma(\AA,d_\MM^2)}\le \cc_2(\MM)\sqrt{\gamma(\AA,\|\cdot\|_H^2)}.
\end{equation}
In the recent work~\cite{Nao17} we proved the following theorem, which improves over~\eqref{eq:trivial spectral gap bound} when $\MM$ is a Banach space.
\begin{theorem}\label{thm:our gamma} Suppose that $(X,\|\cdot\|_X)$  is a Banach space and that $(H,\|\cdot\|_H)$ is a Hilbert space. Then for every $M\in (0,\infty)$ and every matrix $\AA$ as above for which $\lambda_2(\AA)\le 1-M^2/\cc_2(X)^2$ we have
\begin{equation}\label{eq:M version}
\sqrt{\gamma\big(\mA,\|\cdot\|_X^2\big)}\lesssim \frac{\log (M+1)}{M}\cc_2(X)\sqrt{\gamma(\AA,\|\cdot\|_H^2)}.
\end{equation}
\end{theorem}
In the setting of Theorem~\ref{thm:our gamma}, since $\gamma(\AA,\|\cdot\|_H^2)=1/(1-\lambda_2(\AA))$, the bound~\eqref{eq:M version} can be rewritten as
\begin{equation*}
\gamma\big(\mA,\|\cdot\|_X^2\big)\lesssim \bigg(\frac{\log(\cc_2(X)\sqrt{1-\lambda_2(\mA)}+1)}{1-\lambda_2(\mA)}\bigg)^2,
\end{equation*}
which is how Theorem~\ref{thm:our gamma} was stated in~\cite{Nao17}. Note that~\eqref{eq:M version} coincides (up to the implicit constant factor) with the trivial bound~\eqref{eq:trivial spectral gap bound} if $M=O(1)$, but~\eqref{eq:M version} is an asymptotic improvement over~\eqref{eq:trivial spectral gap bound} as $M\to\infty$.

The proof of Theorem~\ref{thm:our gamma} in~\cite{Nao17} is a short interpolation argument that takes as input a theorem from~\cite{Nao14}. While we do not know of a different proof of~\eqref{eq:M version}, below we will present a new and self-contained derivation of the following weaker estimate (using the same notation as in Theorem~\ref{thm:our gamma}) that suffices for deducing Theorem~\ref{thm:average distortion}.
\begin{equation}\label{eq:M-d version}
\sqrt{\gamma\big(\mA,\|\cdot\|_X^2\big)}\lesssim \frac{\log (\cc_2(X)+1)}{M}\cc_2(X)\sqrt{\gamma(\AA,\|\cdot\|_H^2)}.
\end{equation}
The proof of~\eqref{eq:M-d version} appears in Section~\ref{sec:ray} below. We will next show how~\eqref{eq:M-d version} implies   Theorem~\ref{thm:average distortion} as well as the leftmost inequality in~\eqref{eq:linfty bounds in theorem}, which includes as a special case the triple logarithmic estimate in~\eqref{eq:log log log}.

Given $n\in \N$ and an $n$-vertex connected graph $\mathsf{G}=(\n,\mathsf{E}_{\mathsf{G}})$, let $\AA_\G$ be its random walk matrix, i.e., if $\deg_\G(i)$ is the degree in $\G$ of the vertex $i\in \n$, then $(\AA_\G)_{ij}=\1_{\{i,j\}\in \EE_\G}/\deg_\G(i)$ for all $i,j\in \n$. For $i\in \n$ we also denote $\pi_i^\G=\deg_\G(i)/(2|\EE_\G|)$. Thus, $\pi^\G\in \R^n$ is the probability measure on $\n$ with respect to which $\AA_\G$ is reversible. We will use the simpler notation $\lambda_i(\AA_\G)=\lambda_i(\G)$ for every $i\in \n$. For $p\in (0,\infty)$and a metric space $(\MM,d_\MM)$, we will write $\gamma(\G,d_\MM^2)=\gamma(\AA_\G,d_\MM^2)$. The shortest-path metric that is induced by $\G$ on $\n$ will be denoted $d_\G:\n\times \n\to \N\cup\{0\}$.

\begin{theorem}\label{thm:pi version average embedding} There is a universal constant $K>1$ with the following property. Fix $n\in \N$ and $\alpha\ge 1$. Let $\AA\in \M_n(\R)$ and $\pi\in [0,1]^n$ be as above. For every normed space $(X,\|\cdot\|_X)$, if $f:\n\to X$ satisfies
\begin{equation}\label{eq:two averages}
\bigg(\sum_{i=1}^n\sum_{j=1}^n \pi_i a_{ij}\|f(i)-f(j)\|_X^2\bigg)^{\frac12}\le \alpha,
\end{equation}
then necessarily
\begin{equation}\label{eq:dim condition}
\dim(X)\gtrsim K^{\frac{1-\lambda_2(\AA)}{\alpha}\sqrt{\sum_{i=1}^n\sum_{j=1}^n \pi_i\pi_j \|f(i)-f(j)\|_X^2}}.
\end{equation}
In particular, in the special case when $\G=(\n,\EE_\G)$ is a connected graph we have
\begin{equation*}\label{eq:two averages-graph}
\bigg(\frac{1}{|\EE_\G|}\sum_{\{i,j\}\in \EE_\G} \|f(i)-f(j)\|_X^2\bigg)^{\frac12}\le \alpha\implies \dim(X)\gtrsim K^{\frac{1-\lambda_2(\G)}{\alpha}\sqrt{\sum_{i=1}^n\sum_{j=1}^n \pi_i^\G\pi_j^\G \|f(i)-f(j)\|_X^2}}.
%\qquad\mathrm{and}\qquad \sum_{i=1}^n\sum_{j=1}^n \pi_i\pi_j \|f(i)-f(j)\|_X^2\ge  \sum_{i=1}^n\sum_{j=1}^n \pi_i\pi_j d_\G(i,j)^2,
\end{equation*}
\begin{comment}
\begin{equation}\label{eq:upper W12 norm}
\bigg(\frac{1}{|\EE_\G|}\sum_{\{i,j\}\in \EE_\G} \|f(i)-f(j)\|_X^2\bigg)^{\frac12}\le \alpha,
\end{equation}
and
\begin{equation}\label{eq:lower quadratic aaverage dist}
\bigg(\sum_{i=1}^n\sum_{j=1}^n \pi_i\pi_j \|f(i)-f(j)\|_X^2\bigg)^{\frac12}\ge  \bigg(\sum_{i=1}^n\sum_{j=1}^n \pi_i\pi_j d_\G(i,j)^2\bigg)^{\frac12},
\end{equation}

then necessarily
\begin{equation*}\label{eq:dim condition-graph}
\dim(X)\gtrsim K^{\frac{1-\lambda_2(\G)}{\alpha}\sqrt{\sum_{i=1}^n\sum_{j=1}^n \pi_i^\G\pi_j^\G \|f(i)-f(j)\|_X^2}}.
\end{equation*}
\end{comment}
\end{theorem}

In the case of regular graphs with a spectral gap, Theorem~\ref{thm:pi version average embedding}  has the following corollary.

\begin{corollary}\label{cor:regular case} Fix two integers $n,r\ge 3$ and let $\G=(\n,\EE_\G)$ be a connected $r$-regular graph. If $(X,\|\cdot\|_X)$ is a normed space into which there is a mapping $f:\n\to X$ that satisfies
\begin{equation}\label{eq:two averages-regular}
\bigg(\frac{1}{|\EE_\G|}\sum_{\{i,j\}\in \EE_\G} \|f(i)-f(j)\|_X^2\bigg)^{\frac12}\le \alpha\qquad\mathrm{and}\qquad \bigg(\frac{1}{n^2}\sum_{i=1}^n\sum_{j=1}^n \|f(i)-f(j)\|_X^2\bigg)^{\frac12}\ge  \frac{1}{n^2}\sum_{i=1}^n\sum_{j=1}^n d_\G(i,j),
\end{equation}
then necessarily
$$
\dim(X)\gtrsim n^{\frac{c(1-\lambda_2(\G))}{\alpha\log r}},
$$
where $c\in (0,\infty)$ is a universal constant.
\end{corollary}

\begin{proof}This is nothing more than a special case of Theorem~\ref{thm:pi version average embedding}  once  we note that by a straightforward and standard counting argument (see e.g.~\cite{Mat97}) we have $\frac{1}{n}\sum_{i=1}^n\sum_{j=1}^n d_\G(i,j)\gtrsim \log_r n$.
\end{proof}

For every integer $r\ge 3$ there exist arbitrarily large $r$-regular graphs $\G$ with $\lambda_2(\G)=1-\Omega(1)$; see~\cite{HLW} for this  and much more on such {\em spectral expanders}. Corollary~\ref{cor:regular case} shows that the shortest-path metric on any such graph with $r=O(1)$  satisfies the conclusion of Theorem~\ref{thm:average distortion}, because the $\alpha$-Lipschitz assumption of Theorem~\ref{thm:average distortion} implies the first inequality in~\eqref{eq:two averages-regular} and the assumption  $\frac{1}{n^2}\sum_{i=1}^n\sum_{j=1}^n \|f(i)-f(j)\|_X\ge \frac{1}{n^2}\sum_{i=1}^n\sum_{j=1}^n d_\G(i,j)$ of Theorem~\ref{thm:average distortion} implies the second inequality in~\eqref{eq:two averages-regular} (using  Jensen's inequality).

Note that we actually proved above that any expander is "metrically high dimensional" in a stronger sense. Specifically, if $\G=(\n,\EE_\G)$ is a $O(1)$-spectral expander, i.e., it is $O(1)$-regular and $\lambda_2(\G)\le 1-\Omega(1)$, and one finds vectors $x_1,\ldots,x_n$ in a normed space $(X,\|\cdot\|_X)$ for which the averages $\frac{1}{|\EE_\G|}\sum_{\{i,j\}\in \EE_\G} \|x_i-x_j\|_X^2$ and $\frac{1}{n^2}\sum_{i,j=1}^n \|x_i-x_j\|_X^2$ are within a $O(1)$ factor of the averages $\frac{1}{|\EE_\G|}\sum_{\{i,j\}\in \EE_\G} d_\G(i,j)^2=1$ and $\frac{1}{n^2}\sum_{i,j=1}^n d_\G(i,j)^2$, respectively, then this "finitary average distance information" (up to a fixed but potentially very large multiplicative error) forces the ambient space $X$ to be very high (worst-possible) dimensional, namely $\dim(X)\ge n^{\Omega(1)}$.

\begin{remark}\label{rem:for average snowflake} If one replaces~\eqref{eq:two averages-regular} by the requirement that for an increasing modulus $\omega:[0,\infty)\to [0,\infty)$ we have,
\begin{equation}\label{eq:two averages-regular-omega}
\bigg(\frac{1}{|\EE_\G|}\sum_{\{i,j\}\in \EE_\G} \|f(i)-f(j)\|_X^2\bigg)^{\frac12}\le 1\qquad\mathrm{and}\qquad \bigg(\frac{1}{n^2}\sum_{i=1}^n\sum_{j=1}^n \|f(i)-f(j)\|_X^2\bigg)^{\frac12}\ge  \frac{1}{n^2}\sum_{i=1}^n\sum_{j=1}^n \omega\big(d_\G(i,j)\big),
\end{equation}
then the above argument applies mutatis mutandis to yield the conclusion
\begin{equation}\label{eq:omega lower dimension}
\dim(X)\gtrsim e^{c(1-\lambda_2(\G))\omega(c\log_r n)}.
\end{equation}
Indeed, the aforementioned counting argument shows that least $50\%$ of the pairs $(i,j)\in \n^2$ satisfy $d_\G(i,j)\gtrsim \log_r n$. Compare~\eqref{eq:omega lower dimension}  to Theorem~\ref{thm:coarse matousek} which provides a stronger bound if the average requirement~\eqref{eq:two averages-regular-omega} is replaced by its pairwise counterpart~\eqref{eq:coarse condition}. Nevertheless, the bound~\eqref{eq:omega lower dimension} is quite sharp (at least when $r=O(1)$ and $\lambda_2(\G)=1-\Omega(1)$), in the sense that there is a normed space $(X,\|\cdot\|_X)$ for which~\eqref{eq:two averages-regular-omega} holds and
\begin{equation}\label{eq:dim upper compression expander}
\dim(X)\lesssim e^{C(\log r)\omega\left(\frac{C\log n}{\sqrt{1-\lambda_2(\G)}}\right)}\log n,
\end{equation}
where $C>0$ is a universal constant. Indeed, by~\cite{CFM94} the diameter of the metric space $(\n,d_\G)$ satisfies $\diam(\G)\lesssim (\log n)/\sqrt{1-\lambda_2(\G)}$. By an application of~\eqref{eq:linfty bounds in theorem} with $\alpha \asymp (\log_r n)/\omega(\diam(\G))$ there exists a normed space  $(X,\|\cdot\|_X)$ with $\dim(X)\lesssim n^{O(1/\alpha)}\log n$, thus~\eqref{eq:dim upper compression expander} holds, and a mapping $f:\n\to X$ that satisfies  $d_\G(i,j)/\alpha\le\|f(i)-f(j)\|_X\le d_\G(i,j)$ for all $i,j\in \n$. Hence, the first inequality in~\eqref{eq:two averages-regular-omega} holds, and
$$
\bigg(\frac{1}{n^2}\sum_{i=1}^n\sum_{j=1}^n \|f(i)-f(j)\|_X^2\bigg)^{\frac12}\ge \frac{1}{\alpha}  \bigg(\frac{1}{n^2}\sum_{i=1}^n\sum_{j=1}^n d_\G(i,j)^2\bigg)^{\frac12}\gtrsim\frac{\log_r n}{\alpha}\asymp\omega\big(\diam(\G)\big)\ge \frac{1}{n^2}\sum_{i=1}^n\sum_{j=1}^n \omega\big(d_\G(i,j)\big).
$$
\end{remark}

\begin{proof}[Proof of Theorem~\ref{thm:pi version average embedding}  assuming~\eqref{eq:M-d version}]  Let $C\in (0,\infty)$ be the implicit universal constant in~\eqref{eq:M-d version}. Then
\begin{multline}\label{eq:use log gap}
 \bigg(\sum_{i=1}^n\sum_{j=1}^n \pi_i\pi_j \|f(i)-f(j)\|_X^2\bigg)^{\frac12}\stackrel{\eqref{eq:def nonlinear gap}}{\le} \sqrt{\gamma(\AA,\|\cdot\|_X^2)}\bigg(\sum_{i=1}^n\sum_{j=1}^n \pi_i a_{ij} \|f(i)-f(j)\|_X^2\bigg)^{\frac12}\\\stackrel{\eqref{eq:two averages}}{\le} \alpha\sqrt{\gamma(\AA,\|\cdot\|_X^2)}\stackrel{\eqref{eq:M-d version}}{\le} \frac{C\alpha\log(\cc_2(X)+1)}{1-\lambda_2(\AA)},
\end{multline}
where the last step of~\eqref{eq:use log gap} is an application of~\eqref{eq:M-d version} with $M=\cc_2(X)\sqrt{1-\lambda_2(\AA)}$, while using that for a Hilbert space $H$ we have $\gamma(\AA,\|\cdot\|_H^2)=1/(1-\lambda_2(\AA))$. It follows that
\begin{equation}\label{eq:use john}
2\sqrt{\dim(X)}\ge 2\cc_2(X)\ge\cc_2(X)+1\stackrel{\eqref{eq:use log gap}}{\ge} e^{\frac{1-\lambda_2(\AA)}{C\alpha}\sqrt{\sum_{i=1}^n\sum_{j=1}^n \pi_i\pi_j \|f(i)-f(j)\|_X^2}},
\end{equation}
where the first step of~\eqref{eq:use john} uses John's theorem~\cite{Joh48}. This establishes~\eqref{eq:dim condition} with $K=e^{2/C}>1$.
\end{proof}

For non-contracting embeddings (in particular, for bi-Lipschitz embedding), the proof of the following lemma is an adaptation of the proof of~\cite[Theorem~13]{ABN11}.

\begin{lemma}\label{lem:ABN referee}Fix two integers $n,r\ge 3$ and let $\G=(\n,\EE_\G)$ be a connected $r$-regular graph. If $(X,\|\cdot\|_X)$ is a normed space into which there is a mapping $f:\n\to X$ that satisfies
\begin{equation}\label{eq:non contracting}
 \min_{\substack{i,j\in \n\\ i\neq j}}\frac{\|f(i)-f(j)\|_X}{d_\G(i,j)}\ge   1,\qquad \mathrm{and}\qquad \bigg(\frac{1}{|\EE_\G|}\sum_{\{i,j\}\in \EE_\G} \|f(i)-f(j)\|_X^2\bigg)^{\frac12}\le \alpha.
\end{equation}
Then necessarily
\begin{equation}\label{eq:ball net conclusion}
\frac{\log n}{\log r} n^{\frac{1}{2\dim(X)}}\lesssim \alpha \sqrt{\gamma(\G,\|\cdot\|_X^2)}.
\end{equation}
\end{lemma}

Prior to proving Lemma~\ref{lem:ABN referee}, we will derive some of its corollaries.  For the terminology of Corollary~\ref{cor:nontrivial type} below, recall  that a Banach space $Y$ is said to be $B$-convex~\cite{Bec62} if $\ell_1$ is {\em not} finitely representable in $Y$; see the survey~\cite{Mau03} for more on this important notion, including useful analytic, geometric and probabilistic  characterizations.

\begin{corollary}\label{cor:nontrivial type} There is a universal constant $C\in (0,\infty)$ for which the following assertion holds true. Let $Y$ be an infinite dimensional $B$-convex Banach space. For arbitrarily large $n\in \N$, if $\alpha\ge C\log n$, then we have
\begin{equation}\label{eq:K-convex dim reduction}
\kk_n^\alpha(\ell_\infty,Y)\asymp_Y \frac{\log n}{\log\left(\frac{\alpha}{\log n}\right)}.
\end{equation}
Thus, we have  in particular $\kk_n^{C\log n}(\ell_\infty,Y)\asymp_Y\log n$.
\end{corollary}

\begin{proof} The upper bound $\kk_n^\alpha(\ell_\infty,Y)\lesssim(\log n)/\log(\alpha/\log n)$ actually holds for {\em any} infinite dimensional Banach space $Y$. Indeed, by Bourgain's embedding theorem~\cite{Bou85} any $n$-point metric space $\MM$ admits an embedding $f$ into $\ell_2$ with distortion $A\log n$, where $A\in (0,\infty)$ is a universal constant. If $\alpha\ge 4A\log n$, then by applying the JL-Lemma~\cite{JL84} we know that $f(\MM)$ embeds with distortion $\alpha/(2A\log n)$ into $\ell_2^k$, where $k\lesssim (\log n)/\log(\alpha/\log n)$. By Dvoretzky's theorem~\cite{Dvo60}, we know that $\ell_2^k$ embeds with distortion $2$ into $Y$, so overall we obtain an embedding of $\MM$ into a $k$-dimensional subspace of $Y$ with distortion at most $2(A\log n)(\alpha/(2A\log n))=\alpha$.

Conversely, suppose that $\alpha\ge 2\log n$ and that $Y$ is a $B$-convex Banach space. By a theorem of V. Lafforgue~\cite{Laf09}  (see also~\cite{MN14} for a different approach), for arbitrarily large $n\in \N$ there is a $O(1)$-regular graph $\G=(\n,\EE_\G)$ such that $\gamma(\G,\|\cdot\|_Y^2)\lesssim_Y 1$. If $(\n,d_\G)$ embeds with distortion $\alpha$ into a $k$-dimensional subspace of $Y$, then by Lemma~\ref{lem:ABN referee} we have $n^{1/(2k)}\lesssim_Y \alpha/\log n$. Thus $k\gtrsim_Y (\log n)/\log(\alpha/\log n)$, as required.
\end{proof}

\begin{question} Is the assumption of $B$-convexity needed for the conclusion~\eqref{eq:K-convex dim reduction} of Corollary~\ref{cor:nontrivial type}? Perhaps finite cotype suffices for this purpose? This matter is of course closely related to Question~\ref{Q:log n}.
\end{question}

\begin{corollary}\label{coro:combine ball with spectral} Under the assumptions and notation of Lemma~\ref{lem:ABN referee}, we have
$$
\frac{\log n}{\log r} n^{\frac{1}{2\dim(X)}}\lesssim  \frac{\alpha\log(\cc_2(X)+1)}{1-\lambda_2(\G)}.
$$
\end{corollary}

\begin{proof} This is  nothing more than a substitution of~\eqref{eq:M-d version} with $M=\cc_2(X)\sqrt{1-\lambda_2(\G)}$ into~\eqref{eq:ball net conclusion}.
\end{proof}

Since by John's theorem~\cite{Joh48} we have $\cc_2(X)\le \sqrt{\dim(X)}$ and for every $n\in \N$ there exists a graph $\G$ as in Lemma~\ref{lem:ABN referee} with $r=O(1)$ and $\lambda_2(\G)=1-\Omega(1)$, it follows from Corollary~\ref{coro:combine ball with spectral} that
\begin{equation*}
\alpha \log(\kk_n^\alpha(\ell_\infty)+1)\gtrsim n^{\frac{1}{2\kk_n^\alpha(\ell_\infty)}}\log n.
\end{equation*}
This implies the lower bound on $\kk_n^\alpha(\ell_\infty)$ in~\eqref{eq:linfty bounds in theorem}. In particular, for $\alpha\asymp \log n$ it gives the first inequality in~\eqref{eq:log log log}.

\begin{proof}[Proof of Lemma~\ref{lem:ABN referee}] Denote $\gamma=\gamma(\G,\|\cdot\|_X^2)$. For  $i\in \n$ write
\begin{equation}\label{eq:def Ui}
\mathscr{U}_i=f^{-1}\Big(B_X\big(f(i),\alpha\sqrt{2\gamma}\big)\Big)=\Big\{j\in \n:\ \|f(i)-f(j)\|_X\le \alpha\sqrt{2\gamma}\Big\}.
\end{equation}
Let $m\in \n$ satisfy $|\mathscr{U}_m|=\max_{i\in \n}|\mathscr{U}_i|$.  Then
\begin{multline}\label{eq:use point again rho}
n^2\gamma\alpha^2\stackrel{\eqref{eq:def nonlinear gap}\wedge \eqref{eq:non contracting}}{\ge} \sum_{i=1}^n\sum_{j=1}^n  \|f(i)-f(j)\|_X^2\ge \sum_{i=1}^n\sum_{j\in \n\setminus \mathscr{U}_i} \|f(i)-f(j)\|_X^2\\\stackrel{\eqref{eq:def Ui}}> \sum_{i=1}^n \big(\alpha\sqrt{2\gamma}\big)^2(n-|\mathscr{U}_i|)\ge 2n\gamma\alpha^2(n-|\mathscr{U}_m|).
\end{multline}
This simplifies to $|\mathscr{U}_m|\ge \frac12n$. Also, since $\frac{1}{n^2}\sum_{i=1}^n\sum_{j=1}^n  \|f(i)-f(j)\|_X^2\ge \frac{1}{n^2}\sum_{i=1}^n\sum_{j=1}^n d_\G(i,j)^2\gtrsim (\log_r n)^2$,  the first inequality in~\eqref{eq:use point again rho} implies the a priori lower bound $\alpha\sqrt{\gamma}\gtrsim \log _r n$.

Next, fix $\rho\in (0,\infty)$ and let $\mathscr{N}_{2\rho}\subset \mathscr{U}_m$ be a maximal (with respect to inclusion) $2\rho$-separated subset of $\mathscr{U}_m$. Then $\mathscr{U}_m\subset \cup_{i\in \mathscr{N}_\Delta} B_\G(i,2\rho)$, where $B_\G(i,2\rho)$ denotes the ball centered at $i$ of radius $2\rho$ in the shortest-path metric $d_\G$. Since $\G$ is $r$-regular, for each $i\in \n$  we have the (crude) bound $|B_\G(i,2\rho)|\le 2r^{2\rho}$. Hence, $\frac12n\le|\mathscr{U}_m|\le 2r^{2\rho} |\mathscr{N}_{2\rho}|$. So, if we choose $\rho= \frac14\log_r n$, then $|\mathscr{N}_{2\rho}|\gtrsim \sqrt{n}$. Since by~\eqref{eq:non contracting}  distinct $i,j\in \mathscr{N}_{2\rho}$ satisfy $\|f(i)-f(j)\|_X\ge d_\G(i,j)\ge 2\rho$, the $X$-balls $\{B_X(f(i),\rho):\ i\in\mathscr{N}_{2\rho}\}$ have pairwise disjoint interiors. At the same time, since each $i\in \mathscr{N}_{2\rho}$ belongs to $\mathscr{U}_m$, we have $\|f(i)-f(m)\|_X\le \alpha\sqrt{2\gamma}$ (by the definition of $\mathscr{U}_m$), and hence $B_X(f(i), \rho)\subset B_X(f(m),\alpha\sqrt{2\gamma}+\rho)$. So, writing $\dim(X)=k$, we have the following  volume comparison.
\begin{multline*}
(\alpha\sqrt{2\gamma}+\rho)^{k}\vol_k\big(B_X(0,1)\big)=\vol_k\big(B_X(f(m),\alpha\sqrt{2\gamma}+\rho)\big)\ge \vol_k\bigg(\bigcup_{i\in \mathscr{N}_{2\rho}} B_X(f(i),\rho)\bigg)
\\=\sum_{i\in \mathscr{N}_{2\rho}}\vol_k\big(B_X(f(i),\rho)\big)=\rho^k\vol_k\big(B_X(0,1)\big)|\mathscr{N}_{2\rho}|\gtrsim
\rho^k\vol_k\big(B_X(0,1)\big)\sqrt{n}.
\end{multline*}
This simplifies to give
$
n^{\frac{1}{2k}}\lesssim \frac{\alpha\sqrt{2\gamma}}{\rho}+1\asymp \frac{\alpha\sqrt{\gamma}}{\log_r n},
$
where we used the definition of $\rho$, and that $\alpha\sqrt{\gamma}\gtrsim \log_r n$.
\end{proof}

\subsection{Nonlinear Rayleigh quotient inequalities}\label{sec:ray} Our goal in this section is to present a proof of~\eqref{eq:M-d version}. As we stated earlier, the proof that appears below is different from the proof of Theorem~\ref{thm:our gamma} in~\cite{Nao17}.  However,  the reason that underlies its validity is the same as that of the original argument in~\cite{Nao17}. Specifically, we arrived at the ensuing proof because we were driven by an algorithmic need that arose in~\cite{ANNRW18}. This need required proving a point-wise strengthening of an upper bound on nonlinear spectral gaps, which is called in~\cite{ANNRW18} a "nonlinear Rayleigh quotient inequality." We will clarify what we mean by this later;  a detailed discussion appears in~\cite{ANNRW18}.

The need to make the interpolation-based proof in~\cite{Nao17} constructive/algorithmic led us to merge the argument in~\cite{Nao17} with the {\em proof of} a theorem from~\cite{Nao14}, rather than quoting and using the latter as a "black box" as we did in~\cite{Nao17}. In doing so, we realized that for the purpose of obtaining only the weaker bound~\eqref{eq:M-d version} one could more efficiently combine~\cite{Nao14} and~\cite{Nao17} so as to skip the use of complex interpolation and to obtain the estimate~\eqref{eq:M-d version} as well as its  nonlinear Rayleigh quotient counterpart. Thus, despite superficial differences, the argument below amounts to unravelling the proofs in~\cite{Nao14,Nao17}  and removing steps that are needed elsewhere but not for~\eqref{eq:M-d version}. At present, we do not have a  proof of the stronger inequality~\eqref{eq:M version} that differs from its proof  in~\cite{Nao17}, and the interpolation-based approach of~\cite{Nao17} is used for more refined algorithmic results in the forthcoming work~\cite{ANNRW18-uniform}.

We will continue using the notation/conventions that were set at the beginning of Section~\ref{sec:average}. Fix $p\ge 1$ and a metric space $(\MM,d_\MM)$. Let $L_p(\pi;\MM)$ be the metric space $(\MM^n,d_{L_p(\pi;\MM)})$, where $d_{L_p(\pi;\MM)}:\MM^n\times \MM^n\to [0,\infty)$ is
$$
\forall\,x=(x_1,\ldots,x_n),y=(y_1,\ldots,y_n)\in \MM^n,\qquad d_{L_p(\pi;\MM)}(x,y)\eqdef \bigg(\sum_{i=1}^n \pi_id_\MM(x_i,y_i)^p \bigg)^{\frac{1}{p}}.
$$
Throughout what follows, it will be notationally convenient to slightly abuse notation by considering $\MM$ as a subset of $L_p(\pi;\MM)$ through its identification with the {\em diagonal} subset of $\MM^n$, which is an isometric copy of $\MM$ in $L_p(\pi;\MM)$. Namely,  we identify each $x\in \MM$ with the $n$-tuple $(x,x\ldots,x)\in \MM^n$.

If $x=(x_1,\ldots,x_n)\in L_p(\pi;\MM)\setminus \MM$, then the corresponding {\em nonlinear Rayleigh quotient} is defined to be
\begin{equation}\label{eq:ray def}
\ray(x;\mA,d_\MM^p)\eqdef \frac{\sum_{i=1}^n\sum_{j=1}^n\pi_i  a_{ij}d_\MM(x_i,x_j)^p}{\sum_{i=1}^n\sum_{j=1}^n \pi_i\pi_j d_\MM(x_i,x_j)^p}.
\end{equation}
The restriction $x\notin\MM$ was made here only to ensure that the denominator in~\eqref{eq:ray def} does not vanish. By definition,
\begin{equation}\label{eq:gamma ray relation}
\gamma(\mA,d_\MM^p)=\sup_{x\in L_p(\pi;\MM)\setminus \MM} \frac{1}{\ray(x;\mA,d_\MM^p)}.
\end{equation}

Note that $L_p(\pi;X)$ is a Banach space for every Banach space $(X,\|\cdot\|_X)$. In this case,  the matrix $\AA\in \M_n(\R)$ induces a linear operator $\AA\otimes \Id_X:L_p(\pi;X)\to L_p(\pi;X)$ that is given by $(\AA\otimes \Id_X)(x_1,\ldots x_n)=(\sum_{j=1}^n a_{ij}x_j)_{i=1}^n$.

The following lemma records some simple and elementary general properties of nonlinear Rayleigh quotients.

\begin{comment}
Its third assertion is a simple a priori upper bound. We won't use this upper bound in the ensuing proof, but it is worthwhile to be aware of it so as to indicate that the  pertinent issue is to bound $\ray(x;A,d_\MM^p)$ from below, which by~\eqref{eq:gamma ray relation} corresponds to bounding $\gamma(\mA,d_\M^p)$ from above. Nevertheless, it might be meaningful to obtain better a priori upper bounds under additional geometric assumptions on the underlying metric space $(\MM,d_\MM)$; e.g., if the underlying space is a Hilbert space and $p=2$, then one can improve the stated bound using the spectral interpretation of Hilbertain Rayleigh quotients.
\end{comment}

\begin{lemma}\label{lem:Mtype2} Suppose that $(\MM,d_\MM)$ is a metric space, $n\in \N$, $p\in [1,\infty)$ and $\d\in [0,1]$. Let $\pi=(\pi_1,\ldots,\pi_n)$ be a probability measure on $\n$ and $\mA,\mB\in \M_n(\R)$ be  row-stochastic matrices that are reversible with respect to $\pi$. For any $x\in L_p^n(\pi;\MM)\setminus \MM$  we have
\begin{enumerate}
\item  $\ray \big(x;\d\mA+(1-\d)\mB,d_\MM^p\big)=\d\ray \big(x;\mA,d_\MM^p\big)+(1-\d)\ray \big(x;\mB,d_\MM^p\big)$.

    \item $ \ray \big(x;(1-\d)\Id_n+\d\mA,d_\MM^p\big)=\d\ray \big(x;\mA,d_\MM^p\big)$, where $\Id_n\in \M_n(\R)$ is the identity matrix.

\item $
\ray(x;\mA,d_\MM^p)\le 2^p$.
\item $
\ray(x;\mA\mB,d_\MM^p)^{\frac{1}{p}}\le \ray(x;\mA,d_\MM^p)^{\frac{1}{p}}+\ray(x;\mB,d_\MM^p)^{\frac{1}{p}}$.

\item $\ray \big(x;\mA^\tt,d_\MM^p\big)\le \tt^p\ray \big(x;\mA,d_\MM^p\big)$ for every $\tt\in \N$.
\end{enumerate}
\end{lemma}

\begin{proof} The first assertion is an immediate consequence of the definition of nonlinear Rayleigh quotients. The second assertion is a special case of the first assertion, since by definition $\ray(x;\Id_n,d_\MM^p)=0$. The third assertion is justified by noting that by the triangle inequality, for every $i,j,k\in \n$ we have
\begin{equation}\label{eq:to bound rayleigh lower}
d_\MM(x_i,x_j)^p\le \big(d_\MM(x_i,x_k)+d_\MM(x_k,x_j)\big)^p \le 2^{p-1}d_\MM(x_i,x_k)^p+2^{p-1}d_\MM(x_k,x_j)^p.
\end{equation}
where the last step of~\eqref{eq:to bound rayleigh lower} uses the convexity of the function $(t\in [0,\infty))\mapsto t^p$. By multiplying~\eqref{eq:to bound rayleigh lower} by $\pi_i\pi_k a_{ij}$, summing over $i,j,k\in \n$ and using the fact that $\mA$ is reversible with respect to $\pi$, we get
$$
\sum_{i=1}^n\sum_{j=1}^n \pi_i a_{ij} d_\MM(x_i,x_j)^p\le 2^p\sum_{i=1}^n\sum_{j=1}^n\pi_i\pi_j d_\MM(x_i,x_j)^p.
$$
Recalling the notation~\eqref{eq:ray def}, this is precisely the third assertion of Lemma~\ref{lem:Mtype2}.

It remains to justify the fourth assertion of Lemma~\ref{lem:Mtype2}, because its fifth assertion follows from iterating its fourth assertion $\tt-1$ times (with $\mB$ a power of $\mA$). To this end, writing $\mA=(a_{ij})$ and $\mB=(b_{ij})$, we have
\begin{align}
\nonumber \bigg(\sum_{i=1}^n\sum_{j=1}^n \pi_i(\mA\mB)_{ij}d_\MM(x_i,x_j)^p\bigg)^{\frac{1}{p}}  &\le  \bigg(\sum_{i=1}^n\sum_{j=1}^n \pi_i\bigg(\sum_{k=1}^n a_{ik}b_{kj}\bigg)\big(d_\MM(x_i,x_k)+d_\MM(x_k,x_j)\big)^p\bigg)^{\frac{1}{p}}\\ \nonumber
&\le \bigg(\sum_{i=1}^n\sum_{j=1}^n \sum_{k=1}^n \pi_i a_{ik}b_{kj}d_\MM(x_i,x_k)^p\bigg)^{\frac{1}{p}}+\bigg(\sum_{i=1}^n\sum_{j=1}^n \sum_{k=1}^n \pi_i a_{ik}b_{kj}d_\MM(x_k,x_j)^p\bigg)^{\frac{1}{p}}
\\&= \bigg(\sum_{i=1}^n \sum_{k=1}^n\pi_i a_{ik}d_\MM(x_i,x_k)^p\bigg)^{\frac{1}{p}}+\bigg(\sum_{j=1}^n \sum_{k=1}^n \pi_kb_{kj}d_\MM(x_k,x_j)^p\bigg)^{\frac{1}{p}},\label{remove A and B resp}
\end{align}
where the first step of~\eqref{remove A and B resp}  uses the triangle inequality in $\MM$, the second step of~\eqref{remove A and B resp} uses the triangle inequality in  $L_p(\mu)$ with $\mu$ being the measure on $\n^3$ given by $\mu(i,j,k)=\pi_ia_{ik}b_{kj}$ for all $i,j,k\in \n$, and the final step of~\eqref{remove A and B resp} uses the fact that $\mA$ and $\mB$ are both row-stochastic and reversible with respect to $\pi$.
\end{proof}

The identity in the following claim is a consequence of a very simple and standard Hilbertian computation that we record here for ease of later references.

\begin{claim}\label{claim:hilbertial rayleigh} For every Hilbert space $(H,\|\cdot\|_H)$ and every $x\in L_2(\pi;H)\setminus H$ we have $\ray(x;\mA^2,\|\cdot\|_H^2)\le 1$. Moreover, if $\sum_{i=1}^n \pi_i x_i=0$, then
\begin{equation*}\label{eq:square identity}
\frac{\left\|(\mA\otimes \Id_H)x\right\|_{L_2(\pi;H)}}{\|x\|_{L_2(\pi;H)}}=\sqrt{1-\ray(x;\mA^2,\|\cdot\|_H^2)}.
\end{equation*}
\end{claim}

\begin{proof} Let $\langle\cdot,\cdot\rangle:H\times H\to \R$ be the scalar product that induces the Hilbertian norm $\|\cdot\|_H$. Then, the scalar product that induces the norm $\|\cdot\|_{L_2(\pi;H)}$ is given by $\langle y,z\rangle_{L_2(\pi;H)} =\sum_{i=1}^n \pi_i\langle y_i,z_i\rangle$. By expanding the squares while using the fact that $\mA$ is row-stochastic, reversible relative to $\pi$, and $\sum_{i=1}^n \pi_i x_i=0$, we get that
\begin{multline*}
\sum_{i=1}^n\sum_{j=1}^n \pi_i(\mA^2)_{ij}\|x_i-x_j\|_H^2=2\|x\|_{L_2(\pi;H)}^2-2\sum_{i=1}^n \pi_i\bigg\langle  x_i,\sum_{j=1}^n (\mA^2)_{ij}x_j\bigg\rangle\\ =2\|x\|_{L_2(\pi;H)}^2-2 \big\langle x,(\mA^2\otimes \Id_H)x\big\rangle_{L_2(\pi;H)}=2\|x\|_{L_2(\pi;H)}^2-2\big\|(\mA\otimes \Id_H)x\big\|_{L_2(\pi;H)}^2,
\end{multline*}
and
$$
\sum_{i=1}^n\sum_{j=1}^n\pi_i\pi_j\|x_i-x_j\|_H^2=2\sum_{i=1}^n \pi_i\|x_i\|_H^2-2\bigg\|\sum_{i=1}^n \pi_i x_i\bigg\|_H^2=2\|x\|_{L_2(\pi;H)}^2.
$$
Therefore, recalling the definition~\eqref{eq:ray def}, we have
\begin{equation*}
\ray(x;\mA^2,\|\cdot\|_H^2)=1-\frac{\big\|(\mA\otimes \Id_H)x\big\|_{L_2(\pi;H)}^2}{\|x\|_{L_2(\pi;H)}^2}\le 1. \qedhere
\end{equation*}
\end{proof}

\begin{lemma}[point-wise Rayleigh quotient estimate for Hilbert isomorphs]\label{lem:power Rayleigh}Let $(X,\|\cdot\|_X)$ be a normed space and fix $\dd\in [1,\infty)$. Suppose that $\|\cdot\|_H:X\to [0,\infty)$  is  a Hilbertian norm on $X$ that satisfies
\begin{equation}\label{eq:d assumption}
\forall\, y\in X,\qquad \|y\|_H\le \|y\|_X\le \dd\|y\|_H.
\end{equation}
For every $x\in L_2(\pi;X)\setminus X$ define a quantity  $\tt(x,\mA)=\tt(x;\mA,\|\cdot\|_H,\dd)$ to be the minimum $\tt\in \N$ such that
\begin{equation}\label{eq:Euclidean ray condition}
\ray\bigg(x;\Big(\frac12\Id_n+\frac12\mA\Big)^{\!2\tt},\|\cdot\|_H^2\bigg)\ge 1-\frac{1}{4\dd^2},
\end{equation}
with the convention that $\tt(x;\mA)=\infty$ if no such $\tt$ exists. Then,
\begin{equation}\label{eq:t2}
\frac{1}{\ray(x;\mA,\|\cdot\|_X^2)}\lesssim\tt(x;\mA)^2.
\end{equation}
\end{lemma}

\begin{proof} We may assume without loss of generality that $\sum_{i=1}^n\pi_i x_i=0$ and $\tt(x;\mA)<\infty$. Define a matrix
\begin{equation}\label{eq:def mB}
\mB_x\eqdef \Big(\frac12\Id_n+\frac12\mA\Big)^{\!\tt(x;\mA)}\in \M_n(\R).
\end{equation}
Then $\mB_x$ is also a row-stochastic matrix which is reversible with respect to $\pi$, and, by the definition of $\tt(x;\mA)$,
$$
\ray(x;\mB^2_x,\|\cdot\|_H^2)\ge 1-\frac{1}{4\dd^2}.
$$
By Claim~\ref{claim:hilbertial rayleigh}, since $\sum_{i=1}^n\pi_ix_i=0$, this implies that
\begin{equation}\label{eq:hilbertian norm quotient}
\frac{\left\|(\mB_x\otimes \Id_H)x\right\|_{L_2(\pi;H)}}{\|x\|_{L_2(\pi;H)}}\le \sqrt{1-\left(1-\frac{1}{4\dd^2}\right)}=\frac{1}{2\dd}.
\end{equation}
At the same time, due to~\eqref{eq:d assumption} we have
\begin{equation}\label{eq:comparison of norm quotients}
\frac{\left\|(\mB_x\otimes \Id_X)x\right\|_{L_2(\pi;X)}}{\|x\|_{L_2(\pi;X)}}\le \dd\frac{\left\|(\mB_x\otimes \Id_H)x\right\|_{L_2(\pi;H)}}{\|x\|_{L_2(\pi;H)}}.
\end{equation}
By combining~\eqref{eq:hilbertian norm quotient} and~\eqref{eq:comparison of norm quotients} we see that $\left\|(\mB_x\otimes \Id_H)x\right\|_{L_2(\pi;X)}\le\frac12 \|x\|_{L_2(\pi;X)}$. Consequently,
\begin{equation}\label{eq:lower I-B}
\left\|x-\left(\mB_x\otimes \Id_X\right)x\right\|_{L_2(\pi;X)}\ge \|x\|_{L_2(\pi;X)}-\left\|\left(\mB_x\otimes \Id_X\right)x\right\|_{L_2(\pi;X)}\ge \frac12 \|x\|_{L_2(\pi;X)}.
\end{equation}

Observe that
\begin{equation}\label{eq:variance}
\bigg(\sum_{i=1}^n\sum_{j=1}^n \pi_i\pi_j\|x_i-x_j\|_X^2\bigg)^{\frac12}\le \bigg(\sum_{i=1}^n\sum_{j=1}^n \pi_i\pi_j \big(\|x_i\|_X+\|x_j\|_X\big)^2\bigg)^{\frac12}\le 2\|x\|_{L_2(\pi;X)},
\end{equation}
where in the first step of~\eqref{eq:variance} we used the triangle inequality in $X$ and the second step of~\eqref{eq:variance} is an application of the triangle inequality in $L_2(\pi\otimes \pi)$.  Also, since $\mB_x$ is row-stochastic,
\begin{equation}\label{eq:convexity B}
\left\|x-\left(\mB_x\otimes \Id_X\right)x\right\|_{L_2(\pi;X)}=\bigg(\sum_{i=1}^n \pi_i \bigg\|\sum_{j=1}^n (\mB_x)_{ij}(x_i-x_j)\bigg\|_X^2\bigg)^{\frac12}\le \bigg(\sum_{i=1}^n \sum_{j=1}^n \pi_i(\mB_x)_{ij}\|x_i-x_j\|_X^2\bigg)^{\frac12},
\end{equation}
where in the final step of~\eqref{eq:convexity B} we used the convexity of the function $\|\cdot\|_X^2:X\to \R$.

Recalling the definition~\eqref{eq:ray def}, by substituting~\eqref{eq:variance} and~\eqref{eq:convexity B} into~\eqref{eq:lower I-B} we see that
\begin{equation}\label{eq:power is expander}
\ray\bigg(x;\Big(\frac12\Id_n+\frac12\mA\Big)^{\tt(x;\mA)},\|\cdot\|_X^2\bigg)\stackrel{\eqref{eq:def mB}}{=}\ray\big(x;\mB_x,\|\cdot\|_X^2\big)\ge \frac{1}{16}.
\end{equation}
We now conclude the proof of the desired estimate~\eqref{eq:t2} as follows.
\begin{equation}\label{eq:use mtype}
1\stackrel{\eqref{eq:power is expander}}{\lesssim}\ray\bigg(x;\Big(\frac12\Id_n+\frac12\mA\Big)^{\tt(x;\mA)},\|\cdot\|_X^2\bigg)\le \tt(x;\mA)^2\ray\bigg(x;\frac12\Id_n+\frac12\mA,\|\cdot\|_X^2\bigg)=\frac12\tt(x;\mA)^2\ray\big(x;\mA,\|\cdot\|_X^2\big),
\end{equation}
where the second step of~\eqref{eq:use mtype} uses the fifth assertion of Lemma~\ref{lem:Mtype2}, and the final step uses  its second assertion.
\end{proof}

The quantity $\tt(x;\mA)$ of Lemma~\ref{lem:power Rayleigh} can be bounded as follows in terms of the spectral gap of $\AA$.
\begin{lemma} Continuing with the notation of Lemma~\ref{lem:power Rayleigh}, the following estimate holds true.
\begin{equation}\label{eq:t upper}
\tt(x;\mA)\le \left\lceil \frac{\log(2\dd)}{\log\left(\frac{2}{1+\lambda_2(\mA)}\right)}\right\rceil\lesssim \frac{\log(2\dd)}{1-\lambda_2(\mA)}.
\end{equation}
\end{lemma}

\begin{proof}  Since $\AA$ is row-stochastic, $\lambda_n(\AA)\ge -1$. Therefore  $\frac12\Id_n+\frac12\AA$ is a positive semidefinite self-adjoint operator on $L_2(\pi)$ that preserves the hyperplane $L_2^0(\pi)=\{u\in \R^n;\ \sum_{i=1}^n\pi_i u_i=0\}$. The largest eigenvalue of $\frac12\Id_n+\frac12\AA$ on $L_2^0(\pi)$ is $\frac12 +\frac12 \lambda_2(\AA)$, and therefore $\|(\frac12\Id_n+\frac12\AA)^\tt u\|_{L_2(\pi)}\le (\frac12 +\frac12 \lambda_2(\AA))^\tt\|u\|_{L_2(\pi)}$ for  $u\in L_2^0(\pi)$ and $\tt\in \N$.

 If $x\in L_2(\pi;X)\setminus X$ satisfies $\sum_{i=1}^n\pi_ix_i=0$, then we may apply the above observation to the coordinates of $x$ with respect to some orthonormal basis of $H$, each of which is an element of $L_2^0(\pi)$, and deduce that
\begin{equation}\label{eq:in L_2H}
 \Big(\frac12 +\frac12 \lambda_2(\AA)\Big)^\tt\ge \frac{\left\|\left(\left(\frac12\Id_n+\frac12\AA\right)^\tt \otimes \Id_H\right) x\right\|_{L_2(\pi;H)}}{\|x\|_{L_2(\pi;H)}}=\sqrt{1-\ray\bigg(x;\Big(\frac12\Id_n+\frac12\mA\Big)^{\!2\tt},\|\cdot\|_H^2\bigg)},
 \end{equation}
 where in the second step of~\eqref{eq:in L_2H} we applied Lemma~\ref{claim:hilbertial rayleigh} with $\AA$ replaced by $(\frac12\Id_n+\frac12\AA)^\tt$. Hence
 $$
 \ray\bigg(x;\Big(\frac12\Id_n+\frac12\mA\Big)^{\!2\tt},\|\cdot\|_H^2\bigg)\ge 1- \Big(\frac12 +\frac12 \lambda_2(\AA)\Big)^{\!2\tt}.
 $$
Consequently, if $\tt\ge (\log(2\dd))/\log(2/(\lambda_2(\AA)+1))$, then $ \ray(x;(\frac12\Id_n+\frac12\mA)^{\!2\tt},\|\cdot\|_H^2)\ge 1-\frac{1}{4\dd^2}$. By the definition of $\tt(x;\AA)$, this implies the first inequality in~\eqref{eq:t upper}. The second inequality in~\eqref{eq:t upper} follows by elementary calculus.
\end{proof}

\begin{proof}[Proof of~\eqref{eq:M-d version}] By a classical linearization argument~\cite{Enf70} (see~\cite[Chapter~7]{BL} for a modern treatment), for every $\dd>\cc_2(X)$ there is a Hilbertian norm $\|\cdot\|_H$ on $X$ that satisfies~\eqref{eq:d assumption}. We therefore see that for every $\AA$ as above
\begin{multline}\label{eq:deduce slightly weaker}
\sqrt{\gamma(\AA,\|\cdot\|_X^2)}\stackrel{\eqref{eq:gamma ray relation}}{=}\sup_{x\in L_2(\pi;X)\setminus X} \frac{1}{\sqrt{\ray(x;\mA,\|\cdot\|_X^2)}}\stackrel{\eqref{eq:t2}}{\le} \sup_{x\in L_2(\pi;X)\setminus X} \tt(x;\AA)\\\stackrel{\eqref{eq:t upper}}{\lesssim} \frac{\log(\cc_2(X)+1)}{1-\lambda_2(\AA)}=\frac{\log(\cc_2(X)+1)}{\sqrt{1-\lambda_2(\AA)}}\sqrt{\gamma(\AA,\|\cdot\|_H^2)}\le \frac{\log (\cc_2(X)+1)}{M}\cc_2(X)\sqrt{\gamma(\AA,\|\cdot\|_H^2)},
\end{multline}
where, for the final step of~\eqref{eq:deduce slightly weaker} recall that in the context of~\eqref{eq:M-d version} we assume that $\lambda_2(\AA)\le 1-M^2/\cc_2(X)^2$.
\end{proof}

\subsubsection{Structural implications of nonlinear Rayleigh quotient inequalities} Fix integers $n,k,r\ge 3$ (think of $n$ as much larger than $k$). Let $(X,\|\cdot\|_X)$ be a $k$-dimensional normed space. Suppose that  $\G=(\n,\EE_\G)$ is a connected $r$-regular graph. Although we phrased (and used) Corollary~\ref{cor:regular case}  as an impossibility result that provides an obstruction (spectral gap) for faithfully realizing (on average) the metric space $(\n,d_\G)$ in $X$, a key insight of the recent work~\cite{ANNRW18} by Andoni, Nikolov, Razenshteyn, Waingarten and the author is that one could "flip" this point of view to deduce from Corollary~\ref{cor:regular case}  useful information on those graphs that do happen to admit such a faithful geometric realization in $X$, namely they satisfy~\eqref{eq:two averages-regular}. Clearly there are plenty of graphs with this property, including those graphs that arise from discrete approximations of subsets of $X$ (as a "vanilla" example to keep in mind,  fix a small parameter $\d>0$,  consider a $\d$-net in the unit ball of $X$ as the vertices, and join two net points by an edge if their distance in $X$ is $O(\d)$). The conclusion of Corollary~\ref{cor:regular case} for any such graph is that it cannot have a large spectral gap, and by Cheeger's inequality~\cite{Che70,Tan84,AM85,SJ89} it follows that this graph can be partitioned into two pieces with a small (relative) "discrete boundary." On the other hand, if we are given a mapping $f:\n\to X$ that satisfies the first condition in~\eqref{eq:two averages-regular} but not the second condition in~\eqref{eq:two averages-regular}, then there must be a ball in $X$ of relatively small radius that contains a    substantial fraction of the vectors  $\{f(i)\}_{i=1}^n$. The partition of $\n$ that corresponds to this dense ball and its complement encodes useful geometric "clustering" information. We have thus observed a dichotomic behavior that allows one to partition geometrically-induced graphs using either a "spectral partition" or a "dense ball partition."

In~\cite{ANNRW18}, the above idea is used iteratively to construct a  hierarchical partition of $X$. Our overview  suppresses important technical steps, which include both randomization and a re-weighting procedure of the graphs that arise at later stages of the construction (the start of the construction is indeed the above "net graph"); see~\cite{ANNRW18} for the full details.  In particular, one needs to use general row-stochastic matrices due to the re-weighting procedure, i.e., one uses the full strength of Theorem~\ref{thm:pi version average embedding}  rather than only the case of graphs as in Corollary~\ref{cor:regular case}.

In summary, one can use the bound~\eqref{eq:M-d version} on nonlinear spectral gaps  to provide a "cutting rule" that governs an iterative partitioning procedure in which each inductive step is either geometric (a ball and its complement) or a less explicit existential step that follows from spectral information which is deduced from a contrapositive assumption of (rough, average) embeddability. This structural information is used in~\cite{ANNRW18} to design a new data structure for approximate nearest neighbor search in arbitrary norms (see the article of Andoni, Indyk and Razenshteyn in the present volume for an extensive account of approximate nearest neighbor search). Although this yields important (and arguably unexpected) progress on an algorithmic question of central importance, the non-explicitness and potential high complexity of the spectral partitioning step raises issues of efficiency that are not yet fully resolved. Specifically,  the most general data structure that is designed in~\cite{ANNRW18} is efficient only in the so-called "cell probe model," but not in the full polynomial-time sense; we refer to~\cite{ANNRW18} for an explanation of these  complexity-theoretic issues and their significance, because they are beyond the scope of the present article.

While the above issue of efficiency does not occur in our initial investigation within pure mathematics, it is very important from the algorithmic perspective. This is what initially led to the desire to obtain a nonlinear Rayleigh quotient inequality rather than to merely bound the nonlinear spectral gap, though (in hindsight) such inequalities are interesting from the mathematical perspective as well. We did not formally define what we mean by a "nonlinear Rayleigh quotient inequality" because there is some flexibility here, but the basic desire is, given $x_1,\ldots,x_n\in X$, to bound their Rayleigh quotient in $X$ by a Rayleigh quotient of points in a {\em Euclidean space}.

The inequality~\eqref{eq:t2} of Lemma~\ref{lem:power Rayleigh} is of the above form, because the parameter $\tt(x;\AA)$ defined by~\eqref{eq:Euclidean ray condition} involves only examining a certain Rayleigh quotient in a Hilbert space. It should be noted, however, that to date we have not succeeded to use the specific nonlinear Rayleigh quotient inequality of Lemma~\ref{lem:power Rayleigh} for algorithmic purposes (though with more work this may be possible), despite the fact that it was  found with this motivation in mind.

Other nonlinear Rayleigh quotient inequalities were obtained in~\cite{ANNRW18,ANNRW18-uniform} and used to address issues of algorithmic  efficiency. Very roughly,  the drawback of~\eqref{eq:t2} is that the matrix $\AA$ is changed in the Hilbertian Rayleigh quotient of~\eqref{eq:Euclidean ray condition} (the main problem is the potentially high power $2\tt$). A more directly algorithmically-useful nonlinear Rayleigh quotient inequality would be to change the point $x\in L_2(\pi;X)$ but not change the matrix $\AA$. Namely, suppose that we could control $\ray(x;\mA,\|\cdot\|_X^2)$ from below by a function of $\ray(\phi_\AA(x);\mA,\|\cdot\|_H^2)$, for some mapping $\phi_\AA:L_2(\pi;X)\to L_2(\pi;H)$. Nonlinear Rayleigh quotient inequalities of this type are proved in~\cite{ANNRW18,ANNRW18-uniform}, though the associated mappings $\phi_\AA$ turn out to be highly nonlinear and quite complicated.\footnote{Specifically, in~\cite{ANNRW18} such a mapping $\phi_\AA$ is constructed for  Schatten-von Neumann trace classes using the Brouwer fixed-point theorem and estimates from~\cite{Ric15}. In~\cite{ANNRW18-uniform}, $\phi_\AA$ is constructed for general normed spaces using, in addition to Brouwer's theorem, convex programming and (algorithmic variants of) complex interpolation. These lead to data structures that are efficient in all respects other than the "preprocessing stage," which at present remains potentially time-consuming due to the complexity of $\phi_\AA$.}

The upshot of the latter type of nonlinear Rayleigh quotient inequality is that if (due to  existence of a faithful embedding into $X$) we know  that $\ray(x;\mA,\|\cdot\|_X^2)$ is small, then it follows that also $\ray(\phi_\AA(x);\mA,\|\cdot\|_H^2)$ is small. The {\em proof of} Cheeger's inequality (via examination of level sets of the second eigenvector) would now provide a sparse "spectral partition" of $\n$ that has the following auxiliary structure: The partition is determined by thresholding one of the coordinates of $H$ (in some fixed orthonormal basis), namely the part to which each $i\in \n$ belongs depends only on whether the coordinate in question of the transformed vector $\phi_\AA(x)_i\in H$ is above or below a certain value. If in addition $(\AA,x)\mapsto \phi_\AA(x)$ has favorable computational properties (see~\cite{ANNRW18} for a formulation; roughly, what is important here is that after a "preprocessing step" one can decide quickly to which piece of the partition each $i\in \n$ belongs), then this would lead to  fast "query time."

The above description of the algorithmic role of nonlinear Rayleigh quotient inequalities is  impressionistic, but it conveys the core ideas while not delving into (substantial) details.  Such inequalities are interesting in their own right, partially because they necessitate making mathematical arguments constructive, thus leading to new proofs, as we did for~\eqref{eq:M-d version},  and also leading to intrinsically meaningful studies, such as obtaining~\cite{ANNRW18-uniform} algorithmic versions of existential statements that arise from the use of the maximum principle in complex interpolation.

\bigskip
\noindent{\bf Acknowledgements.} I am grateful to Daniel Kane for a conversation that led to a simplification of Proposition~\ref{prop:sigma nk}. I also thank Alexandros Eskenazis, Manor Mendel, Ilya Razenshteyn and Gideon Schechtman for their helpful suggestions.

\bibliographystyle{abbrv}
\bibliography{ICM2018}

\def\polhk#1{\setbox0=\hbox{#1}{\ooalign{\hidewidth
  \lower1.5ex\hbox{`}\hidewidth\crcr\unhbox0}}}
\begin{thebibliography}{100}

\bibitem{ABN11}
I.~Abraham, Y.~Bartal, and O.~Neiman.
\newblock Advances in metric embedding theory.
\newblock {\em Adv. Math.}, 228(6):3026--3126, 2011.

\bibitem{Ach03}
D.~Achlioptas.
\newblock Database-friendly random projections: {J}ohnson-{L}indenstrauss with
  binary coins.
\newblock {\em J. Comput. System Sci.}, 66(4):671--687, 2003.
\newblock Special issue on PODS 2001 (Santa Barbara, CA).

\bibitem{AC09}
N.~Ailon and B.~Chazelle.
\newblock The fast {J}ohnson-{L}indenstrauss transform and approximate nearest
  neighbors.
\newblock {\em SIAM J. Comput.}, 39(1):302--322, 2009.

\bibitem{AL13}
N.~Ailon and E.~Liberty.
\newblock An almost optimal unrestricted fast {J}ohnson-{L}indenstrauss
  transform.
\newblock {\em ACM Trans. Algorithms}, 9(3):Art. 21, 12, 2013.

\bibitem{Ale51}
A.~D. Aleksandrov.
\newblock A theorem on triangles in a metric space and some of its
  applications.
\newblock In {\em Trudy {M}at. {I}nst. {S}teklov., v 38}, Trudy Mat. Inst.
  Steklov., v 38, pages 5--23. Izdat. Akad. Nauk SSSR, Moscow, 1951.

\bibitem{Al86}
N.~Alon.
\newblock The number of polytopes, configurations and real matroids.
\newblock {\em Mathematika}, 33(1):62--71, 1986.

\bibitem{Alo03}
N.~Alon.
\newblock Problems and results in extremal combinatorics. {I}.
\newblock {\em Discrete Math.}, 273(1-3):31--53, 2003.
\newblock EuroComb'01 (Barcelona).

\bibitem{Alo09}
N.~Alon.
\newblock Perturbed identity matrices have high rank: proof and applications.
\newblock {\em Combin. Probab. Comput.}, 18(1-2):3--15, 2009.

\bibitem{AKNSS08}
N.~Alon, H.~Kaplan, G.~Nivasch, M.~Sharir, and S.~Smorodinsky.
\newblock Weak {$\epsilon$}-nets and interval chains.
\newblock In {\em Proceedings of the {N}ineteenth {A}nnual {ACM}-{SIAM}
  {S}ymposium on {D}iscrete {A}lgorithms}, pages 1194--1203. ACM, New York,
  2008.

\bibitem{AK17}
N.~Alon and B.~Klartag.
\newblock Optimal compression of approximate inner products and dimension
  reduction.
\newblock In {\em 58th {IEEE} Annual Symposium on Foundations of Computer
  Science, {FOCS} 2017, Berkeley, CA, USA, October 15-17, 2017}, pages
  639--650, 2017.

\bibitem{AM85}
N.~Alon and V.~D. Milman.
\newblock {$\lambda_1,$} isoperimetric inequalities for graphs, and
  superconcentrators.
\newblock {\em J. Combin. Theory Ser. B}, 38(1):73--88, 1985.

\bibitem{AGS08}
L.~Ambrosio, N.~Gigli, and G.~Savar\'e.
\newblock {\em Gradient flows in metric spaces and in the space of probability
  measures}.
\newblock Lectures in Mathematics ETH Z\"urich. Birkh\"auser Verlag, Basel,
  second edition, 2008.

\bibitem{ACNNH11}
A.~Andoni, M.~S. Charikar, O.~Neiman, and H.~L. Nguyen.
\newblock Near linear lower bound for dimension reduction in {$\ell_1$}.
\newblock In {\em 2011 {IEEE} 52nd {A}nnual {S}ymposium on {F}oundations of
  {C}omputer {S}cience---{FOCS} 2011}, pages 315--323. IEEE Computer Soc., Los
  Alamitos, CA, 2011.

\bibitem{ANN15}
A.~Andoni, A.~Naor, and O.~Neiman.
\newblock Snowflake universality of {W}asserstein spaces.
\newblock To appear in {\em Ann. Sci. \'Ec. Norm. Sup\'er. (4)}. Preperint
  available at~\url{https://arxiv.org/abs/1509.08677}, 2015.

\bibitem{ANN16-ICALP}
A.~Andoni, A.~Naor, and O.~Neiman.
\newblock Impossibility of sketching of the 3{D} transportation metric with
  quadratic cost.
\newblock In {\em 43rd {I}nternational {C}olloquium on {A}utomata, {L}anguages,
  and {P}rogramming}, volume~55 of {\em LIPIcs. Leibniz Int. Proc. Inform.},
  pages Art. No. 83, 14. Schloss Dagstuhl. Leibniz-Zent. Inform., Wadern, 2016.

\bibitem{ANN17}
A.~Andoni, A.~Naor, and O.~Neiman.
\newblock On isomorphic dimension reduction in $\ell_1$.
\newblock Preprint, 2017.

\bibitem{ANNRW18}
A.~Andoni, A.~Naor, A.~Nikolov, I.~Razenshteyn, and E.~Waingarten.
\newblock Data-dependent hashing via nonlinear spectral gaps.
\newblock To appear in {\em S{TOC}'18---{P}roceedings of the 50th {A}nnual
  {ACM} {SIGACT} {S}ymposium on {T}heory of {C}omputing}, 2018.

\bibitem{ANNRW18-uniform}
A.~Andoni, A.~Naor, A.~Nikolov, I.~Razenshteyn, and E.~Waingarten.
\newblock H\"older homeomorphisms and approximate nearest neighbors.
\newblock Forthcoming manuscript, 2018.

\bibitem{ANNRW17}
A.~Andoni, H.~L. Nguyen, A.~Nikolov, I.~Razenshteyn, and E.~Waingarten.
\newblock Approximate near neighbors for general symmetric norms.
\newblock In {\em S{TOC}'17---{P}roceedings of the 49th {A}nnual {ACM} {SIGACT}
  {S}ymposium on {T}heory of {C}omputing}, pages 902--913. ACM, New York, 2017.

\bibitem{ABV98}
J.~Arias-de Reyna, K.~Ball, and R.~Villa.
\newblock Concentration of the distance in finite-dimensional normed spaces.
\newblock {\em Mathematika}, 45(2):245--252, 1998.

\bibitem{AR92}
J.~Arias-de Reyna and L.~Rodr\'iguez-Piazza.
\newblock Finite metric spaces needing high dimension for {L}ipschitz
  embeddings in {B}anach spaces.
\newblock {\em Israel J. Math.}, 79(1):103--111, 1992.

\bibitem{ALN08}
S.~Arora, J.~R. Lee, and A.~Naor.
\newblock Euclidean distortion and the sparsest cut.
\newblock {\em J. Amer. Math. Soc.}, 21(1):1--21 (electronic), 2008.

\bibitem{AV99}
R.~I. Arriaga and S.~Vempala.
\newblock An algorithmic theory of learning: robust concepts and random
  projection.
\newblock In {\em 40th {A}nnual {S}ymposium on {F}oundations of {C}omputer
  {S}cience ({N}ew {Y}ork, 1999)}, pages 616--623. IEEE Computer Soc., Los
  Alamitos, CA, 1999.

\bibitem{Ass83}
P.~Assouad.
\newblock Plongements lipschitziens dans {${\bf R}^{n}$}.
\newblock {\em Bull. Soc. Math. France}, 111(4):429--448, 1983.

\bibitem{AN17}
T.~Austin and A.~Naor.
\newblock On the bi-{L}ipschitz structure of {W}asserstein spaces.
\newblock Preprint, 2017.

\bibitem{ANT13}
T.~Austin, A.~Naor, and R.~Tessera.
\newblock Sharp quantitative nonembeddability of the {H}eisenberg group into
  superreflexive {B}anach spaces.
\newblock {\em Groups Geom. Dyn.}, 7(3):497--522, 2013.

\bibitem{Bal92}
K.~Ball.
\newblock Markov chains, {R}iesz transforms and {L}ipschitz maps.
\newblock {\em Geom. Funct. Anal.}, 2(2):137--172, 1992.

\bibitem{Bal13}
K.~Ball.
\newblock The {R}ibe programme.
\newblock {\em Ast\'erisque}, (352):Exp. No. 1047, viii, 147--159, 2013.
\newblock S\'eminaire Bourbaki. Vol. 2011/2012. Expos\'es 1043--1058.

\bibitem{Ban32}
S.~Banach.
\newblock {\em Th\'eorie des op\'erations lin\'eaires}.
\newblock \'Editions Jacques Gabay, Sceaux, 1993.
\newblock Reprint of the 1932 original.

\bibitem{BG16}
Y.~Bartal and L.-A. Gottlieb.
\newblock Dimension reduction techniques for {$\ell_p$} ({$1\leq p\leq 2$}),
  with applications.
\newblock In {\em 32nd {I}nternational {S}ymposium on {C}omputational
  {G}eometry}, volume~51 of {\em LIPIcs. Leibniz Int. Proc. Inform.}, pages
  Art. 16, 15. Schloss Dagstuhl. Leibniz-Zent. Inform., Wadern, 2016.

\bibitem{BGN15}
Y.~Bartal, L.-A. Gottlieb, and O.~Neiman.
\newblock On the impossibility of dimension reduction for doubling subsets of
  {$\ell_p$}.
\newblock {\em SIAM J. Discrete Math.}, 29(3):1207--1222, 2015.

\bibitem{BLMN04}
Y.~Bartal, N.~Linial, M.~Mendel, and A.~Naor.
\newblock Low dimensional embeddings of ultrametrics.
\newblock {\em European J. Combin.}, 25(1):87--92, 2004.

\bibitem{BM04}
Y.~Bartal and M.~Mendel.
\newblock Dimension reduction for ultrametrics.
\newblock In {\em Proceedings of the {F}ifteenth {A}nnual {ACM}-{SIAM}
  {S}ymposium on {D}iscrete {A}lgorithms}, pages 664--665. ACM, New York, 2004.

\bibitem{BRS11}
Y.~Bartal, B.~Recht, and L.~J. Schulman.
\newblock Dimensionality reduction: beyond the {J}ohnson-{L}indenstrauss bound.
\newblock In {\em Proceedings of the {T}wenty-{S}econd {A}nnual {ACM}-{SIAM}
  {S}ymposium on {D}iscrete {A}lgorithms}, pages 868--887. SIAM, Philadelphia,
  PA, 2011.

\bibitem{BGMN05}
F.~Barthe, O.~Gu\'edon, S.~Mendelson, and A.~Naor.
\newblock A probabilistic approach to the geometry of the {$l^n_p$}-ball.
\newblock {\em Ann. Probab.}, 33(2):480--513, 2005.

\bibitem{BSS12}
J.~Batson, D.~A. Spielman, and N.~Srivastava.
\newblock Twice-{R}amanujan sparsifiers.
\newblock {\em SIAM J. Comput.}, 41(6):1704--1721, 2012.

\bibitem{Bec62}
A.~Beck.
\newblock A convexity condition in {B}anach spaces and the strong law of large
  numbers.
\newblock {\em Proc. Amer. Math. Soc.}, 13:329--334, 1962.

\bibitem{BL}
Y.~Benyamini and J.~Lindenstrauss.
\newblock {\em Geometric nonlinear functional analysis. {V}ol. 1}, volume~48 of
  {\em American Mathematical Society Colloquium Publications}.
\newblock American Mathematical Society, Providence, RI, 2000.

\bibitem{Bol78}
B.~Bollob\'as.
\newblock {\em Extremal graph theory}, volume~11 of {\em London Mathematical
  Society Monographs}.
\newblock Academic Press, Inc. [Harcourt Brace Jovanovich, Publishers],
  London-New York, 1978.

\bibitem{Bol01}
B.~Bollob\'as.
\newblock {\em Random graphs}, volume~73 of {\em Cambridge Studies in Advanced
  Mathematics}.
\newblock Cambridge University Press, Cambridge, second edition, 2001.

\bibitem{BS00}
M.~Bonk and O.~Schramm.
\newblock Embeddings of {G}romov hyperbolic spaces.
\newblock {\em Geom. Funct. Anal.}, 10(2):266--306, 2000.

\bibitem{Bou28}
G.~{Bouligand}.
\newblock {Ensembles impropres et nombre dimensionnel. I, II.}
\newblock {\em {Bull. Sci. Math., II. S\'er.}}, 52:320--344, 1928.

\bibitem{BP99}
M.~Bourdon and H.~Pajot.
\newblock Poincar\'e inequalities and quasiconformal structure on the boundary
  of some hyperbolic buildings.
\newblock {\em Proc. Amer. Math. Soc.}, 127(8):2315--2324, 1999.

\bibitem{Bou84}
J.~Bourgain.
\newblock New {B}anach space properties of the disc algebra and {$H^{\infty
  }$}.
\newblock {\em Acta Math.}, 152(1-2):1--48, 1984.

\bibitem{Bou85}
J.~Bourgain.
\newblock On {L}ipschitz embedding of finite metric spaces in {H}ilbert space.
\newblock {\em Israel J. Math.}, 52(1-2):46--52, 1985.

\bibitem{Bou84-pisier}
J.~Bourgain.
\newblock Subspaces of {$l^\infty_N$}, arithmetical diameter and {S}idon sets.
\newblock In {\em Probability in {B}anach spaces, {V} ({M}edford, {M}ass.,
  1984)}, volume 1153 of {\em Lecture Notes in Math.}, pages 96--127. Springer,
  Berlin, 1985.

\bibitem{Bourgain-trees}
J.~Bourgain.
\newblock The metrical interpretation of superreflexivity in {B}anach spaces.
\newblock {\em Israel J. Math.}, 56(2):222--230, 1986.

\bibitem{Bou87}
J.~Bourgain.
\newblock Remarks on the extension of {L}ipschitz maps defined on discrete sets
  and uniform homeomorphisms.
\newblock In {\em Geometrical aspects of functional analysis (1985/86)}, volume
  1267 of {\em Lecture Notes in Math.}, pages 157--167. Springer, Berlin, 1987.

\bibitem{BDN15}
J.~Bourgain, S.~Dirksen, and J.~Nelson.
\newblock Toward a unified theory of sparse dimensionality reduction in
  {E}uclidean space.
\newblock {\em Geom. Funct. Anal.}, 25(4):1009--1088, 2015.

\bibitem{BFM86}
J.~Bourgain, T.~Figiel, and V.~Milman.
\newblock On {H}ilbertian subsets of finite metric spaces.
\newblock {\em Israel J. Math.}, 55(2):147--152, 1986.

\bibitem{BG12}
J.~Bourgain and A.~Gamburd.
\newblock A spectral gap theorem in {${\rm SU}(d)$}.
\newblock {\em J. Eur. Math. Soc. (JEMS)}, 14(5):1455--1511, 2012.

\bibitem{BLM89}
J.~Bourgain, J.~Lindenstrauss, and V.~Milman.
\newblock Approximation of zonoids by zonotopes.
\newblock {\em Acta Math.}, 162(1-2):73--141, 1989.

\bibitem{BM85}
J.~Bourgain and V.~Milman.
\newblock Dichotomie du cotype pour les espaces invariants.
\newblock {\em C. R. Acad. Sci. Paris S\'er. I Math.}, 300(9):263--266, 1985.

\bibitem{BMW}
J.~Bourgain, V.~Milman, and H.~Wolfson.
\newblock On type of metric spaces.
\newblock {\em Trans. Amer. Math. Soc.}, 294(1):295--317, 1986.

\bibitem{BHH16}
F.~G. S.~L. Brand\~ao, A.~W. Harrow, and M.~Horodecki.
\newblock Local random quantum circuits are approximate polynomial-designs.
\newblock {\em Comm. Math. Phys.}, 346(2):397--434, 2016.

\bibitem{Bra99}
P.~Bra{\ss}.
\newblock On equilateral simplices in normed spaces.
\newblock {\em Beitr\"age Algebra Geom.}, 40(2):303--307, 1999.

\bibitem{BNR12}
J.~Bri\"et, A.~Naor, and O.~Regev.
\newblock Locally decodable codes and the failure of cotype for projective
  tensor products.
\newblock {\em Electron. Res. Announc. Math. Sci.}, 19:120--130, 2012.

\bibitem{BC05}
B.~Brinkman and M.~Charikar.
\newblock On the impossibility of dimension reduction in {$l_1$}.
\newblock {\em J. ACM}, 52(5):766--788 (electronic), 2005.

\bibitem{BKL07}
B.~Brinkman, A.~Karagiozova, and J.~R. Lee.
\newblock Vertex cuts, random walks, and dimension reduction in series-parallel
  graphs.
\newblock In {\em S{TOC}'07---{P}roceedings of the 39th {A}nnual {ACM}
  {S}ymposium on {T}heory of {C}omputing}, pages 621--630. ACM, New York, 2007.

\bibitem{BGP92}
Y.~Burago, M.~Gromov, and G.~Perelman.
\newblock A. {D}. {A}leksandrov spaces with curvatures bounded below.
\newblock {\em Uspekhi Mat. Nauk}, 47(2(284)):3--51, 222, 1992.

\bibitem{Bur10}
C.~J.~C. Burges.
\newblock Dimension reduction: A guided tour.
\newblock {\em Foundations and Trends in Machine Learning}, 2(4):275--365,
  2010.

\bibitem{CGT10}
T.-H.~H. Chan, A.~Gupta, and K.~Talwar.
\newblock Ultra-low-dimensional embeddings for doubling metrics.
\newblock {\em J. ACM}, 57(4):Art. 21, 26, 2010.

\bibitem{CS02}
M.~Charikar and A.~Sahai.
\newblock Dimension reduction in the {\textbackslash}ell {\_}1 norm.
\newblock In {\em 43rd Symposium on Foundations of Computer Science {(FOCS}
  2002), 16-19 November 2002, Vancouver, BC, Canada, Proceedings}, pages
  551--560, 2002.

\bibitem{Cha02}
M.~S. Charikar.
\newblock Similarity estimation techniques from rounding algorithms.
\newblock In {\em Proceedings of the {T}hirty-{F}ourth {A}nnual {ACM}
  {S}ymposium on {T}heory of {C}omputing}, pages 380--388. ACM, New York, 2002.

\bibitem{Che14}
S.~Chechik.
\newblock Approximate distance oracles with constant query time.
\newblock In {\em S{TOC}'14---{P}roceedings of the 2014 {ACM} {S}ymposium on
  {T}heory of {C}omputing}, pages 654--663. ACM, New York, 2014.

\bibitem{Che70}
J.~Cheeger.
\newblock A lower bound for the smallest eigenvalue of the {L}aplacian.
\newblock In {\em Problems in analysis ({P}apers dedicated to {S}alomon
  {B}ochner, 1969)}, pages 195--199. Princeton Univ. Press, Princeton, N. J.,
  1970.

\bibitem{Che99}
J.~Cheeger.
\newblock Differentiability of {L}ipschitz functions on metric measure spaces.
\newblock {\em Geom. Funct. Anal.}, 9(3):428--517, 1999.

\bibitem{CK06}
J.~Cheeger and B.~Kleiner.
\newblock On the differentiability of {L}ipschitz maps from metric measure
  spaces to {B}anach spaces.
\newblock In {\em Inspired by {S}. {S}. {C}hern}, volume~11 of {\em Nankai
  Tracts Math.}, pages 129--152. World Sci. Publ., Hackensack, NJ, 2006.

\bibitem{CK09-RNP}
J.~Cheeger and B.~Kleiner.
\newblock Differentiability of {L}ipschitz maps from metric measure spaces to
  {B}anach spaces with the {R}adon-{N}ikod\'ym property.
\newblock {\em Geom. Funct. Anal.}, 19(4):1017--1028, 2009.

\bibitem{CK10}
J.~Cheeger and B.~Kleiner.
\newblock Differentiating maps into {$L^1$}, and the geometry of {BV}
  functions.
\newblock {\em Ann. of Math. (2)}, 171(2):1347--1385, 2010.

\bibitem{CKN11}
J.~Cheeger, B.~Kleiner, and A.~Naor.
\newblock Compression bounds for {L}ipschitz maps from the {H}eisenberg group
  to {$L_1$}.
\newblock {\em Acta Math.}, 207(2):291--373, 2011.

\bibitem{CFM94}
F.~R.~K. Chung, V.~Faber, and T.~A. Manteuffel.
\newblock An upper bound on the diameter of a graph from eigenvalues associated
  with its {L}aplacian.
\newblock {\em SIAM J. Discrete Math.}, 7(3):443--457, 1994.

\bibitem{Cir74}
B.~S. Cirel'son.
\newblock It is impossible to imbed {$\ell_{p}$} or {$c_{0}$} into an arbitrary
  {B}anach space.
\newblock {\em Funkcional. Anal. i Prilo\v zen.}, 8(2):57--60, 1974.

\bibitem{CW71}
R.~R. Coifman and G.~Weiss.
\newblock {\em Analyse harmonique non-commutative sur certains espaces
  homog\`enes}.
\newblock Lecture Notes in Mathematics, Vol. 242. Springer-Verlag, Berlin-New
  York, 1971.
\newblock \'Etude de certaines int\'egrales singuli\`eres.

\bibitem{DG03}
S.~Dasgupta and A.~Gupta.
\newblock An elementary proof of a theorem of {J}ohnson and {L}indenstrauss.
\newblock {\em Random Structures Algorithms}, 22(1):60--65, 2003.

\bibitem{DS97}
G.~David and S.~Semmes.
\newblock {\em Fractured fractals and broken dreams}, volume~7 of {\em Oxford
  Lecture Series in Mathematics and its Applications}.
\newblock The Clarendon Press, Oxford University Press, New York, 1997.
\newblock Self-similar geometry through metric and measure.

\bibitem{DS13}
G.~David and M.~Snipes.
\newblock A non-probabilistic proof of the {A}ssouad embedding theorem with
  bounds on the dimension.
\newblock {\em Anal. Geom. Metr. Spaces}, 1:36--41, 2013.

\bibitem{Dek00}
B.~V. Dekster.
\newblock Simplexes with prescribed edge lengths in {M}inkowski and {B}anach
  spaces.
\newblock {\em Acta Math. Hungar.}, 86(4):343--358, 2000.

\bibitem{DFS03}
J.~Diestel, J.~Fourie, and J.~Swart.
\newblock The projective tensor product. {I}.
\newblock In {\em Trends in {B}anach spaces and operator theory ({M}emphis,
  {TN}, 2001)}, volume 321 of {\em Contemp. Math.}, pages 37--65. Amer. Math.
  Soc., Providence, RI, 2003.

\bibitem{DFS08}
J.~Diestel, J.~H. Fourie, and J.~Swart.
\newblock {\em The metric theory of tensor products}.
\newblock American Mathematical Society, Providence, RI, 2008.
\newblock Grothendieck's r\'esum\'e revisited.

\bibitem{Dir16}
S.~Dirksen.
\newblock Dimensionality reduction with subgaussian matrices: a unified theory.
\newblock {\em Found. Comput. Math.}, 16(5):1367--1396, 2016.

\bibitem{DG09}
A.~Dmitriyuk and Y.~Gordon.
\newblock Generalizing the {J}ohnson-{L}indenstrauss lemma to {$k$}-dimensional
  affine subspaces.
\newblock {\em Studia Math.}, 195(3):227--241, 2009.

\bibitem{DG14}
A.~Dmitriyuk and Y.~Gordon.
\newblock Randomized large distortion dimension reduction.
\newblock {\em Positivity}, 18(4):767--784, 2014.

\bibitem{Dur10}
R.~Durrett.
\newblock {\em Probability: theory and examples}, volume~31 of {\em Cambridge
  Series in Statistical and Probabilistic Mathematics}.
\newblock Cambridge University Press, Cambridge, fourth edition, 2010.

\bibitem{Dvo60}
A.~Dvoretzky.
\newblock Some results on convex bodies and {B}anach spaces.
\newblock In {\em Proc. {I}nternat. {S}ympos. {L}inear {S}paces ({J}erusalem,
  1960)}, pages 123--160. Jerusalem Academic Press, Jerusalem, 1961.

\bibitem{Efr09}
K.~Efremenko.
\newblock 3-query locally decodable codes of subexponential length.
\newblock In {\em S{TOC}'09---{P}roceedings of the 2009 {ACM} {I}nternational
  {S}ymposium on {T}heory of {C}omputing}, pages 39--44. ACM, New York, 2009.

\bibitem{Enf70}
P.~Enflo.
\newblock Uniform structures and square roots in topological groups. {I}, {II}.
\newblock {\em Israel J. Math. 8 (1970), 230-252; ibid.}, 8:253--272, 1970.

\bibitem{Enf72}
P.~Enflo.
\newblock Banach spaces which can be given an equivalent uniformly convex norm.
\newblock {\em Israel J. Math.}, 13:281--288 (1973), 1972.

\bibitem{Enf78}
P.~Enflo.
\newblock On infinite-dimensional topological groups.
\newblock In {\em S\'eminaire sur la {G}\'eom\'etrie des {E}spaces de {B}anach
  (1977--1978)}, pages Exp. No. 10--11, 11. \'Ecole Polytech., Palaiseau, 1978.

\bibitem{Erd64}
P.~Erd\H{o}s.
\newblock Extremal problems in graph theory.
\newblock In {\em Theory of {G}raphs and its {A}pplications ({P}roc. {S}ympos.
  {S}molenice, 1963)}, pages 29--36. Publ. House Czechoslovak Acad. Sci.,
  Prague, 1964.

\bibitem{EMN17-bifurcation}
A.~Eskenazis, M.~Mendel, and A.~Naor.
\newblock Diamond convexity: {A} bifurcation in the {R}ibe program.
\newblock Preprint, 2017.

\bibitem{EMN17}
A.~Eskenazis, M.~Mendel, and A.~Naor.
\newblock Nonpositive curvature is not coarsely universal.
\newblock Preprint, 2017.

\bibitem{FJ74}
T.~Figiel and W.~B. Johnson.
\newblock A uniformly convex {B}anach space which contains no {$l_{p}$}.
\newblock {\em Compositio Math.}, 29:179--190, 1974.

\bibitem{FM88}
P.~Frankl and H.~Maehara.
\newblock The {J}ohnson-{L}indenstrauss lemma and the sphericity of some
  graphs.
\newblock {\em J. Combin. Theory Ser. B}, 44(3):355--362, 1988.

\bibitem{Fre10}
M.~Fr\'echet.
\newblock Les dimensions d'un ensemble abstrait.
\newblock {\em Math. Ann.}, 68(2):145--168, 1910.

\bibitem{FL94}
Z.~F\"uredi and P.~A. Loeb.
\newblock On the best constant for the {B}esicovitch covering theorem.
\newblock {\em Proc. Amer. Math. Soc.}, 121(4):1063--1073, 1994.

\bibitem{Gen00}
A.~Genz.
\newblock Methods for generating random orthogonal matrices.
\newblock In {\em Monte {C}arlo and quasi-{M}onte {C}arlo methods 1998
  ({C}laremont, {CA})}, pages 199--213. Springer, Berlin, 2000.

\bibitem{GNS11}
O.~Giladi, M.~Mendel, and A.~Naor.
\newblock Improved bounds in the metric cotype inequality for {B}anach spaces.
\newblock {\em J. Funct. Anal.}, 260(1):164--194, 2011.

\bibitem{GN10}
O.~Giladi and A.~Naor.
\newblock Improved bounds in the scaled {E}nflo type inequality for {B}anach
  spaces.
\newblock {\em Extracta Math.}, 25(2):151--164, 2010.

\bibitem{GNS12}
O.~Giladi, A.~Naor, and G.~Schechtman.
\newblock Bourgain's discretization theorem.
\newblock {\em Ann. Fac. Sci. Toulouse Math. (6)}, 21(4):817--837, 2012.

\bibitem{God17}
G.~Godefroy.
\newblock De {G}rothendieck \`a {N}aor: une promenade dans l'analyse m\'etrique
  des espaces de {B}anach.
\newblock {\em Gaz. Math.}, (151):13--24, 2017.

\bibitem{Gor88}
Y.~Gordon.
\newblock On {M}ilman's inequality and random subspaces which escape through a
  mesh in {${\bf R}^n$}.
\newblock In {\em Geometric aspects of functional analysis (1986/87)}, volume
  1317 of {\em Lecture Notes in Math.}, pages 84--106. Springer, Berlin, 1988.

\bibitem{GK15}
L.-A. Gottlieb and R.~Krauthgamer.
\newblock A nonlinear approach to dimension reduction.
\newblock {\em Discrete Comput. Geom.}, 54(2):291--315, 2015.

\bibitem{Gro93}
M.~Gromov.
\newblock Asymptotic invariants of infinite groups.
\newblock In {\em Geometric group theory, {V}ol.\ 2 ({S}ussex, 1991)}, volume
  182 of {\em London Math. Soc. Lecture Note Ser.}, pages 1--295. Cambridge
  Univ. Press, Cambridge, 1993.

\bibitem{Gro03}
M.~Gromov.
\newblock Random walk in random groups.
\newblock {\em Geom. Funct. Anal.}, 13(1):73--146, 2003.

\bibitem{Gro52}
A.~Grothendieck.
\newblock R\'esum\'e des r\'esultats essentiels dans la th\'eorie des produits
  tensoriels topologiques et des espaces nucl\'eaires.
\newblock {\em Ann. Inst. Fourier Grenoble}, 4:73--112 (1954), 1952.

\bibitem{Gro53}
A.~Grothendieck.
\newblock R\'esum\'e de la th\'eorie m\'etrique des produits tensoriels
  topologiques.
\newblock {\em Bol. Soc. Mat. S\~ao Paulo}, 8:1--79, 1953.

\bibitem{Gro53-dvo-conj}
A.~Grothendieck.
\newblock Sur certaines classes de suites dans les espaces de {B}anach et le
  th\'eor\`eme de {D}voretzky-{R}ogers.
\newblock {\em Bol. Soc. Mat. S\~ao Paulo}, 8:81--110 (1956), 1953.

\bibitem{Gro55}
A.~Grothendieck.
\newblock Produits tensoriels topologiques et espaces nucl\'eaires.
\newblock {\em Mem. Amer. Math. Soc.}, No. 16:140, 1955.

\bibitem{GKL03}
A.~Gupta, R.~Krauthgamer, and J.~R. Lee.
\newblock Bounded geometries, fractals, and low-distortion embeddings.
\newblock In {\em 44th Symposium on Foundations of Computer Science {(FOCS}
  2003), 11-14 October 2003, Cambridge, MA, USA, Proceedings}, pages 534--543,
  2003.

\bibitem{GNRS04}
A.~Gupta, I.~Newman, Y.~Rabinovich, and A.~Sinclair.
\newblock Cuts, trees and {$l_1$}-embeddings of graphs.
\newblock {\em Combinatorica}, 24(2):233--269, 2004.

\bibitem{GT11}
A.~Gupta and K.~Talwar.
\newblock Making doubling metrics geodesic.
\newblock {\em Algorithmica}, 59(1):66--80, 2011.

\bibitem{Har11}
S.~Har-Peled.
\newblock {\em Geometric approximation algorithms}, volume 173 of {\em
  Mathematical Surveys and Monographs}.
\newblock American Mathematical Society, Providence, RI, 2011.

\bibitem{HM06}
S.~Har-Peled and M.~Mendel.
\newblock Fast construction of nets in low-dimensional metrics and their
  applications.
\newblock {\em SIAM J. Comput.}, 35(5):1148--1184, 2006.

\bibitem{Hei03}
J.~Heinonen.
\newblock {\em Geometric embeddings of metric spaces}, volume~90 of {\em
  Report. University of Jyv\"askyl\"a Department of Mathematics and
  Statistics}.
\newblock University of Jyv\"askyl\"a, Jyv\"askyl\"a, 2003.

\bibitem{HM82}
S.~Heinrich and P.~Mankiewicz.
\newblock Applications of ultrapowers to the uniform and {L}ipschitz
  classification of {B}anach spaces.
\newblock {\em Studia Math.}, 73(3):225--251, 1982.

\bibitem{HS06}
G.~E. Hinton and R.~R. Salakhutdinov.
\newblock Reducing the dimensionality of data with neural networks.
\newblock {\em Science}, 313(5786):504--507, 2006.

\bibitem{HLW}
S.~Hoory, N.~Linial, and A.~Wigderson.
\newblock Expander graphs and their applications.
\newblock {\em Bull. Amer. Math. Soc. (N.S.)}, 43(4):439--561 (electronic),
  2006.

\bibitem{HN16}
T.~Hyt\"onen and A.~Naor.
\newblock Heat flow and quantitative differentiation.
\newblock To appear in {\em J. Eur. Math. Soc. (JEMS)}, preprint available at
  \url{https://arxiv.org/abs/1608.01915}, 2016.

\bibitem{Ind01}
P.~Indyk.
\newblock Algorithmic applications of low-distortion geometric embeddings.
\newblock In {\em 42nd {IEEE} {S}ymposium on {F}oundations of {C}omputer
  {S}cience ({L}as {V}egas, {NV}, 2001)}, pages 10--33. IEEE Computer Soc., Los
  Alamitos, CA, 2001.

\bibitem{IM99}
P.~Indyk and R.~Motwani.
\newblock Approximate nearest neighbors: towards removing the curse of
  dimensionality.
\newblock In {\em S{TOC} '98 ({D}allas, {TX})}, pages 604--613. ACM, New York,
  1999.

\bibitem{IN07}
P.~Indyk and A.~Naor.
\newblock Nearest-neighbor-preserving embeddings.
\newblock {\em ACM Trans. Algorithms}, 3(3):Art. 31, 12, 2007.

\bibitem{IT03}
P.~Indyk and N.~Thaper.
\newblock Fast image retrieval via embeddings.
\newblock In {\em ICCV '03: Proceedings of the 3rd International Workshop on
  Statistical and Computational Theories of Vision}, 2003.

\bibitem{Jam72}
R.~C. James.
\newblock Super-reflexive {B}anach spaces.
\newblock {\em Canad. J. Math.}, 24:896--904, 1972.

\bibitem{Joh48}
F.~John.
\newblock Extremum problems with inequalities as subsidiary conditions.
\newblock In {\em Studies and {E}ssays {P}resented to {R}. {C}ourant on his
  60th {B}irthday, {J}anuary 8, 1948}, pages 187--204. Interscience Publishers,
  Inc., New York, N. Y., 1948.

\bibitem{JL84}
W.~B. Johnson and J.~Lindenstrauss.
\newblock Extensions of {L}ipschitz mappings into a {H}ilbert space.
\newblock In {\em Conference in modern analysis and probability ({N}ew {H}aven,
  {C}onn., 1982)}, volume~26 of {\em Contemp. Math.}, pages 189--206. Amer.
  Math. Soc., Providence, RI, 1984.

\bibitem{JLS87}
W.~B. Johnson, J.~Lindenstrauss, and G.~Schechtman.
\newblock On {L}ipschitz embedding of finite metric spaces in low-dimensional
  normed spaces.
\newblock In {\em Geometrical aspects of functional analysis (1985/86)}, volume
  1267 of {\em Lecture Notes in Math.}, pages 177--184. Springer, Berlin, 1987.

\bibitem{JN10}
W.~B. Johnson and A.~Naor.
\newblock The {J}ohnson-{L}indenstrauss lemma almost characterizes {H}ilbert
  space, but not quite.
\newblock {\em Discrete Comput. Geom.}, 43(3):542--553, 2010.

\bibitem{JS01}
W.~B. Johnson and G.~Schechtman.
\newblock Finite dimensional subspaces of {$L_p$}.
\newblock In {\em Handbook of the geometry of {B}anach spaces, {V}ol. {I}},
  pages 837--870. North-Holland, Amsterdam, 2001.

\bibitem{JS09}
W.~B. Johnson and G.~Schechtman.
\newblock Diamond graphs and super-reflexivity.
\newblock {\em J. Topol. Anal.}, 1(2):177--189, 2009.

\bibitem{Kah81}
J.-P. Kahane.
\newblock H\'elices et quasi-h\'elices.
\newblock In {\em Mathematical analysis and applications, {P}art {B}}, volume~7
  of {\em Adv. in Math. Suppl. Stud.}, pages 417--433. Academic Press, New
  York-London, 1981.

\bibitem{Kal08}
N.~J. Kalton.
\newblock The nonlinear geometry of {B}anach spaces.
\newblock {\em Rev. Mat. Complut.}, 21(1):7--60, 2008.

\bibitem{KN14}
D.~M. Kane and J.~Nelson.
\newblock Sparser {J}ohnson-{L}indenstrauss transforms.
\newblock {\em J. ACM}, 61(1):Art. 4, 23, 2014.

\bibitem{KN06}
S.~Khot and A.~Naor.
\newblock {Nonembeddability theorems via Fourier analysis}.
\newblock {\em Mathematische Annalen}, 334(4):821--852, 2006.

\bibitem{KN12}
S.~Khot and A.~Naor.
\newblock Grothendieck-type inequalities in combinatorial optimization.
\newblock {\em Comm. Pure Appl. Math.}, 65(7):992--1035, 2012.

\bibitem{KM05}
B.~Klartag and S.~Mendelson.
\newblock Empirical processes and random projections.
\newblock {\em J. Funct. Anal.}, 225(1):229--245, 2005.

\bibitem{KM15}
P.~K. Kothari and R.~Meka.
\newblock Almost optimal pseudorandom generators for spherical caps [extended
  abstract].
\newblock In {\em S{TOC}'15---{P}roceedings of the 2015 {ACM} {S}ymposium on
  {T}heory of {C}omputing}, pages 247--256. ACM, New York, 2015.

\bibitem{KW11}
F.~Krahmer and R.~Ward.
\newblock New and improved {J}ohnson-{L}indenstrauss embeddings via the
  restricted isometry property.
\newblock {\em SIAM J. Math. Anal.}, 43(3):1269--1281, 2011.

\bibitem{KLMN05}
R.~Krauthgamer, J.~R. Lee, M.~Mendel, and A.~Naor.
\newblock Measured descent: a new embedding method for finite metrics.
\newblock {\em Geom. Funct. Anal.}, 15(4):839--858, 2005.

\bibitem{Laakso}
T.~J. Laakso.
\newblock Ahlfors {$Q$}-regular spaces with arbitrary {$Q>1$} admitting weak
  {P}oincar\'e inequality.
\newblock {\em Geom. Funct. Anal.}, 10(1):111--123, 2000.

\bibitem{Laa02}
T.~J. Laakso.
\newblock Plane with {$A_\infty$}-weighted metric not bi-{L}ipschitz embeddable
  to {${\mathbb{R}}^N$}.
\newblock {\em Bull. London Math. Soc.}, 34(6):667--676, 2002.

\bibitem{Laf09}
V.~Lafforgue.
\newblock Propri\'et\'e ({T}) renforc\'ee {B}anachique et transformation de
  {F}ourier rapide.
\newblock {\em J. Topol. Anal.}, 1(3):191--206, 2009.

\bibitem{LN14}
V.~Lafforgue and A.~Naor.
\newblock A doubling subset of {$L_p$} for {$p>2$} that is inherently infinite
  dimensional.
\newblock {\em Geom. Dedicata}, 172:387--398, 2014.

\bibitem{LN14-poincare}
V.~Lafforgue and A.~Naor.
\newblock Vertical versus horizontal {P}oincar\'e inequalities on the
  {H}eisenberg group.
\newblock {\em Israel J. Math.}, 203(1):309--339, 2014.

\bibitem{LP01}
U.~Lang and C.~Plaut.
\newblock Bilipschitz embeddings of metric spaces into space forms.
\newblock {\em Geom. Dedicata}, 87(1-3):285--307, 2001.

\bibitem{LS05}
U.~Lang and T.~Schlichenmaier.
\newblock Nagata dimension, quasisymmetric embeddings, and {L}ipschitz
  extensions.
\newblock {\em Int. Math. Res. Not.}, (58):3625--3655, 2005.

\bibitem{GN16}
K.~G. Larsen and J.~Nelson.
\newblock The {J}ohnson-{L}indenstrauss lemma is optimal for linear
  dimensionality reduction.
\newblock In {\em 43rd {I}nternational {C}olloquium on {A}utomata, {L}anguages,
  and {P}rogramming}, volume~55 of {\em LIPIcs. Leibniz Int. Proc. Inform.},
  pages Art. No. 82, 11. Schloss Dagstuhl. Leibniz-Zent. Inform., Wadern, 2016.

\bibitem{GN17}
K.~G. Larsen and J.~Nelson.
\newblock Optimality of the {J}ohnson--{L}indenstrauss lemma.
\newblock In {\em 58th {IEEE} Annual Symposium on Foundations of Computer
  Science, {FOCS} 2017, Berkeley, CA, USA, October 15-17, 2017}, pages
  633--638, 2017.

\bibitem{LUW95}
F.~Lazebnik, V.~A. Ustimenko, and A.~J. Woldar.
\newblock A new series of dense graphs of high girth.
\newblock {\em Bull. Amer. Math. Soc. (N.S.)}, 32(1):73--79, 1995.

\bibitem{LdMM13}
J.~R. Lee, A.~de~Mesmay, and M.~Moharrami.
\newblock Dimension reduction for finite trees in {$\ell_1$}.
\newblock {\em Discrete Comput. Geom.}, 50(4):977--1032, 2013.

\bibitem{LMN05}
J.~R. Lee, M.~Mendel, and A.~Naor.
\newblock Metric structures in {$L_1$}: dimension, snowflakes, and average
  distortion.
\newblock {\em European J. Combin.}, 26(8):1180--1190, 2005.

\bibitem{LN04}
J.~R. Lee and A.~Naor.
\newblock Embedding the diamond graph in {$L_p$} and dimension reduction in
  {$L_1$}.
\newblock {\em Geom. Funct. Anal.}, 14(4):745--747, 2004.

\bibitem{LN06}
J.~R. Lee and A.~Naor.
\newblock ${L}_p$ metrics on the {H}eisenberg group and the {G}oemans-{L}inial
  conjecture.
\newblock In {\em Proceedings of 47th Annual IEEE Symposium on Foundations of
  Computer Science (FOCS 2006)}, pages 99--108, 2006.
\newblock Available at
  \url{https://web.math.princeton.edu/~naor/homepage\%20files/L_pHGL.pdf}.

\bibitem{LNP09}
J.~R. Lee, A.~Naor, and Y.~Peres.
\newblock Trees and {M}arkov convexity.
\newblock {\em Geom. Funct. Anal.}, 18(5):1609--1659, 2009.

\bibitem{LPW17}
D.~A. Levin, Y.~Peres, and E.~L. Wilmer.
\newblock {\em Markov chains and mixing times}.
\newblock American Mathematical Society, Providence, RI, 2017.
\newblock Second edition of [ MR2466937], With a chapter on ``Coupling from the
  past'' by James G. Propp and David B. Wilson.

\bibitem{LP68}
J.~Lindenstrauss and A.~Pe{\l}czy\'nski.
\newblock Absolutely summing operators in {$L_{p}$}-spaces and their
  applications.
\newblock {\em Studia Math.}, 29:275--326, 1968.

\bibitem{LR69}
J.~Lindenstrauss and H.~P. Rosenthal.
\newblock The {$\mathcal{L}_{p}$} spaces.
\newblock {\em Israel J. Math.}, 7:325--349, 1969.

\bibitem{Lin02}
N.~Linial.
\newblock Finite metric-spaces---combinatorics, geometry and algorithms.
\newblock In {\em Proceedings of the {I}nternational {C}ongress of
  {M}athematicians, {V}ol. {III} ({B}eijing, 2002)}, pages 573--586. Higher Ed.
  Press, Beijing, 2002.

\bibitem{LLR}
N.~Linial, E.~London, and Y.~Rabinovich.
\newblock The geometry of graphs and some of its algorithmic applications.
\newblock {\em Combinatorica}, 15(2):215--245, 1995.

\bibitem{MP84}
M.~B. Marcus and G.~Pisier.
\newblock Characterizations of almost surely continuous {$p$}-stable random
  {F}ourier series and strongly stationary processes.
\newblock {\em Acta Math.}, 152(3-4):245--301, 1984.

\bibitem{Mat92}
J.~Matou\v{s}ek.
\newblock Note on bi-{L}ipschitz embeddings into normed spaces.
\newblock {\em Comment. Math. Univ. Carolin.}, 33(1):51--55, 1992.

\bibitem{Mat92-ramsey}
J.~Matou\v{s}ek.
\newblock Ramsey-like properties for bi-{L}ipschitz mappings of finite metric
  spaces.
\newblock {\em Comment. Math. Univ. Carolin.}, 33(3):451--463, 1992.

\bibitem{Mat96}
J.~Matou\v{s}ek.
\newblock On the distortion required for embedding finite metric spaces into
  normed spaces.
\newblock {\em Israel J. Math.}, 93:333--344, 1996.

\bibitem{Mat97}
J.~Matou\v{s}ek.
\newblock On embedding expanders into $\ell_p$ spaces.
\newblock {\em Israel J. Math.}, 102:189--197, 1997.

\bibitem{Mat02}
J.~Matou\v{s}ek.
\newblock {\em Lectures on discrete geometry}, volume 212 of {\em Graduate
  Texts in Mathematics}.
\newblock Springer-Verlag, New York, 2002.

\bibitem{Mat08}
J.~Matou\v{s}ek.
\newblock On variants of the {J}ohnson-{L}indenstrauss lemma.
\newblock {\em Random Structures Algorithms}, 33(2):142--156, 2008.

\bibitem{Mau03}
B.~Maurey.
\newblock Type, cotype and {$K$}-convexity.
\newblock In {\em Handbook of the geometry of {B}anach spaces, {V}ol.\ 2},
  pages 1299--1332. North-Holland, Amsterdam, 2003.

\bibitem{MP73}
B.~Maurey and G.~Pisier.
\newblock Caract\'erisation d'une classe d'espaces de {B}anach par des
  propri\'et\'es de s\'eries al\'eatoires vectorielles.
\newblock {\em C. R. Acad. Sci. Paris S\'er. A-B}, 277:A687--A690, 1973.

\bibitem{MP-type-cotype}
B.~Maurey and G.~Pisier.
\newblock S\'eries de variables al\'eatoires vectorielles ind\'ependantes et
  propri\'et\'es g\'eom\'etriques des espaces de {B}anach.
\newblock {\em Studia Math.}, 58(1):45--90, 1976.

\bibitem{Maz29}
S.~{Mazur}.
\newblock {Une remarque sur l'hom\'eomorphie des champs fonctionels.}
\newblock {\em {Stud. Math.}}, 1:83--85, 1929.

\bibitem{Men09}
M.~Mendel.
\newblock Metric dichotomies.
\newblock In {\em Limits of graphs in group theory and computer science}, pages
  59--76. EPFL Press, Lausanne, 2009.

\bibitem{MN07-ramsey}
M.~Mendel and A.~Naor.
\newblock Ramsey partitions and proximity data structures.
\newblock {\em J. Eur. Math. Soc. (JEMS)}, 9(2):253--275, 2007.

\bibitem{MN07-scaled}
M.~Mendel and A.~Naor.
\newblock Scaled {E}nflo type is equivalent to {R}ademacher type.
\newblock {\em Bull. Lond. Math. Soc.}, 39(3):493--498, 2007.

\bibitem{MN08-cotype}
M.~Mendel and A.~Naor.
\newblock Metric cotype.
\newblock {\em Ann. of Math. (2)}, 168(1):247--298, 2008.

\bibitem{MN11-arxiv}
M.~Mendel and A.~Naor.
\newblock A note on dichotomies for metric transforms.
\newblock Available at \url{http://arxiv.org/abs/1102.1800}, 2011.

\bibitem{MN13}
M.~Mendel and A.~Naor.
\newblock Markov convexity and local rigidity of distorted metrics.
\newblock {\em J. Eur. Math. Soc. (JEMS)}, 15(1):287--337, 2013.

\bibitem{MN14}
M.~Mendel and A.~Naor.
\newblock Nonlinear spectral calculus and super-expanders.
\newblock {\em Publ. Math. Inst. Hautes \'Etudes Sci.}, 119:1--95, 2014.

\bibitem{MN15}
M.~Mendel and A.~Naor.
\newblock Expanders with respect to {H}adamard spaces and random graphs.
\newblock {\em Duke Math. J.}, 164(8):1471--1548, 2015.

\bibitem{MPT07}
S.~Mendelson, A.~Pajor, and N.~Tomczak-Jaegermann.
\newblock Reconstruction and subgaussian operators in asymptotic geometric
  analysis.
\newblock {\em Geom. Funct. Anal.}, 17(4):1248--1282, 2007.

\bibitem{MT08}
S.~Mendelson and N.~Tomczak-Jaegermann.
\newblock A subgaussian embedding theorem.
\newblock {\em Israel J. Math.}, 164:349--364, 2008.

\bibitem{Men28}
K.~{Menger}.
\newblock {Dimensionstheorie.}
\newblock {Leipzig: B. G. Teubner. iv, 318 S. (1928).}, 1928.

\bibitem{Mez07}
F.~Mezzadri.
\newblock How to generate random matrices from the classical compact groups.
\newblock {\em Notices Amer. Math. Soc.}, 54(5):592--604, 2007.

\bibitem{MS}
V.~D. Milman and G.~Schechtman.
\newblock {\em Asymptotic theory of finite-dimensional normed spaces}, volume
  1200 of {\em Lecture Notes in Mathematics}.
\newblock Springer-Verlag, Berlin, 1986.
\newblock With an appendix by M. Gromov.

\bibitem{MS95}
V.~D. Milman and G.~Schechtman.
\newblock An ``isomorphic'' version of {D}voretzky's theorem.
\newblock {\em C. R. Acad. Sci. Paris S\'er. I Math.}, 321(5):541--544, 1995.

\bibitem{Mil64}
J.~Milnor.
\newblock On the {B}etti numbers of real varieties.
\newblock {\em Proc. Amer. Math. Soc.}, 15:275--280, 1964.

\bibitem{Nao12-quasi}
A.~Naor.
\newblock An application of metric cotype to quasisymmetric embeddings.
\newblock In {\em Metric and differential geometry}, volume 297 of {\em Progr.
  Math.}, pages 175--178. Birkh\"auser/Springer, Basel, 2012.

\bibitem{Nao12}
A.~Naor.
\newblock An introduction to the {R}ibe program.
\newblock {\em Jpn. J. Math.}, 7(2):167--233, 2012.

\bibitem{Nao14}
A.~Naor.
\newblock Comparison of metric spectral gaps.
\newblock {\em Anal. Geom. Metr. Spaces}, 2:1--52, 2014.

\bibitem{Nao16-riesz}
A.~Naor.
\newblock Discrete {R}iesz transforms and sharp metric {$X_p$} inequalities.
\newblock {\em Ann. of Math. (2)}, 184(3):991--1016, 2016.

\bibitem{Nao17}
A.~Naor.
\newblock A spectral gap precludes low-dimensional embeddings.
\newblock In {\em 33rd {I}nternational {S}ymposium on {C}omputational
  {G}eometry}, volume~77 of {\em LIPIcs. Leibniz Int. Proc. Inform.}, pages
  Art. No. 50, 16. Schloss Dagstuhl. Leibniz-Zent. Inform., Wadern, 2017.

\bibitem{NN12}
A.~Naor and O.~Neiman.
\newblock Assouad's theorem with dimension independent of the snowflaking.
\newblock {\em Rev. Mat. Iberoam.}, 28(4):1123--1142, 2012.

\bibitem{NPSS06}
A.~Naor, Y.~Peres, O.~Schramm, and S.~Sheffield.
\newblock Markov chains in smooth {B}anach spaces and {G}romov-hyperbolic
  metric spaces.
\newblock {\em Duke Math. J.}, 134(1):165--197, 2006.

\bibitem{NPS18}
A.~Naor, G.~Pisier, and G.~Schechtman.
\newblock Impossibility of dimension reduction in the nuclear norm.
\newblock In {\em Proceedings of the Twenty-Ninth Annual {ACM-SIAM} Symposium
  on Discrete Algorithms, {SODA} 2018, New Orleans, LA, USA, January 7-10,
  2018}, pages 1345--1352, 2018.
\newblock Available at \url{https://arxiv.org/abs/1710.08896}.

\bibitem{NR17}
A.~Naor and Y.~Rabani.
\newblock On {L}ipschitz extension from finite subsets.
\newblock {\em Israel J. Math.}, 219(1):115--161, 2017.

\bibitem{NS07}
A.~Naor and G.~Schechtman.
\newblock Planar earthmover is not in {$L_1$}.
\newblock {\em SIAM J. Comput.}, 37(3):804--826, 2007.

\bibitem{NS11}
A.~Naor and L.~Silberman.
\newblock Poincar\'e inequalities, embeddings, and wild groups.
\newblock {\em Compos. Math.}, 147(5):1546--1572, 2011.

\bibitem{NY17}
A.~Naor and R.~Young.
\newblock The integrality gap of the {G}oemans-{L}inial {SDP} relaxation for
  sparsest cut is at least a constant multiple of {$\sqrt{\log n}$}.
\newblock In {\em S{TOC}'17---{P}roceedings of the 49th {A}nnual {ACM} {SIGACT}
  {S}ymposium on {T}heory of {C}omputing}, pages 564--575. ACM, New York, 2017.

\bibitem{NY17-3dim}
A.~Naor and R.~Young.
\newblock Foliated corona decompositions.
\newblock Preprint, 2018.

\bibitem{NY17-versus}
A.~Naor and R.~Young.
\newblock Vertical perimeter versus horizontal perimeter.
\newblock {\em Ann. of Math. (2)}, 188(1):171--279, 2018.

\bibitem{Nas54}
J.~Nash.
\newblock {$C^1$} isometric imbeddings.
\newblock {\em Ann. of Math. (2)}, 60:383--396, 1954.

\bibitem{Nei16}
O.~Neiman.
\newblock Low dimensional embeddings of doubling metrics.
\newblock {\em Theory Comput. Syst.}, 58(1):133--152, 2016.

\bibitem{Nel16}
J.~Nelson.
\newblock Chaining introduction with some computer science applications.
\newblock {\em Bull. Eur. Assoc. Theor. Comput. Sci. EATCS}, (120):42--65,
  2016.

\bibitem{NR13}
I.~Newman and Y.~Rabinovich.
\newblock On multiplicative {$\lambda$}-approximations and some geometric
  applications.
\newblock {\em SIAM J. Comput.}, 42(3):855--883, 2013.

\bibitem{Nob31}
G.~N\"obeling.
\newblock \"uber eine {$n$}-dimensionale {U}niversalmenge im {$R^{2n+1}$}.
\newblock {\em Math. Ann.}, 104(1):71--80, 1931.

\bibitem{NY12}
P.~W. Nowak and G.~Yu.
\newblock {\em Large scale geometry}.
\newblock EMS Textbooks in Mathematics. European Mathematical Society (EMS),
  Z\"urich, 2012.

\bibitem{OR16}
M.~Ostrovskii and B.~Randrianantoanina.
\newblock Metric spaces admitting low-distortion embeddings into all
  {$n$}-dimensional {B}anach spaces.
\newblock {\em Canad. J. Math.}, 68(4):876--907, 2016.

\bibitem{Ost11}
M.~I. Ostrovskii.
\newblock On metric characterizations of some classes of {B}anach spaces.
\newblock {\em C. R. Acad. Bulgare Sci.}, 64(6):775--784, 2011.

\bibitem{Ost13}
M.~I. Ostrovskii.
\newblock {\em Metric embeddings}, volume~49 of {\em De Gruyter Studies in
  Mathematics}.
\newblock De Gruyter, Berlin, 2013.
\newblock Bilipschitz and coarse embeddings into Banach spaces.

\bibitem{Pan89}
P.~Pansu.
\newblock M\'etriques de {C}arnot-{C}arath\'eodory et quasiisom\'etries des
  espaces sym\'etriques de rang un.
\newblock {\em Ann. of Math. (2)}, 129(1):1--60, 1989.

\bibitem{Pau01}
S.~D. Pauls.
\newblock The large scale geometry of nilpotent {L}ie groups.
\newblock {\em Comm. Anal. Geom.}, 9(5):951--982, 2001.

\bibitem{Pel77}
A.~Pe\l~czy\'nski.
\newblock {\em Banach spaces of analytic functions and absolutely summing
  operators}.
\newblock American Mathematical Society, Providence, R.I., 1977.
\newblock Expository lectures from the CBMS Regional Conference held at Kent
  State University, Kent, Ohio, July 11--16, 1976, Conference Board of the
  Mathematical Sciences Regional Conference Series in Mathematics, No. 30.

\bibitem{Pisier-martingales}
G.~Pisier.
\newblock Martingales with values in uniformly convex spaces.
\newblock {\em Israel J. Math.}, 20(3-4):326--350, 1975.

\bibitem{Pis80}
G.~Pisier.
\newblock Un th\'eor\`eme sur les op\'erateurs lin\'eaires entre espaces de
  {B}anach qui se factorisent par un espace de {H}ilbert.
\newblock {\em Ann. Sci. \'Ecole Norm. Sup. (4)}, 13(1):23--43, 1980.

\bibitem{Pis81}
G.~Pisier.
\newblock Remarques sur un r\'esultat non publi\'e de {B}. {M}aurey.
\newblock In {\em Seminar on {F}unctional {A}nalysis, 1980--1981}, pages Exp.
  No. V, 13. \'Ecole Polytech., Palaiseau, 1981.

\bibitem{Pis83}
G.~Pisier.
\newblock Counterexamples to a conjecture of {G}rothendieck.
\newblock {\em Acta Math.}, 151(3-4):181--208, 1983.

\bibitem{Pisier-type}
G.~Pisier.
\newblock Probabilistic methods in the geometry of {B}anach spaces.
\newblock In {\em Probability and analysis (Varenna, 1985)}, volume 1206 of
  {\em Lecture Notes in Math.}, pages 167--241. Springer, Berlin, 1986.

\bibitem{Pis92}
G.~Pisier.
\newblock Factorization of operator valued analytic functions.
\newblock {\em Adv. Math.}, 93(1):61--125, 1992.

\bibitem{Pis-92-clapem}
G.~Pisier.
\newblock Random series of trace class operators.
\newblock In {\em Proceedings {C}uarto {C}{L}{A}{P}{E}{M} {M}exico 1990.
  {C}ontribuciones en probabilidad y estadistica matematica}, pages 29--42,
  1992.
\newblock Available at \url{http://arxiv.org/abs/1103.2090}.

\bibitem{Pis12}
G.~Pisier.
\newblock Grothendieck's theorem, past and present.
\newblock {\em Bull. Amer. Math. Soc. (N.S.)}, 49(2):237--323, 2012.

\bibitem{PDG17}
G.~Puy, M.~E. Davies, and R.~Gribonval.
\newblock Recipes for stable linear embeddings from {H}ilbert spaces to
  {$\mathbb{R}^m$}.
\newblock {\em IEEE Trans. Inform. Theory}, 63(4):2171--1287, 2017.

\bibitem{Rab08}
Y.~Rabinovich.
\newblock On average distortion of embedding metrics into the line.
\newblock {\em Discrete Comput. Geom.}, 39(4):720--733, 2008.

\bibitem{Reg13}
O.~Regev.
\newblock Entropy-based bounds on dimension reduction in {$L^1$}.
\newblock {\em Israel J. Math.}, 195(2):825--832, 2013.

\bibitem{Ribe76}
M.~Ribe.
\newblock On uniformly homeomorphic normed spaces.
\newblock {\em Ark. Mat.}, 14(2):237--244, 1976.

\bibitem{Ric15}
E.~Ricard.
\newblock H\"older estimates for the noncommutative {M}azur maps.
\newblock {\em Arch. Math. (Basel)}, 104(1):37--45, 2015.

\bibitem{Roe03}
J.~Roe.
\newblock {\em Lectures on coarse geometry}, volume~31 of {\em University
  Lecture Series}.
\newblock American Mathematical Society, Providence, RI, 2003.

\bibitem{Rya02}
R.~A. Ryan.
\newblock {\em Introduction to tensor products of {B}anach spaces}.
\newblock Springer Monographs in Mathematics. Springer-Verlag London, Ltd.,
  London, 2002.

\bibitem{Sch87}
G.~Schechtman.
\newblock More on embedding subspaces of {$L_p$} in {$l^n_r$}.
\newblock {\em Compositio Math.}, 61(2):159--169, 1987.

\bibitem{Sch17}
G.~Schechtman.
\newblock Asymptotic geometric analysis, {P}art {I} [book review of
  {MR}3331351].
\newblock {\em Bull. Amer. Math. Soc. (N.S.)}, 54(2):341--345, 2017.

\bibitem{ST18}
G.~Schechtman and N.~Tomczak-Jaegermann.
\newblock Polylog dimensional subspaces of $\ell_\infty^{N}$.
\newblock Preprint, 2018.

\bibitem{Sch14}
R.~Schneider.
\newblock {\em Convex bodies: the {B}runn-{M}inkowski theory}, volume 151 of
  {\em Encyclopedia of Mathematics and its Applications}.
\newblock Cambridge University Press, Cambridge, expanded edition, 2014.

\bibitem{Sch38}
I.~J. Schoenberg.
\newblock Metric spaces and positive definite functions.
\newblock {\em Trans. Amer. Math. Soc.}, 44(3):522--536, 1938.

\bibitem{Sem96}
S.~Semmes.
\newblock On the nonexistence of bi-{L}ipschitz parameterizations and geometric
  problems about {$A_\infty$}-weights.
\newblock {\em Rev. Mat. Iberoamericana}, 12(2):337--410, 1996.

\bibitem{Sem99}
S.~Semmes.
\newblock Bilipschitz embeddings of metric spaces into {E}uclidean spaces.
\newblock {\em Publ. Mat.}, 43(2):571--653, 1999.

\bibitem{SJ89}
A.~Sinclair and M.~Jerrum.
\newblock Approximate counting, uniform generation and rapidly mixing {M}arkov
  chains.
\newblock {\em Inform. and Comput.}, 82(1):93--133, 1989.

\bibitem{Sta82}
A.~J. Stam.
\newblock Limit theorems for uniform distributions on spheres in
  high-dimensional {E}uclidean spaces.
\newblock {\em J. Appl. Probab.}, 19(1):221--228, 1982.

\bibitem{Ste80}
G.~W. Stewart.
\newblock The efficient generation of random orthogonal matrices with an
  application to condition estimators.
\newblock {\em SIAM J. Numer. Anal.}, 17(3):403--409 (loose microfiche suppl.),
  1980.

\bibitem{Stu03}
K.-T. Sturm.
\newblock Probability measures on metric spaces of nonpositive curvature.
\newblock In {\em Heat kernels and analysis on manifolds, graphs, and metric
  spaces ({P}aris, 2002)}, volume 338 of {\em Contemp. Math.}, pages 357--390.
  Amer. Math. Soc., Providence, RI, 2003.

\bibitem{Tal90}
M.~Talagrand.
\newblock Embedding subspaces of {$L_1$} into {$l^N_1$}.
\newblock {\em Proc. Amer. Math. Soc.}, 108(2):363--369, 1990.

\bibitem{Tal92-helix}
M.~Talagrand.
\newblock Approximating a helix in finitely many dimensions.
\newblock {\em Ann. Inst. H. Poincar\'e Probab. Statist.}, 28(3):355--363,
  1992.

\bibitem{Tal95}
M.~Talagrand.
\newblock Embedding subspaces of {$L_p$} in {$l^N_p$}.
\newblock In {\em Geometric aspects of functional analysis ({I}srael,
  1992--1994)}, volume~77 of {\em Oper. Theory Adv. Appl.}, pages 311--325.
  Birkh\"auser, Basel, 1995.

\bibitem{Tal14}
M.~Talagrand.
\newblock {\em Upper and lower bounds for stochastic processes}, volume~60 of
  {\em Ergebnisse der Mathematik und ihrer Grenzgebiete. 3. Folge. A Series of
  Modern Surveys in Mathematics [Results in Mathematics and Related Areas. 3rd
  Series. A Series of Modern Surveys in Mathematics]}.
\newblock Springer, Heidelberg, 2014.
\newblock Modern methods and classical problems.

\bibitem{Tan84}
R.~M. Tanner.
\newblock Explicit concentrators from generalized {$N$}-gons.
\newblock {\em SIAM J. Algebraic Discrete Methods}, 5(3):287--293, 1984.

\bibitem{Tho65}
R.~Thom.
\newblock Sur l'homologie des vari\'et\'es alg\'ebriques r\'eelles.
\newblock In {\em Differential and {C}ombinatorial {T}opology ({A} {S}ymposium
  in {H}onor of {M}arston {M}orse)}, pages 255--265. Princeton Univ. Press,
  Princeton, N.J., 1965.

\bibitem{TZ05}
M.~Thorup and U.~Zwick.
\newblock Approximate distance oracles.
\newblock {\em J. ACM}, 52(1):1--24, 2005.

\bibitem{Tom74}
N.~Tomczak-Jaegermann.
\newblock The moduli of smoothness and convexity and the {R}ademacher averages
  of trace classes {$S_{p}(1\leq p<\infty )$}.
\newblock {\em Studia Math.}, 50:163--182, 1974.

\bibitem{Vai99}
J.~V\"ais\"al\"a.
\newblock The free quasiworld. {F}reely quasiconformal and related maps in
  {B}anach spaces.
\newblock In {\em Quasiconformal geometry and dynamics ({L}ublin, 1996)},
  volume~48 of {\em Banach Center Publ.}, pages 55--118. Polish Acad. Sci.
  Inst. Math., Warsaw, 1999.

\bibitem{Vem04}
S.~S. Vempala.
\newblock {\em The random projection method}, volume~65 of {\em DIMACS Series
  in Discrete Mathematics and Theoretical Computer Science}.
\newblock American Mathematical Society, Providence, RI, 2004.
\newblock With a foreword by Christos H. Papadimitriou.

\bibitem{Vil09}
C.~Villani.
\newblock {\em Optimal transport}, volume 338 of {\em Grundlehren der
  Mathematischen Wissenschaften [Fundamental Principles of Mathematical
  Sciences]}.
\newblock Springer-Verlag, Berlin, 2009.
\newblock Old and new.

\bibitem{vN37}
J.~von Neumann.
\newblock Some matrix-inequalities and metrization of matric-space.
\newblock {\em Tomsk Univ. Rev.}, 1:286--300, 1937.
\newblock Reprinted in Collected Works (Pergamon Press, 1962), iv, 205--219.

\bibitem{Whi36}
H.~Whitney.
\newblock Differentiable manifolds.
\newblock {\em Ann. of Math. (2)}, 37(3):645--680, 1936.

\end{thebibliography}

 \end{document}